\documentclass[10pt,oneside, reqno]{amsproc}
\usepackage[utf8]{inputenc}
\usepackage[english]{babel}

\usepackage{setspace}
\usepackage{hyperref}
\usepackage{amssymb}
\usepackage{amsmath}
\usepackage{amsthm}
\usepackage{stmaryrd}
\usepackage{graphicx}
\usepackage{mathrsfs}
\usepackage[noabbrev]{cleveref}

\usepackage{mathtools}

\usepackage{xcolor}
\hypersetup{
    colorlinks=true,
    linkcolor=blue,
    filecolor=magenta,      
    urlcolor=cyan,
    citecolor=blue,
}
\usepackage{enumerate}

\input xy
\xyoption{all}
\usepackage[color,matrix,arrow]{xy}
\usepackage{tikz}
\usepackage{tikz-cd}
\usetikzlibrary{bending}
\usetikzlibrary{arrows.meta}
\usetikzlibrary{matrix,arrows,decorations.markings, calc, positioning}
\usetikzlibrary {shapes.geometric}
\usetikzlibrary {intersections}

\usepackage{geometry}
\makeatletter
\providecommand{\leftsquigarrow}{%
  \mathrel{\mathpalette\reflect@squig\relax}%
}
\newcommand{\reflect@squig}[2]{%
  \reflectbox{$\m@th#1\rightsquigarrow$}%
}
\makeatother

\let\L=\relax

\let\O=\relax

\let\epsilon=\relax

\let\th=\relax

\newcommand{\w}{w}

\DeclareMathOperator{\Dic}{Dic}

\DeclareMathOperator{\Sub}{Sub}

\newcommand{\FF}{\mathbb{F}}

\newcommand{\smashy}{\wedge}

\newcommand{\th}{\text{th}}

\newcommand{\epsilon}{\varepsilon}

\newcommand{\L}{\mathscr{L}}

\newcommand{\O}{\mathscr{O}}

\theoremstyle{plain}
\newtheorem{theorem}{Theorem}[section]
\newtheorem{lemma}[theorem]{Lemma}
\newtheorem{proposition}[theorem]{Proposition}
\newtheorem{corollary}[theorem]{Corollary}

\theoremstyle{definition}
\newtheorem{definition}[theorem]{Definition}
\newtheorem{conjecture}[theorem]{Conjecture}
\newtheorem{remark}[theorem]{Remark}
\newtheorem{example}[theorem]{Example}

\newcommand{\Hull}{\text{Hull}}

\renewcommand{\theequation}{\thesection.\arabic{equation}}

\usepackage{graphicx} 
\usepackage{environ} 
\usepackage{float}
\usepackage{comment}
\usepackage{hyperref}
\usepackage{tikz}
\newcommand{\tikzxmark}{%
\tikz[scale=0.23] {
    \draw[line width=0.7,line cap=round] (0,0) to [bend left=6] (1,1);
    \draw[line width=0.7,line cap=round] (0.2,0.95) to [bend right=3] (0.8,0.05);
}}

\DeclareSymbolFont{extraup}{U}{zavm}{m}{n}
\DeclareMathSymbol{\varheart}{\mathalpha}{extraup}{86}
\DeclareMathSymbol{\vardiamond}{\mathalpha}{extraup}{87}

\NewEnviron{NORMAL}{%
    \scalebox{2}{$\BODY$} 
} 

\title{Characterizing Transfer Systems for Non-Abelian Groups}

\author{Sarah Klanderman}
\address{Department of Mathematical and Computational Sciences, Marian University, Indianapolis}
\email{\href{mailto:sklanderman@marian.edu}{sklanderman@marian.edu}}

\author{Chloe Lewis}
\address{Department of Mathematics, University of Wisconsin-Eau Claire, Eau Claire, Wisconsin}
\email{\href{mailto:lewischl@uwec.edu}{lewischl@uwec.edu}}

\author{Harlea Monson}
\address{Department of Mathematics, Wake Forest University, Winston-Salem, North Carolina}
\email{\href{mailto:monshm25@wfu.edu}{monshm25@wfu.edu}}

\author{Koki Shibata}
\address{Department of Mathematics, University of Wisconsin-Eau Claire, Eau Claire, Wisconsin}
\email{\href{mailto:shibatak7973@uwec.edu}{shibatak7973@uwec.edu}}

 \author{Danika Van Niel}
 \address{Department of Mathematics and Statistics, Binghamton University, Binghamton, New York} \email{\href{mailto:dvanniel@binghamton.edu}{dvanniel@binghamton.edu}}

\begin{document}

\vspace*{-.5cm}

\subjclass[2020]{Primary 06B99; Secondary 55P91}
\date{November 15, 2025.}

\keywords{Transfer systems, homotopical combinatorics}

\begin{abstract}

For a finite group $G$, the notion of a $G$-transfer system provides homotopy theorists with a combinatorial way to study equivariant objects. In this paper, we focus on the properties of transfer systems for non-abelian groups. We explicitly describe the width of all dihedral groups, quaternion groups, and dicyclic groups. For a given $G$, the set of all $G$-transfer systems forms a poset lattice under inclusion; these are a useful resource to homotopical combinatorialists for detecting patterns and checking conjectures. We expand the suite of known transfer system lattices for non-abelian groups including those which are dihedral, dicyclic, Frobenius, and alternating.  
\end{abstract}

\maketitle

\section{Introduction}

In recent years there has been an explosion of research in homotopical combinatorics. Homotopical combinatorics provides a way to study objects which arise in equivariant homotopy theory contexts using combinatorial tools like directed graphs. This accessible inroad to equivariant homotopy theory has proven useful in the study of $N_\infty$-operads, ring $G$-spectra, bi-incomplete Tambara functors, and model structures among other objects \cite{BHOperads,GWViaOperads,NinftyOperads,RubinCombNinftyOperads,RubinDetectingOperads,NinftyOperads, BPGenuineEquivaiantOperads,BiIncomplete,ModelStructuresOnFiniteTotalOrders,ChanTambara,CharModStru}.

For a finite group $G$, the main combinatorial object of study is that of a $G$-transfer system, a partial order relation on the subgroup lattice of $G$ which satisfies certain axioms. The collection of $G$-transfer systems form their own partial order relation under inclusion; this poset is often represented as a Hasse diagram. Work of Rubin \cite{RubinDetectingOperads,RubinCombNinftyOperads} provides several examples of such transfer system lattices, including cyclic groups, a quaternion group, and a permutation group. Additional Hasse diagrams are found in \cite{eCHTREU}. These figures have proven incredibly useful for homotopical combinatorialists who are formulating and testing conjectures about $G$-transfer systems. 

Much of the existing literature on $G$-transfer systems considers cases where $G$ is an abelian group. One notices a similar focus on abelian groups in the literature on equivariant homotopy theory as well, largely due to the fact that non-commutative equivariant algebra is more complex. In the setting of transfer systems, additional complications arise from having to consider all of the conjugacy data in a non-abelian group. In fact, one of the five axioms which define a transfer system does not apply at all when dealing with abelian groups.

For a small enough subgroup lattice $\Sub(G)$, it is a simple exercise to draw every $G$-transfer system and place them in a Hasse diagram organized via inclusion. However, the total number of transfer systems quickly becomes unwieldily. For example, the total number of $C_{p^n}$-transfer systems is the $(n + 1)^\th$ Catalan number Cat$(n+1) = \frac{(2n + 2)!}{(n+2)!(n+1)!}$ \cite{NinftyOperads}. The first few numbers in this sequence from $n = 0$ to $6$ are $1,2,5,14,42,132,429,1430$. It is thus not reasonable to construct by hand all possible transfer systems, even for relatively simple groups. Scott Balchin's software package \cite{balchincode} is an incredibly useful tool for generating transfer system information for larger groups. 

 In this paper we include some well known Hasse diagrams, for $G = C_{p^2}, C_{pq},$ and $D_p$. We also construct some Hasse diagrams absent from the existing literature, for $G = A_4, D_{p^2}, \Dic_p$, and $F_5$. Further, we label each transfer system based on important properties they possess, namely if they are saturated, cosaturated, or lesser simply paired, as defined in \cite{LSP}.

One can also examine $G$-transfer systems to produce group invariants. For instance, the authors in \cite{rainbowMRC} introduce the notion of width and complexity for a group, as measured by counting arrows in a $G$-transfer system. The width of some abelian groups, including $C_{p^n}, C_{p^nq},$ and $C_{p_1 p_2 \ldots p_n}$, is computed in \cite{rainbowMRC}. In this paper, we compute the width of all dihedral groups, all quaternion groups, all dicyclic groups, and some Frobenius groups.

\subsection{Organization}

We begin with some necessary background in \Cref{sec:Background}, including, the definition of transfer systems, conjugacy axiom examples, and other properties of transfer systems we will use in this paper. In \Cref{sec: width} we compute the width of all dihedral groups, all quaternion groups, all dicyclic groups, and some Frobenius groups. Lastly, in \Cref{sec: lattices} we discuss future work and provide Hasse diagrams of various $G$-transfer systems, namely $C_{p^2}, C_{pq}, D_p, A_4, D_{p^2}, \Dic_p$, and $F_5$.

\subsection{Acknowledgments}

We would like to thank \href{https://people.maths.bris.ac.uk/~matyd/GroupNames/}{\textit{Groupnames}} as a convenient resource for studying these groups, as well as Jonathan Rubin for establishing LaTeX notation for lattices. Matthew Brin contributed helpful conversations, and Scott Balchin, Ben Spitz, and Kurt Stoeckl shared incredibly useful code. We thank Katharine Adamyk for helpful conversations and useful feedback on an early draft of this paper. We are grateful to the anonymous referee for their valuable comments and helpful suggestions.  Work on this project was inspired by Lewis and Van Niel's participation in the AMS Mathematics Research Communities Program \textit{Homotopical Combinatorics} and NSF grant DMS--1916439. Support for this project came from Student Blugold Commitment Differential Tuition funds through the University of Wisconsin-Eau Claire.

\section{Background}\label{sec:Background}

Throughout this paper, we will assume that $G$ is a finite group. To begin, we introduce $G$-transfer systems, and discuss some of the unique properties of $G$-transfer systems for a non-abelian $G$. We then provide some examples of said properties which arise from conjugacy classes. Lastly, we recall some characterizations of transfer systems, which we will study in the following sections.

\subsection{Transfer Systems}

We first recall the definition of a $G$-transfer system. 

\begin{definition}\label{def:tranfsys} Let $G$ be a finite group. A $G$-\emph{transfer system} $T$ is a partial order relation on the subgroup lattice of $G$, $\Sub(G)$, represented by edges
$\to$, that satisfies the following conditions:
\begin{enumerate}
    \item (Subgroup) $K \to H$ implies $K \leq  H$,
    \item(Reflexivity) $H\to H$ for all $H\leq G$,
    \item (Composition) $L\to K$ and $K\to H$   implies $L\to H$,
    \item (Restriction) $K\to H$ and $L\leq H$ implies $(K\cap L)\to (H\cap L)$, and
    \item (Conjugation) $K\to H$ implies $gKg^{-1}\to gHg^{-1}$ for all $g\in G$.
\end{enumerate}
\end{definition}

All images of $G$-transfer systems that follow in this paper will not include any reflexive edges for the sake of clarity. For any finite group $G$ we have the following definitions. 

\begin{definition}
    A $G$-transfer system is called \emph{trivial} if it contains only the edges $H \to H$ for all $H \leq G$. Conversely, a \emph{complete} transfer system contains all possible edges in the subgroup lattice. 
\end{definition}

When $G$ is abelian, we do not need to consider the conjugation condition. At first glance, this seems like a minor simplification. It is our aim in this paper to illustrate the subtle complexities that arise in studying $G$-transfer systems for a non-abelian group with conjugacy classes. For instance, abelian subgroup lattices all share the property of \textit{modularity}. 

\begin{definition}
    A \emph{modular lattice} is a lattice that does not contain the pentagon $N_5$, depicted in \Cref{fig: nonmod lattice}, as a sublattice.
\begin{figure}[h]
\begin{tikzcd}[column sep=small, row sep=small]
	& \bullet \\
	&& \bullet \\
	\bullet && \bullet \\
	& \bullet
	\arrow[from=2-3, to=1-2]
	\arrow[from=3-1, to=1-2]
	\arrow[from=3-3, to=2-3]
	\arrow[from=4-2, to=3-1]
	\arrow[from=4-2, to=3-3]
\end{tikzcd}
\caption{The pentagon $N_5$.} \label{fig: nonmod lattice}
\end{figure}
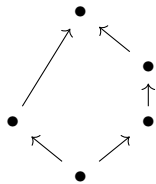
\end{definition}

The only subgroup lattices that are non-modular are associated to non-abelian groups. An interesting property arises when one studies short edges in $G$-transfer systems on a non-modular lattice. 

\begin{definition}\label{def: covering relation}
    A \emph{short edge} (or covering relation) on a subgroup lattice is an edge $K \to H$ such that $K < H$ and there exists no subgroup $L$ such that $K < L < H$.
\end{definition}

 The pentagon $N_5$ in \Cref{fig: nonmod lattice} is in fact the subgroup lattice (mod conjugacy) for the alternating group $A_4$. Since some groups may have an exceedingly large number of conjugate subgroups, we will often depict $\Sub(G)/G$, which is the poset of conjugacy classes of subgroups of $G$, instead of $\Sub(G)$ for simplicity. As the authors of \cite{Lossless} put it, this is a more ``human friendly" depiction.

On a modular lattice, the restriction of any short edge is again a short edge; on a non-modular lattice this is not always the case. Consider the $A_4$-transfer system in \Cref{fig: a4 lattice}, where the solid edge restricts to both of the dotted edges. We will depict $n$-many conjugate copies of a subgroup $H < G$ as ${}_nH$ in $\Sub(G)/G$.
\begin{figure}[h]
\begin{tikzcd}[column sep=small, row sep=small]
	& A_4 \\
	&& C_2^2 \\
	{}_4C_3 && {}_3C_2 \\
	& e
	\arrow[from=3-1, to=1-2]
	\arrow[dashed, from=4-2, to=2-3]
	\arrow[dashed, from=4-2, to=3-3]
\end{tikzcd}
\caption{Restrictions on the non-modular lattice $\Sub(A_4)/A_4$.}
\label{fig: a4 lattice}
\end{figure}
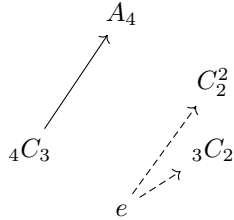

\noindent One of the dotted edges is a short edge and the other is not. This anomaly is a result of the non-modular structure.

The reader may wonder if this transfer system on $\Sub(A_4)/A_4$ lifts to a transfer system on $\Sub(A_4)$. The following definition and proposition determine those groups $G$ for which transfer systems can always be lifted from $\Sub(G)/G$ to $\Sub(G)$.

\begin{definition}\cite[Definition 2.6]{Lossless}
    A \emph{lossless group} is a group $G$ where for all pairs of subgroups $K \leq H$ such that $gKg^{-1} \leq H$ for some $g \in G$, there exists some $h \in N_G(H)$ such that $hKh^{-1} = gKg^{-1}$. A \emph{lossy group} is a group that is not lossless.
\end{definition}

For example, $A_4$ and all abelian groups are lossless. For the following, let $[K]$ denote the conjugacy class of $K$.

\begin{proposition}\cite[Lemma 3.1]{Lossless}
    Let $G$ be a lossless group, and $T$ a transfer system on $\Sub(G)/G$. The transfer system $T$ lifts to a $G$-transfer system if and only if for all edges $[K] \to [H]$ and any $K' \leq H$ with $[K'] = [K]$ we have the edge $[K \cap K'] \to [H]$ in $T$.
\end{proposition}

These authors provide several examples of lossless and lossy groups, as well as more properties of those transfer systems in \cite{Lossless}. In this paper, we will make use of the fact that all dihedral, quaternion, and dicyclic groups are lossless.

\subsection{Conjugacy Axiom Examples} \label{subsec: conjugacy axiom}

In a non-abelian group, there could be several conjugate copies of a particular subgroup. For example, the dihedral group of order $18$, $D_9$, has 9 conjugate copies of $C_2$. Instead of writing out a vertex for each copy of $C_2$, it is easier to consider $\Sub(D_9)/D_9$. Even though all dihedral groups are lossless, one may rightfully still be concerned about the data that is lost when collapsing the conjugacy classes. It is true that having several conjugate copies of a subgroup can lead to some interesting restrictions, even within a lossless group, that would not arise in abelian cases. Consider the following example.

\begin{example}\label{Ex: Cpq lookalike}
Let $p$ and $q$ be distinct primes, for $p$ odd. Consider the groups $C_{pq}$ and $D_p$, and note that the lattice structures of $\Sub(C_{pq})$ and $\Sub(D_p)/D_p$ are the same, as shown in \Cref{fig: Cpq vs Dp picture}. One can restrict an edge $C_2 \to D_p$ by any other copy of $C_2$ to get the edge $e \to C_2$. Therefore the first and last of these pictures are transfer systems for their respective groups but the middle picture is not closed under restriction.
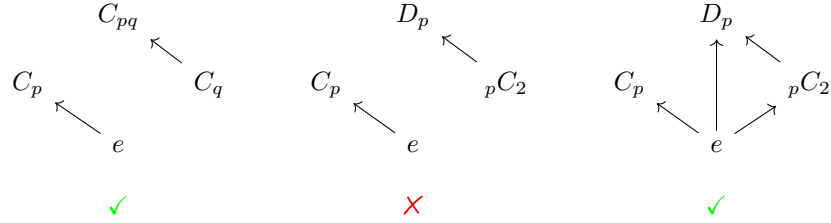
\begin{figure}[h]
\begin{tikzcd}[column sep=small, row sep=small]
	& {C_{pq}} &&&& {D_p} &&&& {D_p} \\
	{C_p} && {C_q} && {C_p} && {{}_pC_2} && {C_p} && {{}_pC_2} \\
	& e &&&& e &&&& e \\
	& \textcolor{green}{\checkmark} &&&& \textcolor{red}{\tikzxmark} &&&& \textcolor{green}{\checkmark}
	\arrow[from=2-3, to=1-2]
	\arrow[from=2-7, to=1-6]
	\arrow[from=2-11, to=1-10]
	\arrow[from=3-2, to=2-1]
	\arrow[from=3-6, to=2-5]
	\arrow[from=3-10, to=1-10]
	\arrow[from=3-10, to=2-9]
	\arrow[from=3-10, to=2-11]
\end{tikzcd}
\caption{An example of a $C_{pq}$-transfer system, a non-example of a $D_{p}$-transfer system, and a $D_p$-transfer system.}
\label{fig: Cpq vs Dp picture}
\end{figure}
\end{example}

For brevity, we will refer to the restrictions that arise only from analyzing conjugacy classes as \emph{due to conjugacy (DTC restrictions)}. These classes of restrictions are the focus of our analysis in many of the non-abelian transfer system examples that follow.

Not every edge out of a conjugacy class will give DTC restrictions, consider the following example.

\begin{example} Consider the Frobenius group of order 20, $F_5$. Since each conjugate copy of $C_4$ only contains one copy of $C_2$ and there are equal number of copies of $C_2$ as there are copies of $C_4$ then no edge $C_2 \to C_4$ will restrict to an edge $e \to C_2$. This transfer system is depicted in \Cref{fig:F_5DTCrestriction}

    \begin{figure}[h]
\[\begin{tikzcd}[column sep=small, row sep=small]
	& {F_5} \\
	{{}_{5}C_4} && {D_5} \\
	{{}_{5}C_2} && {C_5} \\
	& e
	\arrow[from=3-1, to=2-1]
\end{tikzcd}\]
    \caption{A $F_5$-transfer system with no DTC restrictions.}
    \label{fig:F_5DTCrestriction}
\end{figure}
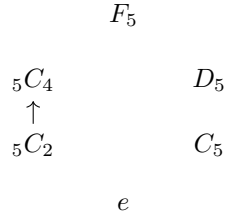
\end{example}

Let us now consider DTC restrictions on $D_{p^n}$.

\begin{example} \label{ex: dihedral odds}
    Let $p$ be an odd prime and consider the complete $D_{p^n}$-transfer systems, with composition edges omitted. Note that the lattices $\Sub(D_{p^n})/D_{p^n}$ build inductively, as shown in \Cref{fig: Dp^n lattices}.
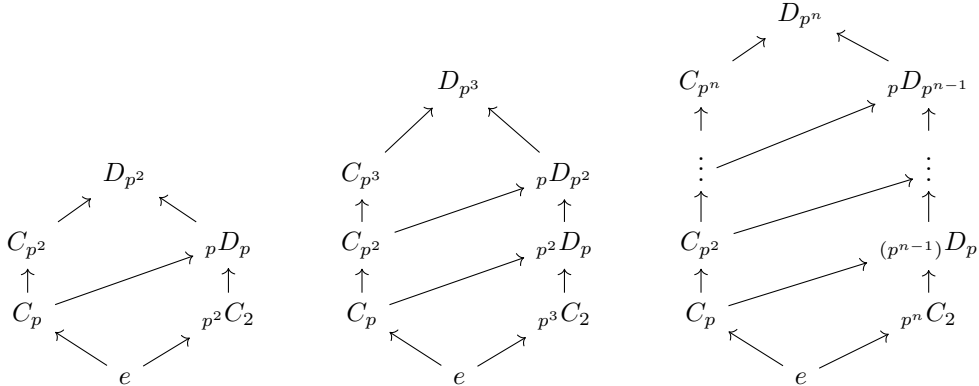
\begin{figure}[h]
\begin{tikzcd}[column sep=small, row sep=small]
	&&&&&&&&& {D_{p^n}} \\
	&&&&& {D_{p^3}} &&& {C_{p^n}} && {{}_{p}D_{p^{n-1}}} \\
	& {D_{p^2}} &&& {C_{p^3}} && {{}_{p}D_{p^2}} && \vdots && \vdots \\
	{C_{p^2}} && {{}_pD_{p}} && {C_{p^2}} && {{}_{p^2}D_{p}} && {C_{p^2}} && {{}_{(p^{n-1})}D_{p}} \\
	{C_p} && {{}_{p^2}C_2} && {C_p} && {{}_{p^3}C_2} && {C_p} && {{}_{p^n}C_2} \\
	& e &&&& e &&&& e
	\arrow[from=2-9, to=1-10]
	\arrow[from=2-11, to=1-10]
	\arrow[from=3-5, to=2-6]
	\arrow[from=3-7, to=2-6]
	\arrow[from=3-9, to=2-9]
	\arrow[from=3-9, to=2-11]
	\arrow[from=3-11, to=2-11]
	\arrow[from=4-1, to=3-2]
	\arrow[from=4-3, to=3-2]
	\arrow[from=4-5, to=3-5]
	\arrow[from=4-5, to=3-7]
	\arrow[from=4-7, to=3-7]
	\arrow[from=4-9, to=3-9]
	\arrow[from=4-9, to=3-11]
	\arrow[from=4-11, to=3-11]
	\arrow[from=5-1, to=4-1]
	\arrow[from=5-1, to=4-3]
	\arrow[from=5-3, to=4-3]
	\arrow[from=5-5, to=4-5]
	\arrow[from=5-5, to=4-7]
	\arrow[from=5-7, to=4-7]
	\arrow[from=5-9, to=4-9]
	\arrow[from=5-9, to=4-11]
	\arrow[from=5-11, to=4-11]
	\arrow[from=6-2, to=5-1]
	\arrow[from=6-2, to=5-3]
	\arrow[from=6-6, to=5-5]
	\arrow[from=6-6, to=5-7]
	\arrow[from=6-10, to=5-9]
	\arrow[from=6-10, to=5-11]
\end{tikzcd}
\caption{The lattices $\Sub(D_{p^n})/D_{p^n}$ for $n=2,3$ and $n$ arbitrary.}
\label{fig: Dp^n lattices}
\end{figure}
As in \Cref{Ex: Cpq lookalike} above, a transfer system on one of these lattices may have DTC restrictions; we explore these now. 

\begin{figure}[h]
\begin{tikzcd}[column sep=small, row sep=small]
	& {D_{p^3}} &&&&& {D_{p^3}} \\
	{C_{p^3}} && {{}_pD_{p^2}} &&& {C_{p^3}} && {{}_pD_{p^2}} \\
	{C_{p^2}} && {{}_{p^2}D_{p}} &&& {C_{p^2}} && {{}_{p^2}D_{p}} \\
	{C_p} && {{}_{p^3}C_2} &&& {C_p} && {{}_{p^3}C_2} \\
	& e &&&&& e
	\arrow[from=3-8, to=2-8]
	\arrow[bend right=50, from=4-3, to=2-3]
	\arrow[dashed, thick, from=4-6, to=3-8]
	\arrow[dashed, thick, from=5-2, to=4-3]
\end{tikzcd}
\caption{Two DTC restrictions in $D_{p^3}$-transfer systems.}
\label{fig: conj resetrictions for Dp^3}
\end{figure}
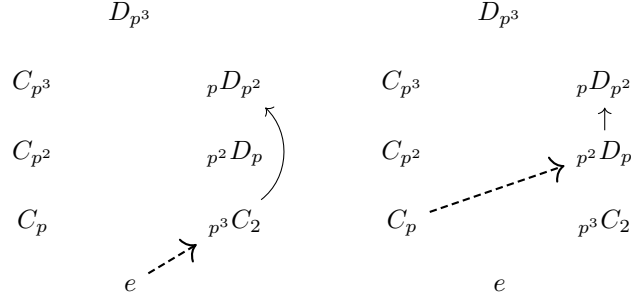

Restricting the edge ${}_{p^n}C_2 \to {}_{(p^{n-1})}D_p$ by a conjugate copy of $C_2$ produces the edge $e \to {}_{p^n}C_2$, for any $n$. An example of this restriction in $D_{p^3}$ is shown on the left in \Cref{fig: conj resetrictions for Dp^3}.
Restricting the edge ${}_{(p^{n-i})}D_{p^i} \to {}_{(p^{n-i - 1})}D_{p^{i+1}}$ by a conjugate copy of ${}_{(p^{n-i})}D_{p^i}$ which is contained in the same conjugate copy of ${}_{(p^{n-i - 1})}D_{p^{i+1}}$ yields an edge $C_{p^i} \to {}_{(p^{n-i})}D_{p^i}$ for $i \geq 1$ and any $n$.  An example of this restriction in $D_{p^3}$ is shown on the right in \Cref{fig: conj resetrictions for Dp^3}.
\end{example}

We now explore another class of dihedral groups, those of the form $D_{2^n}$ for $n \geq 2$.

\begin{example}\label{ex: dihedral evens}
    Consider the complete $D_{2^n}$-transfer systems below, with composition edges omitted. Note that the lattices $\Sub(D_{2^n})/D_{2^n}$ build inductively, as shown in \Cref{fig: D2^n lattices} and resemble the skeleton of a bug; we invoke this metaphor to refer to some components of the lattices. Let us call the middle string of vertical edges the spine and the string of vertical edges on either side of the spine the wings. 
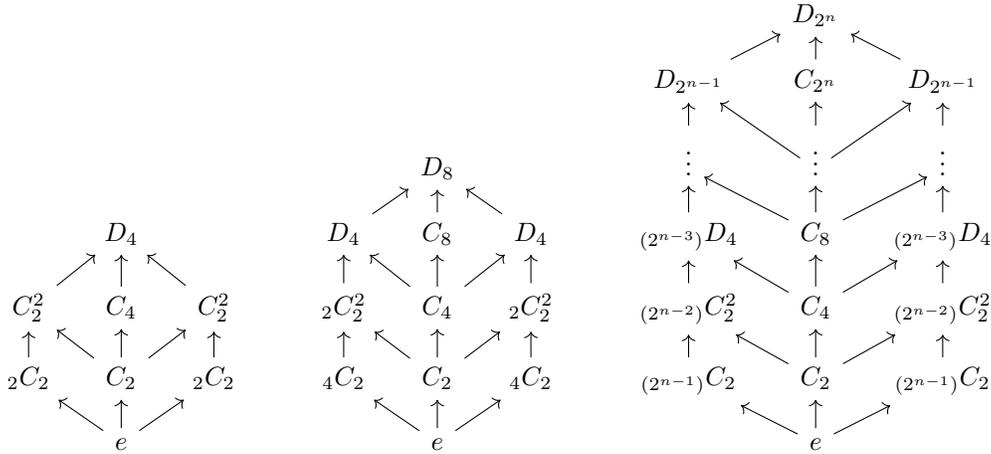
\begin{figure}[h]
\begin{tikzcd}[column sep=small, row sep=small]
	&&&&&&&&& {D_{2^n}} \\
	&&&&&&&& {D_{2^{n-1}}} & {C_{2^n}} & {D_{2^{n-1}}} \\
	&&&&& {D_8} &&& \vdots & \vdots & \vdots \\
	& {D_4} &&& {D_4} & {C_8} & {D_4} && {{}_{(2^{n-3})}D_4} & {C_8} & {{}_{(2^{n-3})}D_4} \\
	{D_2} & {C_4} & {D_2} && {{}_2D_2} & {C_4} & {{}_2D_2} && {{}_{(2^{n-2})}D_2} & {C_4} & {{}_{(2^{n-2})}D_2} \\
	{{}_2C_2} & {C_2} & {{}_2C_2} && {{}_4C_2} & {C_2} & {{}_4C_2} && {{}_{(2^{n-1})}C_2} & {C_2} & {{}_{(2^{n-1})}C_2} \\
	& e &&&& e &&&& e
	\arrow[from=2-9, to=1-10]
	\arrow[from=2-10, to=1-10]
	\arrow[from=2-11, to=1-10]
	\arrow[from=3-9, to=2-9]
	\arrow[from=3-10, to=2-9]
	\arrow[from=3-10, to=2-10]
	\arrow[from=3-10, to=2-11]
	\arrow[from=3-11, to=2-11]
	\arrow[from=4-5, to=3-6]
	\arrow[from=4-6, to=3-6]
	\arrow[from=4-7, to=3-6]
	\arrow[from=4-9, to=3-9]
	\arrow[from=4-10, to=3-9]
	\arrow[from=4-10, to=3-10]
	\arrow[from=4-10, to=3-11]
	\arrow[from=4-11, to=3-11]
	\arrow[from=5-1, to=4-2]
	\arrow[from=5-2, to=4-2]
	\arrow[from=5-3, to=4-2]
	\arrow[from=5-5, to=4-5]
	\arrow[from=5-6, to=4-5]
	\arrow[from=5-6, to=4-6]
	\arrow[from=5-6, to=4-7]
	\arrow[from=5-7, to=4-7]
	\arrow[from=5-9, to=4-9]
	\arrow[from=5-10, to=4-9]
	\arrow[from=5-10, to=4-10]
	\arrow[from=5-10, to=4-11]
	\arrow[from=5-11, to=4-11]
	\arrow[from=6-1, to=5-1]
	\arrow[from=6-2, to=5-1]
	\arrow[from=6-2, to=5-2]
	\arrow[from=6-2, to=5-3]
	\arrow[from=6-3, to=5-3]
	\arrow[from=6-5, to=5-5]
	\arrow[from=6-6, to=5-5]
	\arrow[from=6-6, to=5-6]
	\arrow[from=6-6, to=5-7]
	\arrow[from=6-7, to=5-7]
	\arrow[from=6-9, to=5-9]
	\arrow[from=6-10, to=5-9]
	\arrow[from=6-10, to=5-10]
	\arrow[from=6-10, to=5-11]
	\arrow[from=6-11, to=5-11]
	\arrow[from=7-2, to=6-1]
	\arrow[from=7-2, to=6-2]
	\arrow[from=7-2, to=6-3]
	\arrow[from=7-6, to=6-5]
	\arrow[from=7-6, to=6-6]
	\arrow[from=7-6, to=6-7]
	\arrow[from=7-10, to=6-9]
	\arrow[from=7-10, to=6-10]
	\arrow[from=7-10, to=6-11]
\end{tikzcd}
\caption{The lattices $\Sub(D_{2^n})/D_{2^n}$ for $n=2,3$ and $n$ arbitrary.}
\label{fig: D2^n lattices}
\end{figure}

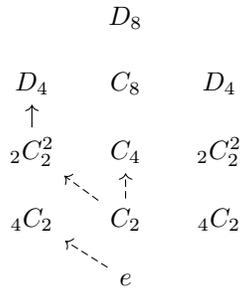
\begin{figure}[h]
\begin{tikzcd}[column sep=small, row sep=small]
	& {D_8} \\
	{D_4} & {C_8} & {D_4} \\
	{{}_2D_2} & {C_4} & {{}_2D_2} \\
	{{}_4C_2} & {C_2} & {{}_4C_2} \\
	& e
	\arrow[from=3-1, to=2-1]
	\arrow[dashed, from=4-2, to=3-1]
	\arrow[dashed, from=4-2, to=3-2]
	\arrow[dashed, from=5-2, to=4-1]
\end{tikzcd}
\caption{A DTC restriction that arises in a $D_8$-transfer system.}
\label{fig: d8 dtc restrictions}
\end{figure}

We now explore some DTC restrictions. The vertical edges in the wings each induce an edge from the spine to the wing. For example, ${}_{(2^{n-1})}C_2 \to {}_{(2^{n-2})}D_2$ induces the edge $e \to {}_{(2^{n-1})}C_2$ via restriction by a conjugate copy of $C_2$.
Note the source of the first edge is the target of the second edge. In general, if ${}_{(2^{n - i})}X \to {}_{(2^{n - i - 1})}Y$ is a vertical edge on a wing, this induces the short edge $C_{2^{i-1}} \to {}_{(2^{n - i})}X$, for $1 \leq i < n$. \Cref{fig: d8 dtc restrictions} depicts an example of a DTC restriction in $D_8$, where the dotted edges are induced by the restriction of the solid edge. Note, we also have edges $C_2 \to C_4$ and $e \to {}_4 C_2$ in \Cref{fig: d8 dtc restrictions}; these are the usual restrictions of ${}_2D_2 \to D_4$ and $C_2 \to {}_2D_2$, respectively. 
\end{example}

Another example of lossless groups are dicyclic groups \cite{Lossless}. Dicyclic, or binary dihedral groups $\Dic_n$, are groups of order $4n$ with a unique non-split extension $C_{2n} \cdot C_2$, where $C_2$ acts by $-1$. We wish to consider transfer systems on $\Dic_{p^n}$, for $p$ an odd prime. 

\begin{example}\label{ex: Dic p^n}
    Let $p$ be an odd prime and consider the complete $\Dic_{p^n}$-transfer systems, with composition edgers omitted. Note that the lattices $\Sub(\Dic_{p^n})/\Dic_{p^n}$ build inductively, as shown in \Cref{fig: dic_p^n lattices}.
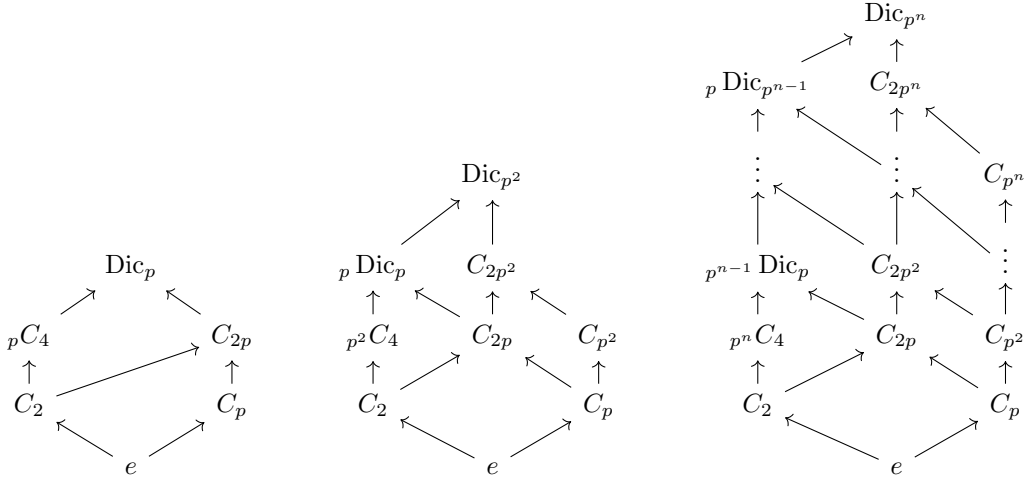
\begin{figure}[h]
\begin{tikzcd}[column sep=small, row sep=small]
	&&&&&&&&& {\Dic_{p^{n}}} \\
	&&&&&&&& {{}_{p}\Dic_{p^{n-1}}} & {C_{2p^n}} \\
	&&&&& {\Dic_{p^2}} &&& \vdots & \vdots & {C_{p^n}} \\
	& {\Dic_{p}} &&& {{}_p\Dic_p} & {C_{2p^2}} &&& {{}_{p^{n-1}}\Dic_p} & {C_{2p^2}} & \vdots \\
	{{}_pC_4} && {C_{2p}} && {{}_{p^2}C_4} & {C_{2p}} & {C_{p^2}} && {{}_{p^n}C_4} & {C_{2p}} & {C_{p^2}} \\
	{C_2} && {C_p} && {C_2} && {C_p} && {C_2} && {C_p} \\
	& e &&&& e &&&& e
	\arrow[from=2-9, to=1-10]
	\arrow[from=2-10, to=1-10]
	\arrow[from=3-9, to=2-9]
	\arrow[from=3-10, to=2-9]
	\arrow[from=3-10, to=2-10]
	\arrow[from=3-11, to=2-10]
	\arrow[from=4-5, to=3-6]
	\arrow[from=4-6, to=3-6]
	\arrow[from=4-9, to=3-9]
	\arrow[from=4-10, to=3-9]
	\arrow[from=4-10, to=3-10]
	\arrow[from=4-11, to=3-10]
	\arrow[from=4-11, to=3-11]
	\arrow[from=5-1, to=4-2]
	\arrow[from=5-3, to=4-2]
	\arrow[from=5-5, to=4-5]
	\arrow[from=5-6, to=4-5]
	\arrow[from=5-6, to=4-6]
	\arrow[from=5-7, to=4-6]
	\arrow[from=5-9, to=4-9]
	\arrow[from=5-10, to=4-9]
	\arrow[from=5-10, to=4-10]
	\arrow[from=5-11, to=4-10]
	\arrow[from=5-11, to=4-11]
	\arrow[from=6-1, to=5-1]
	\arrow[from=6-1, to=5-3]
	\arrow[from=6-3, to=5-3]
	\arrow[from=6-5, to=5-5]
	\arrow[from=6-5, to=5-6]
	\arrow[from=6-7, to=5-6]
	\arrow[from=6-7, to=5-7]
	\arrow[from=6-9, to=5-9]
	\arrow[from=6-9, to=5-10]
	\arrow[from=6-11, to=5-10]
	\arrow[from=6-11, to=5-11]
	\arrow[from=7-2, to=6-1]
	\arrow[from=7-2, to=6-3]
	\arrow[from=7-6, to=6-5]
	\arrow[from=7-6, to=6-7]
	\arrow[from=7-10, to=6-9]
	\arrow[from=7-10, to=6-11]
\end{tikzcd}
\caption{The lattices $\Sub(\Dic_{p^n})/\Dic_{p^n}$ for $n=1, 2$ and $n$ arbitrary.}
\label{fig: dic_p^n lattices}
\end{figure}

We now explore some DTC restrictions. Note that each copy of $\Dic_p$ has $p$ conjugate copies of $C_4$, consider two of them, say $H$ and $K$. The edge $H \to \Dic_p$ induces the edge $C_2 \to K$ via restriction by $K$. This DTC restriction in a $\Dic_{p^2}$-transfer system is shown on the left in \Cref{fig: Dic_p restrictions}, where the bold dotted edge is the DTC restriction and the other dotted edges are restrictions of the solid edge.

For $n > 1$, consider a copy of $\Dic_{p^{n-i+1}}$ and let $H$ and $K$ be two of its $\Dic_{p^{n-i}}$ subgroups, for $1 \leq i \leq n-1$. The edge $H \to \Dic_{p^{n-i+1}}$ induces the edge $C_{2p^{n-i}} \to K$ via restriction by $K$. The right side of \Cref{fig: Dic_p restrictions} depicts an example of such a restriction, where the bold dotted edge is the DTC restriction and the other dotted edges are restrictions of either the solid edge or the DTC restriction.
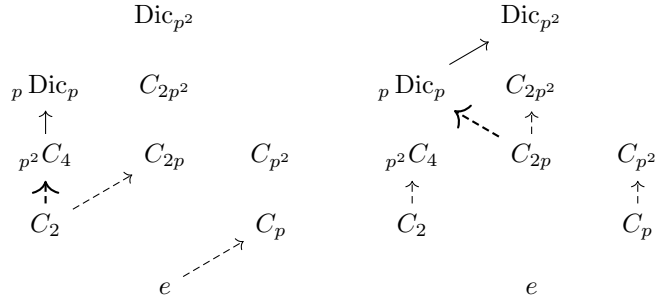
\begin{figure}[h]
\[\begin{tikzcd}[column sep=small, row sep=small]
	& {\Dic_{p^2}} &&&& {\Dic_{p^2}} \\
	{{}_p\Dic_p} & {C_{2p^2}} &&& {{}_p\Dic_p} & {C_{2p^2}} \\
	{{}_{p^2}C_4} & {C_{2p}} & {C_{p^2}} && {{}_{p^2}C_4} & {C_{2p}} & {C_{p^2}} \\
	{C_2} && {C_p} && {C_2} && {C_p} \\
	& e &&&& e
	\arrow[from=2-5, to=1-6]
	\arrow[from=3-1, to=2-1]
	\arrow[dashed, thick, from=3-6, to=2-5]
	\arrow[dashed, from=3-6, to=2-6]
	\arrow[dashed, thick, from=4-1, to=3-1]
	\arrow[dashed, from=4-1, to=3-2]
	\arrow[dashed, from=4-5, to=3-5]
	\arrow[dashed, from=4-7, to=3-7]
	\arrow[dashed, from=5-2, to=4-3]
\end{tikzcd}\]
\caption{Some DTC restrictions that arise in $\Dic_{p^2}$-transfer systems.}
\label{fig: Dic_p restrictions}
\end{figure}
\end{example}

Now that we have discussed some examples of transfer systems and explored the challenges that arise in the non-abelian setting, we will discuss properties of transfer systems that are relevant for studying equivariant homotopy theory.

\subsection{Properties of Transfer Systems} In this section we will define compatible pairs of $G$-transfer systems, discuss properties of $G$-transfer systems, and their pairs. In fact, some properties of $G$-transfer systems can help us study the group $G$. We will start by working towards defining a group invariant determined by $G$-transfer systems.

Rubin gives a way to generate a $G$-transfer system from a collection of edges, see \cite[Construction A.1]{RubinDetectingOperads}. A $G$-transfer system, say $T$, has at least one \emph{minimal generating set}, a collection of edges that generate $T$ such that if any edge is removed from the set it will no longer generate $T$; see \cite[Definition 2.6]{rainbowMRC}. Note that minimal generating sets never have any reflexive edges since one of the steps in Rubin's algorithm inserts all reflexive edges.

The authors of \cite{rainbowMRC} give a well-defined map that sends a $G$-transfer system to the cardinality of one of its minimal generating sets. Of particular interest to us is where this map sends the complete transfer system.

\begin{definition}\cite[Definition 4.1]{rainbowMRC}\label{def:width}
    The \emph{width} of a group $G$, denoted $\w(G)$, is the cardinality of a minimal generating set of the complete $G$-transfer system.
\end{definition}

The authors of \cite{rainbowMRC} compute $\w(G)$ for some groups, including $S_5$ and $F_8$. They also prove that $\w(C_{p_1 ... p_n})=n$, $\w(C_{p^nq}) = n + 1$, and use results of \cite{LatticeVia} to demonstrate that $\w(C_{p^n}) = n$. We compute the width of all dihedral and quaternion groups in \Cref{sec: dihedral and quats}, the width of all dicyclic groups in \Cref{sec: dicyclic}, and the width of some Frobenius groups in \Cref{sec: Frobenius}.

\begin{remark}
    The authors of \cite{rainbowMRC} also define a related invariant called the \emph{complexity} of a group $G$. One can find a minimal generating set for every $G$-transfer system; the complexity is the maximum size of any such minimal generating set. Note that complexity is a more difficult invariant to compute; the authors of \cite{rainbowMRC} do so for several abelian groups. We wish to consider the complexity of some non-abelian groups in future work.
\end{remark}

The collection of $G$-transfer systems for a fixed $G$ form a poset under inclusion and create a lattice. We often depict these lattices as \emph{Hasse diagrams}, where only the short edges are drawn so there are no composite edges. These Hasse diagrams can be used as tools to visualize patterns and check conjectures, among other things. In \Cref{sec: lattices}, we provide new Hasse diagrams for some non-abelian groups. In addition, we have indicated the following properties for each transfer system in the Hasse diagram: saturated, cosaturated, and lesser simply paired (LSP). We will start by working towards recalling LSP, but first we need to recall the definition of compatible pairs of transfer systems.

\begin{definition}\cite[Definition 4.6]{ChanTambara}\label{def:CompatTSs}
    Let $T_1$ and $T_2$ be $G$-transfer systems. Then $(T_1,T_2)$ is a \emph{compatible pair} if $T_1$ and $T_2$ satisfy the following criteria. 
    \begin{enumerate}
        \item $T_1 \subseteq T_2$ 
        \item Suppose $B,C \leq A \leq G$. If $B \to A$ is in $T_1$ and $B\cap C \to B$ is in $T_2$, then $C \to A$ must be in $T_2$.
    \end{enumerate}
\end{definition}

This second condition can be visualized as in \Cref{fig:compatibility condition} where, if the solid edges are in the corresponding transfer systems, then the dotted edge must be in $T_2$ for the transfer systems to be compatible. Observe that the edge $B \cap C \to C$ is in $T_1$ via restriction of $B \to A$ by $C$.

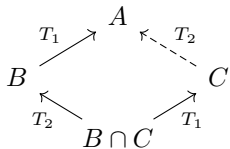
\begin{figure}[hbt!]
\centering
\begin{tikzcd}[column sep=small, row sep=small]
	& A \\
	B && C \\
	& {B \cap C}
	\arrow["{T_1}", from=2-1, to=1-2]
	\arrow["{T_2}"', dashed, from=2-3, to=1-2]
	\arrow["{T_2}", from=3-2, to=2-1]
	\arrow["{T_1}"', from=3-2, to=2-3]
\end{tikzcd}
\caption{Compatibility condition from \Cref{def:CompatTSs}.}
    \label{fig:compatibility condition}
\end{figure}

In the case of $B = B \cap C$, the diagram  in \Cref{fig:compatibility condition} reduces to \Cref{fig:saturation}. This special case is related to the following definition about the \emph{saturation} of a transfer system.

\begin{definition}
    A $G$-transfer system $T$ is \emph{saturated} if for any sequence of subgroups $L \leq K \leq H \leq G$, the inclusion of edge $L \to H$ also implies $K \to H$ is an edge in $T$. This condition can be visualized as
\[\begin{tikzcd}[column sep=small, row sep=small]
	L & K & H,
	\arrow[from=1-1, to=1-2]
	\arrow[bend left, from=1-1, to=1-3]
	\arrow[dashed, from=1-2, to=1-3]
\end{tikzcd}\]
\noindent where if the solid edges are in $T$ then the dotted edge must be in $T$ for the transfer system to be saturated. Note that the edge $L \to K$ will be in $T$ since $T$ has $L \to H$ and is closed under restriction.
\end{definition}

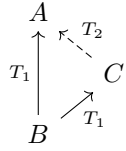
\begin{figure}[hbt!]
\centering
\begin{tikzcd}[column sep=small, row sep=small]
	A \\
	& C \\
	B
	\arrow["{T_2}"', dashed, from=2-2, to=1-1]
	\arrow["{T_1}", from=3-1, to=1-1]
	\arrow["{T_1}"', from=3-1, to=2-2]
\end{tikzcd}
\caption{Compatibility condition when $B = B \cap C$.}
    \label{fig:saturation}
\end{figure}

One can saturate any transfer system $T$ by adding in only the edges necessary for it to be saturated. This process produces the same transfer system as the following definition.

\begin{definition}
    Let $T$ be a $G$-transfer system. The \emph{saturated hull} of $T$, denoted $\Hull(T)$, is the smallest saturated $G$-transfer system that contains $T$.
\end{definition}

Due to the special case of the compatibility condition demonstrated in \Cref{fig:saturation}, we know that for $(T,T')$ to be compatible, $\Hull(T) \subseteq T'$. Therefore for any $G$-transfer system, $T$, the smallest $G$-transfer system compatible with $T$ is $\Hull(T)$; that is, $(T,\Hull(T))$ is always a compatible pair. Recall that the complete $G$-transfer system has all possible edges, therefore it can not cause a compatibility failure by missing an edge. So every $G$-transfer system $T$ is compatible with the complete transfer system $C$; that is, $(T,C)$ is always a compatible pair.  

\begin{definition}\cite[Definition 3.9]{LSP}
    A $G$-transfer system $T$ is \emph{lesser simply paired (LSP)} if, for a compatible pair $(T,T')$, $T'$ can only be the hull of $T$ or the complete transfer system.
\end{definition}

There exists a dual notion to saturation for transfer systems as well.

\begin{definition}
    A $G$-transfer system $T$ is \emph{cosaturated} (or \emph{disk-like}) if all edges in $T$ are generated from edges of the form $H \to G$ for $H \leq G$. That is, there exists a minimal generating set of $T$ that only has edges of the form $H \to G$.
\end{definition}

\begin{remark}
    Apriori, the given definitions of saturated and cosaturated do not appear dual. However, the transfer systems on any lattice $P$ are in bijection with the weak factorization systems on $P$ \cite{LatticeVia}. One may reframe the notion of a (co)saturated transfer system in the language of factorization systems to observe this duality; for further details on duality in transfer and factorization systems see \cite{DualityMRC}. 
\end{remark}  

Some of the most interesting transfer systems are those which are both saturated and cosaturated; we refer to such transfer systems as \emph{bisaturated}. In \Cref{sec: lattices}, we use the decorated Hasse diagrams to form a conjecture about bisaturated paths in the lattice of transfer systems. We plan to return to this conjecture in future work. 

\section{Width of Some Non-Abelian Groups} \label{sec: width}

In this section, we compute the width (see \Cref{def:width}) of all dihedral, quaternion, and dicyclic groups. We also investigate the width of some Frobenius groups. Crucial to this enumeration is a result of \cite{rainbowMRC}, which allows us to determine the width of a group $G$ by counting the number of \emph{meet-irreducible} subgroups of $G$. In a lattice, the meet of two vertices $a$ and $b$ is denoted $a \smashy b$ and is defined to be the greatest (in terms of the poset structure on the lattice) vertex that has an edge into both $a$ and $b$. On $\Sub(G)$, the meet can be expressed as the intersection of groups.

\begin{definition}
    Let $G$ be a group. We say that a subgroup $H 
    < G$ is \emph{meet-irreducible} if it cannot be expressed as $K \cap L$ for $K$ and $L$ distinct subgroups such that $H < K,L < G$. 
\end{definition} 

\begin{theorem}[\cite{rainbowMRC}, Proposition 4.5]\label{thm: width is meet irred}
    Let $G$ be a finite group. Then $\w(G)$ is equal to the number of conjugacy classes of meet-irreducible subgroups of $G$. In particular, the collection
    \[
    S = \{[H] \to G \mid H < G \text{ is meet-irreducible}\}
    \]
    is a minimal generating set for the complete transfer system. 
\end{theorem}

Note that all maximal subgroups of any group $G$ are meet-irreducible. When $G$ is abelian, a proper subgroup is meet-irreducible if and only if it is the source of exactly one short edge in $\Sub(G)$. The same is not true on the reduced lattice for non-abelian groups, however.
Even if $K$ is the source of a single arrow on the reduced lattice $\Sub(G)/G$, $K \rightarrow H$, it may be that $K = H \cap gHg^{-1}$, thus $K$ is not meet-irreducible. 
For example, consider the case where $K < H < G$ such that there exists a $g \in G$ where $K = gKg^{-1}$ and $H \neq gHg^{-1}$.
Therefore, if the target $H$ is normal in $G$, this issue of conjugacy does not arise. 

More concretely, let $F_8$ be the Frobenius group of order 56; on $\Sub(F_8)/F_8$ it appears that the 7 conjugate copies of $C_2$ would be meet-irreducible, but they are not \cite[Remark 4.8]{rainbowMRC}. The authors of \cite{rainbowMRC} note that these examples seem to only appear for lossy groups, in the sense of \cite{Lossless}.

There exist groups $G$ such that the trivial group $e$ is meet-irreducible, for example the quaternions. The following result helps us to say something about the width of these groups. On a general lattice, a transfer system is a poset that is closed under pullbacks and composition, and has all reflexive edges.

\begin{lemma}
    \cite[Lemma 4.10]{eCHTREU} \label{lem: remove vertex} Consider a lattice $P$ and let $v$ be the bottom-most vertex, in terms of the poset structure on $P$. There is a bijection between transfer systems on $P$ which have every edge $v \to a$ for all vertices $a$ in $P$, and transfer systems on $P\backslash \{v\}$, where $P \backslash \{v\}$ is the lattice $P$ without the bottom vertex $v$.
\end{lemma}

\begin{proposition} \label{prop: width formula for lolipops}
    Let $G$ and $G'$ be finite groups, where the subgroup $e$ is meet-irreducible in $G$. If $\Sub(G')$ is the same (undecorated) lattice as $\Sub(G)\backslash\{e\}$ then $\w(G) = \w(G') + 1$.
\end{proposition}

\begin{proof}
    Since $e$ is meet-irreducible, then any minimal generating set of the complete $G$-transfer system must contain an edge from $e$ to some non-trivial subgroup of $G$. \Cref{lem: remove vertex} tells us that the complete transfer system on $\Sub(G)$ is in relation to the complete transfer system on $\Sub(G')$. Removing this vertex will not cause another vertex to become meet-irreducible, or cause another vertex to no longer be meet-irreducible. Therefore $\w(G) = \w(G') + 1$ as desired.
\end{proof}

We will find the width of the dihedral groups and then use the above result to compute the width of the quaternion groups.

\subsection{Dihedral and Quaternion Groups}\label{sec: dihedral and quats}

Recall that $D_n$ is the dihedral group of order $2n$. Questions about dihedral transfer systems have been studied previously; for instance, the authors of \cite{DihedralTransferSystems} give a recursive formula to enumerate $D_{p^n}$-transfer systems. Currently, a closed formula to determine the number of dihedral transfer systems is not known however. In this section, we contribute further to the body of knowledge on non-abelian transfer systems by computing the width of all dihedral groups. Additionally, we extend our technique to determine the width of quaternion groups. For reference, the Hasse diagram of all $D_{4}$ and $D_{p^2}$-transfer systems, for $p$ an odd prime, are included in \Cref{sec: lattices}.

\begin{theorem}\label{thm: dihedral width} Let $D_n$ be a dihedral group of order $2n$ for $n > 1$. The width of $D_n$ is completely determined by the prime factorization of $n=2^mp_1^{i_1} \cdots p_k^{i_k}$ using the following equation:
    \[\w(D_{2^mp_1^{i_1} \cdots p_k^{i_k}}) = 2m + 1 + \sum_{\ell=1}^k i_\ell.\]
\end{theorem}

Before the proof of this theorem, we recall some helpful facts about $\Sub(D_n)$ in general, and provide the reader with the following example to build some intuition. The reader can also recall Examples \ref{ex: dihedral odds} and \ref{ex: dihedral evens}, however those examples have only one distinct prime and therefore do not capture the full story.

\begin{remark}
Let $D_n$ be the dihedral group of order $2n$. Every subgroup of $D_n$ is cyclic or dihedral. Recall that the Klein 4-group goes by many different names and notations; in particular, it can be expressed with dihedral group notation as $C_2^2 \cong D_2$. Further, $D_\ell,C_\ell \leq D_n$ for every divisor $\ell \mid n$. Let us recall some facts about the number of copies of each subgroup $D_\ell$ and $C_\ell$, along with their conjugacy classes:
    \begin{itemize}
    \item There is exactly one copy of $C_\ell$, unless $\ell = 2$, in which case there are cases based on the parity of $n$:
    \begin{itemize}
        \item If $n$ is even, then there are $n + 1$ copies of $C_2$. These split into three conjugacy classes: two with $\frac{n}{2}$ copies and one normal subgroup (this is the center).
        \item If $n$ is odd, then there are $n$ copies of $C_2$ that are all in one conjugacy class.
    \end{itemize}
    \item There are $\frac{n}{\ell}$ copies of $D_\ell$. The conjugacy classes depend on the parity of $\frac{n}{\ell}$:
    \begin{itemize}
        \item If $\frac{n}{\ell}$ is even, then there are two conjugacy classes, each with $\frac{n}{2\ell}$ copies of $D_\ell$.
        \item If $\frac{n}{\ell}$ is odd, then there is one conjugacy class with all copies of $D_\ell.$
    \end{itemize}
\end{itemize}
\end{remark}

\begin{example}\label{ex: dihedral example}
    Consider $D_{36}$, the dihedral group of order $72$, and note that the prime factorization is $36 = 2^2 \cdot 3^2$. \Cref{fig: D36} depicts the complete transfer system on the $\Sub(D_{36})/D_{36}$ lattice. For ease of notation we omit all reflexive and composition edges.
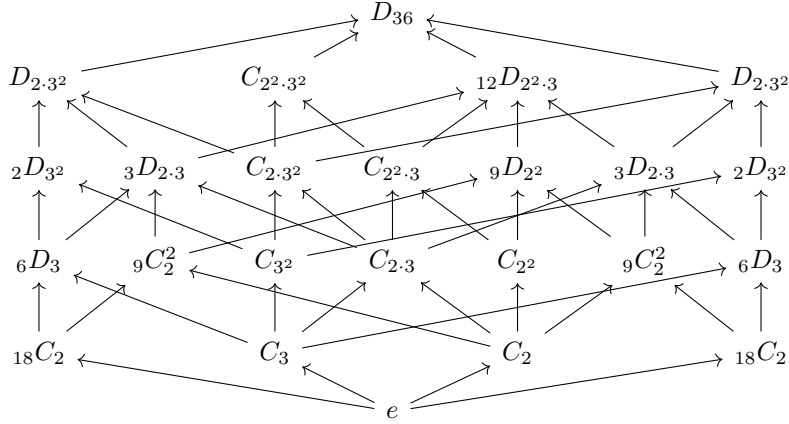
\begin{figure}[h]
\[\begin{tikzcd}[column sep=small, row sep=small]
	&&& {D_{36}} \\
	{D_{2 \cdot 3^2}} && {C_{2^2 \cdot 3^2}} && {{}_{3}D_{2^2 \cdot 3}} && {D_{2\cdot 3^2}} \\
	\\
	{{}_{2}D_{3^2}} & {{}_{3}D_{2\cdot 3}} & {C_{2\cdot 3^2}} & {C_{2^2 \cdot 3}} & {{}_{9}D_{2^2}} & {{}_{3}D_{2\cdot 3}} & {{}_{2}D_{3^2}} \\
	\\
	{{}_{6}D_3} & {{}_{9}D_2} & {C_{3^2}} & {C_{2\cdot 3}} & {C_{2^2}} & {{}_{9}D_2} & {{}_{6}D_3} \\
	\\
	{{}_{18}C_2} && {C_3} && {C_2} && {{}_{18}C_2} \\
	&&& e
	\arrow[from=2-1, to=1-4]
	\arrow[from=2-3, to=1-4]
	\arrow[from=2-5, to=1-4]
	\arrow[from=2-7, to=1-4]
	\arrow[from=4-1, to=2-1]
	\arrow[from=4-2, to=2-1]
	\arrow[from=4-2, to=2-5]
	\arrow[from=4-3, to=2-1]
	\arrow[from=4-3, to=2-3]
	\arrow[from=4-3, to=2-7]
	\arrow[from=4-4, to=2-3]
	\arrow[from=4-4, to=2-5]
	\arrow[from=4-5, to=2-5]
	\arrow[from=4-6, to=2-5]
	\arrow[from=4-6, to=2-7]
	\arrow[from=4-7, to=2-7]
	\arrow[from=6-1, to=4-1]
	\arrow[from=6-1, to=4-2]
	\arrow[from=6-2, to=4-2]
	\arrow[from=6-2, to=4-5]
	\arrow[from=6-3, to=4-1]
	\arrow[from=6-3, to=4-3]
	\arrow[from=6-3, to=4-7]
	\arrow[from=6-4, to=4-2]
	\arrow[from=6-4, to=4-3]
	\arrow[from=6-4, to=4-4]
	\arrow[from=6-4, to=4-6]
	\arrow[from=6-5, to=4-4]
	\arrow[from=6-5, to=4-5]
	\arrow[from=6-6, to=4-5]
	\arrow[from=6-6, to=4-6]
	\arrow[from=6-7, to=4-6]
	\arrow[from=6-7, to=4-7]
	\arrow[from=8-1, to=6-1]
	\arrow[from=8-1, to=6-2]
	\arrow[from=8-3, to=6-1]
	\arrow[from=8-3, to=6-3]
	\arrow[from=8-3, to=6-4]
	\arrow[from=8-3, to=6-7]
	\arrow[from=8-5, to=6-2]
	\arrow[from=8-5, to=6-4]
	\arrow[from=8-5, to=6-5]
	\arrow[from=8-5, to=6-6]
	\arrow[from=8-7, to=6-6]
	\arrow[from=8-7, to=6-7]
	\arrow[from=9-4, to=8-1]
	\arrow[from=9-4, to=8-3]
	\arrow[from=9-4, to=8-5]
	\arrow[from=9-4, to=8-7]
\end{tikzcd}\]
\caption{The subgroup lattice for the group $D_{36}$.}
\label{fig: D36}
\end{figure}

By \Cref{thm: width is meet irred}, we only need to find the meet-irreducible subgroups to determine the width of $D_{36}$. First one may consider the maximal subgroups of any given group, as they are always meet-irreducible. In this example, the maximal subgroups are $C_{36}, D_{12}$, and two non-conjugate copies of $D_{18}$. The subgroups $D_{\frac{n}{p}}$ and $C_n$ are maximal for any dihedral group $D_n$, where $p$ is an odd prime factor of $n$. Note that the pair of non-conjugate maximal subgroups, $D_{\frac{n}{2}}$, is a feature only enjoyed by dihedral groups $D_n$ for $n$ even.

One can think of these maximal dihedral subgroups as the start of a ``path", where one travels down a path by dividing out a prime factor from the decoration of a group. Starting at $D_n$ and dividing out two distinct primes will not result in a meet-irreducible subgroup. For instance, $D_{2 \cdot 3}$ is not meet-irreducible as $D_{2^2\cdot 3} \wedge D_{2 \cdot 3^2} = D_{2 \cdot 3}$.

On the other hand, if one starts at $D_n$ and divides out the same prime repeatedly until that prime is gone, this yields a path of meet-irreducible subgroups. For instance, consider the following path of edges:
\[
D_{36}=D_{2^2 \cdot 3^2} \leftarrow D_{2^2 \cdot 3} \leftarrow D_{2^2}.
\]
\noindent From the lattice in \Cref{fig: D36} we see that $D_{2^2 \cdot 3}$ and $D_{2^2}$ are meet-irreducible candidates. Through inspection of their conjugates, one can find that they are, in fact, meet-irreducible.

Similarly we have two non-conjugate copies of the following path (one on the left side of the lattice and one on the right):
\[
D_{36}=D_{2^2 \cdot 3^2} \leftarrow D_{2 \cdot 3^2} \leftarrow D_{3^2}.
\]
\noindent By a similar inspection, we see that $D_{2 \cdot 3^2}$ and $D_{3^2}$ are meet-irreducible.

After counting all of these meet-irreducible subgroups, we conclude that the width of $D_{36}$ is $7$, and a minimal generating set for the complete transfer system is as follows:
\[
S = \{C_{36} \to D_{36}, D_{2^2 \cdot 3} \to D_{36}, D_{2^2} \to D_{36}, D_{2 \cdot 3^2} \to D_{36}, D_{2 \cdot 3^2} \to D_{36}, D_{3^2} \to D_{36}, D_{3^2} \to D_{36}\}.
\]

\end{example}

This example agrees with the formula of \Cref{thm: dihedral width}.

\begin{proof}
    By \Cref{thm: width is meet irred}, it suffices to determine the number of conjugacy classes of meet-irreducible subgroups of $D_n$ to compute its width. Note that all subgroups of $D_n$ are either cyclic or dihedral.
    
    The cyclic groups are all of the form $C_\ell$, where $\ell$ divides $n$. Recall that if $C_\ell < D_s$ then $\ell$ divides $s$, so if $\ell \neq n$ then $D_\ell \smashy C_n = C_\ell$. Therefore the only cyclic meet-irreducible subgroup is the maximal subgroup $C_n$. We add the edge $C_n \to D_n$ to our minimal generating set.

    We now consider the meet-irreducible subgroups of $D_n$ which are dihedral. Note that if $D_r \leq D_s$, then $r$ must divide $s$. For any two distinct primes $q$ and $p_j$ in the prime factorization of $n$, the subgroup $D_t$ for $t = \frac{2^mp_1^{i_1} \cdots p_k^{i_k}}{q \cdot p_j}$ will not be meet-irreducible as $D_{tq} \smashy D_{tp_j} = D_t$. Therefore the only dihedral groups which are candidates to be meet-irreducible are of the form $D_t$ for $t = 2^mp_1^{i_1}\cdots p_\ell^{(i_\ell - r)} \cdots p_k^{i_k}$ for $1 \leq r \leq i_\ell$, or $t = 2^{(m-r)}p_1^{i_1}\cdots p_n^{i_n}$ for $1 \leq r \leq m$.

    The subgroup $D_{2^mp_1^{i_1} \cdots p_\ell^{(i_\ell-1)} \cdots p_k^{i_k}}$ for all $1 \leq \ell \leq k$ is maximal, and thus meet-irreducible. Since $2^mp_1^{i_1} \cdots p_\ell^{(i_\ell-2)} \cdots p_k^{i_k}$ only divides $2^mp_1^{i_1} \cdots p_\ell^{(i_\ell-1)} \cdots p_k^{i_k}$ and does not divide $s$ for any other proper dihedral subgroup $D_s$ in $D_n$, we have that $D_{2^mp_1^{i_1} \cdots p_\ell^{(i_\ell-2)} \cdots p_k^{i_k}}$ is meet-irreducible. Inductively, using the fact that 
    \[D_{2^mp_1^{i_1} \cdots p_\ell^{(i_\ell-y)} \cdots p_k^{i_k}} \wedge D_{2^mp_1^{i_1} \cdots p_\ell^{(i_\ell-z)} \cdots p_k^{i_k}} = D_{2^mp_1^{i_1} \cdots p_\ell^{(i_\ell-\max(y,z))} \cdots p_k^{i_k}},\]
    one sees that the following collection of subgroups are meet-irreducible
    \[
    S' = \{D_{2^mp_1^{i_1} \cdots p_\ell^{(i_\ell-r)} \cdots p_k^{i_k}} \mid 1 \leq r \leq i_\ell \text{ and } 1 \leq \ell \leq k\}.
    \]
    An edge from each subgroup in $S'$ to $D_n$ is included in the minimal generating set (described in \Cref{thm: width is meet irred}), adding $\sum_{\ell=1}^k i_\ell$ edges. 

    Finally, by a similar inductive argument as above we have the following collection of meet-irreducible subgroups
    \[
    S'' = \{ D_{2^{(m-r)}p_1^{i_1} \cdots p_k^{i_k}} \mid 1 \leq r \leq m \}.
    \]
    An edge from each subgroup in $S''$ to $D_n$ is included in the minimal generating set (described in \Cref{thm: width is meet irred}), since there are two non-conjugate copies of each subgroup $D_{2^{(m-r)}p_1^{i_1} \cdots p_k^{i_k}}$ this adds $2m$ edges.

    Therefore, the following is a minimal generating set for the complete transfer system on $D_n$
    \[
    \{H \to D_n \mid H=C_n,\ H\in S', \text{ or } \, H \in S''\}.
    \]
    This determines that the width of $D_n$ is
     \[\w(D_{2^mp_1^{i_1} \cdots p_k^{i_k}}) = 2m + 1 + \sum_{\ell=1}^k i_\ell.\]\end{proof}

One can observe that the subgroup lattices of quaternion groups resemble those of dihedral groups $D_{2^m}$ with an additional edge at the very bottom of the diagram coming from $e$. We can extend proof technique of \Cref{thm: dihedral width} to quaternion groups by using \Cref{prop: width formula for lolipops} to account for this additional edge. 

\begin{corollary}\label{cor: quaterion width}
    The width of a quaternion group $Q_{2^{m+2}}$ is $4$ when $m = 1$ and is $\w(D_{2^{m}}) + 1$ when $m >1$, meaning
    \[
    \w(Q_{2^{m+2}}) = 2m + 2.
    \]
\end{corollary}

\begin{proof}
    For the case when $m = 1$, we have the complete lattice shown in \Cref{fig: quaternion} and can count the four meet-irreducible subgroups, each non-conjugate $C_4$ and e. 
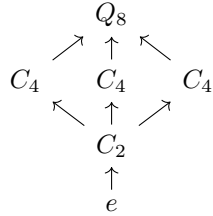
\begin{figure}[h]
\begin{tikzcd}[column sep=small, row sep=small]
	& {Q_8} \\
	{C_4} & {C_4} & {C_4} \\
	& {C_2} \\
	& e
	\arrow[from=2-1, to=1-2]
	\arrow[from=2-2, to=1-2]
	\arrow[from=2-3, to=1-2]
	\arrow[from=3-2, to=2-1]
	\arrow[from=3-2, to=2-2]
	\arrow[from=3-2, to=2-3]
	\arrow[from=4-2, to=3-2]
\end{tikzcd}
    \caption{The subgroup lattice for the group $Q_8$.}
    \label{fig: quaternion}
\end{figure}

    For the case when $m > 1$, we combine the results of \Cref{prop: width formula for lolipops} and \Cref{thm: dihedral width} to get our desired result.\end{proof}

\begin{remark}
   The width of generalized quaternion groups was also proven independently by Evans in work to appear \cite{EvansQuaternions}. Evans' results agree with our calculation of width and further explore the complexity of quaternion groups.  
\end{remark}

\subsection{Dicyclic groups} \label{sec: dicyclic}

Dicyclic groups, or binary dihedral groups, $\Dic_n$ are groups of order $4n$, with a unique non-split extension $C_{2n} \cdot C_2$ where $C_2$ acts by $-1$ \cite{GroupNames}. One can also express $\Dic_n$ as $C_n \rtimes C_4$ for $n$ odd or $C_m \rtimes Q_{2^{k+2}}$ for $n= 2^km$ \cite{GroupNames}. Note then that for $n = 2^m$, $\Dic_n = Q_{2^{m+2}}$.

\begin{theorem} \label{thm: dicyclic width}
    Let $\Dic_n$ be the dicyclic group of order $4n$ for $n \geq 2$. The width of $\Dic_n$ is completely determined by the prime factorization of $n=2^m p_1^{i_1} \cdots p_k^{i_k}$ using the following equation:
    \[
    \w(\Dic_{2^m p_1^{i_1} \cdots p_k^{i_k}}) = 2m + 2 + \sum_{\ell=1}^k i_\ell.
    \]
\end{theorem}

Before the proof of this theorem, we provide the reader with the following example to build some intuition. The reader can also recall \Cref{ex: Dic p^n} however that example only has one distinct prime and therefore does not capture the full story.

\begin{example}
Consider $\Dic_{18}$, the dicyclic group of order $72$, and note that the prime factorization is $18 = 2 \cdot 3^2$. \Cref{fig: Dic18} depicts the complete transfer system on the $\Sub(\Dic_{18})/\Dic_{18}$ lattice. For ease of notation we omit all reflexive and composition edges.

    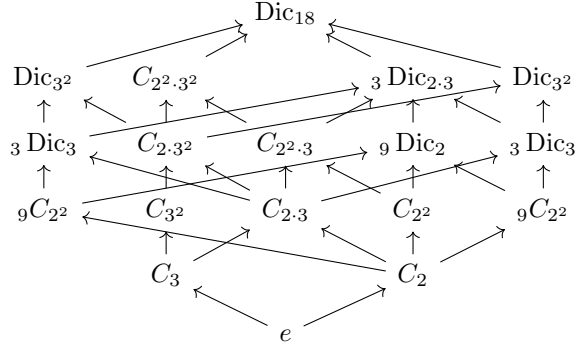
\begin{figure}[h]
\begin{tikzcd}[column sep=small, row sep=small]
	&& {\Dic_{18}} \\
	{\Dic_{3^2}} & {C_{2^2 \cdot 3^2}} && {{}_3\Dic_{2 \cdot 3}} & {\Dic_{3^2}} \\
	{{}_3\Dic_{3}} & {C_{2 \cdot 3^2}} & {C_{2^2 \cdot 3}} & {{}_{9}\Dic_{2}} & {{}_3\Dic_{3}} \\
	{{}_9C_{2^2}} & {C_{3^2}} & {C_{2 \cdot 3}} & {C_{2^2}} & {{}_9C_{2^2}} \\
	& {C_3} && {C_2} \\
	&& e
	\arrow[from=2-1, to=1-3]
	\arrow[from=2-2, to=1-3]
	\arrow[from=2-4, to=1-3]
	\arrow[from=2-5, to=1-3]
	\arrow[from=3-1, to=2-1]
	\arrow[from=3-1, to=2-4]
	\arrow[from=3-2, to=2-1]
	\arrow[from=3-2, to=2-2]
	\arrow[from=3-2, to=2-5]
	\arrow[from=3-3, to=2-2]
	\arrow[from=3-3, to=2-4]
	\arrow[from=3-4, to=2-4]
	\arrow[from=3-5, to=2-4]
	\arrow[from=3-5, to=2-5]
	\arrow[from=4-1, to=3-1]
	\arrow[from=4-1, to=3-4]
	\arrow[from=4-2, to=3-2]
	\arrow[from=4-3, to=3-1]
	\arrow[from=4-3, to=3-2]
	\arrow[from=4-3, to=3-3]
	\arrow[from=4-3, to=3-5]
	\arrow[from=4-4, to=3-3]
	\arrow[from=4-4, to=3-4]
	\arrow[from=4-5, to=3-4]
	\arrow[from=4-5, to=3-5]
	\arrow[from=5-2, to=4-2]
	\arrow[from=5-2, to=4-3]
	\arrow[from=5-4, to=4-1]
	\arrow[from=5-4, to=4-3]
	\arrow[from=5-4, to=4-4]
	\arrow[from=5-4, to=4-5]
	\arrow[from=6-3, to=5-2]
	\arrow[from=6-3, to=5-4]
\end{tikzcd}
\caption{The subgroup lattice for the group $\Dic_{18}$.}
\label{fig: Dic18}
    \end{figure}
    
Using a similar process as in \Cref{ex: dihedral example} we find that the following are the meet-irreducible subgroups:
\[\Dic_{3^2}, C_{2^2 \cdot 3^2}, \Dic_{2 \cdot 3}, \Dic_2, \text{and}  C_{3^2},\]
\end{example}

\noindent where the two copies of $\Dic_{3^2}$ are not conjugate. This example agrees with the formula of \Cref{thm: dicyclic width}. We will now prove the theorem in full generality.

\begin{proof}
  By \Cref{thm: width is meet irred}, it suffices to determine the number of conjugacy classes of meet-irreducible subgroups of $\Dic_n$. All subgroups of $\Dic_n$ are either cyclic or dicyclic. 

  First, we consider the case of $\Dic_{2^m}$. This is the quaternion group $Q_{2^{2m+2}}$ whose width is $2m+2$, by \Cref{cor: quaterion width}. 
    
    Now, consider $\Dic_n$ where $n$ contains at least one odd prime factor. We will show that the only cyclic meet-irreducible subgroups are the maximal subgroup $C_{2n}$ and $C_q$ for $q \coloneqq p_1^{i_1} \cdots p_k^{i_k}$. All maximal subgroups are meet-irreducible, therefore $C_{2n}$ is meet-irreducible, however it will take more work to show that $C_q$ is meet-irreducible. Note that $C_q \leq \Dic_s$ only when $q \mid 4s$, and since $q$ is odd, this is only true when $q \mid s$. Therefore the only subgroups of $\Dic_n$ which have $C_q$ as a proper subgroup are $C_{2^\ell q}$ for $1 \leq \ell \leq m + 1$ and $\Dic_{2^j q}$ for $0 \leq j \leq m$. The meet of any two cyclic groups, say $C_r$ and $C_s$ is $C_{\gcd(r,s)}$, therefore $C_q$ does not arise as the meet of two cyclic groups of the form $C_{2^\ell q}$. Similarly, $C_q$ will not be the meet of any two dicyclic groups of the form $\Dic_{2^j q}$ for any $j$. Lastly, the meet of any cyclic group of the form $C_{2^\ell q}$ for $1 \leq \ell \leq m+1$ and dicyclic group of the form $\Dic_{2^j q}$ for $0 \leq j \leq m$ has the following formula
    \[
    C_{2^\ell q} \smashy \Dic_{2^j q} = \begin{cases} C_{2^{j+1}q} & \ell \geq j + 1 \\
    C_{2^\ell q} & \ell < j + 1.
    \end{cases}
    \]
    \noindent Therefore $C_q$ is meet-irreducible. Further, using the formula above and similar formulas one can see that no other cyclic group is meet-irreducible.

    Let us now discuss which dicyclic subgroups are meet-irreducible. Our argument will be very similar to that in \Cref{thm: dihedral width}. The subgroups $\Dic_{2^m p_1^{i_1} \cdots p_\ell^{(i_\ell - 1)} \cdots p_k^{i_k}}$ and $\Dic_{2^{(m-1)} p_1^{i_1} \cdots p_k^{i_k}}$ are maximal for all $1 \leq \ell \leq k$, and therefore meet-irreducible. Let us now look at the subgroups of these maximal subgroups. Recall that if $\Dic_r < \Dic_s$, then $r \mid s$. If $i_\ell > 1$, then the dicyclic group $D_{2^m p_1^{i_1} \cdots p_\ell^{(i_\ell - 2)} \cdots p_k^{i_k}}$ is a subgroup of $\Dic_{2^m p_1^{i_1} \cdots p_\ell^{(i_\ell - 1)} \cdots p_k^{i_k}}$, but no other subgroup as it does not divide $s$ for any other proper dicyclic group $\Dic_s$ in $\Dic_n$. Similarly, if $m > 1$, then $\Dic_{2^{(m-2)} p_1^{i_1} \cdots p_k^{i_k}}$ is a maximal subgroup of $\Dic_{2^{(m-1)} p_1^{i_1} \cdots p_k^{i_k}}$. We can do an inductive argument similar to the proof of \Cref{thm: dihedral width} to show that the following collections of subgroups are meet-irreducible:
    \[
    S' = \{\Dic_{2^m p_1^{i_1} \cdots p_\ell^{(i_\ell - r)} \cdots p_k^{i_k}} \mid 1 \leq \ell \leq k, 1 \leq r \leq i_\ell\}, \, \, \text{and} \, \, \, S'' =  \{\Dic_{2^{(m-r)} p_1^{i_1} \cdots p_k^{i_k}} \mid 1 \leq r \leq m\}.
    \]

    Therefore, the following is a minimal generating set for the complete transfer system on $\Dic_n$
    \[
    S = \{ H \to \Dic_n \mid H = C_{2n}, H = C_q, H \in S', \text{or} \, H \in S''\}
    \]
    \noindent This determines that the width of $\Dic_n$ is
    \[
    \w(\Dic_{2^m p_1^{i_1} \cdots p_k^{i_k}}) = 2m + 2 + \sum_{\ell=1}^k i_\ell.
    \] \end{proof}

\subsection{Frobenius groups} \label{sec: Frobenius}

For $p$ prime, the Frobenius group $F_{p^n}$ is a group of order $p^n(p^n-1)$ and can be expressed as $\FF_{p^n} \rtimes \FF_{p^n}^\times$ \cite{GroupNames}. 

Using the subgroup lattices on \cite{GroupNames} or the code from \cite{balchincode}, one can compute the number of meet-irreducible subgroups for the following Frobenius groups. We have that $\w(F_p)$ is the sum of the exponents in the prime factorization of $p(p-1)$ when $p = 5,7,9,11,13,17,$ and $19$. On the other hand $\w(F_8)$ (see \cite[Remark 4.8]{rainbowMRC}) is one less than the sum of the exponents and $\w(F_{16})$ is one more than the sum of the exponents. This information is summarized in the table below.
\begin{center}

 \begin{tabular}{c|c|c}
     Group & Prime Factorization of Order & Width \\
     \hline
     $F_5$ & $2^2 \cdot 5$ & $3$ \\
     $F_7$ & $2 \cdot 3 \cdot 7$ & $3$ \\
     $F_8$ & $2^3 \cdot 7$ & $3$ \\
     $F_9$ & $2^3 \cdot 3^2$ & $5$ \\
     $F_{11}$ &$2 \cdot 5 \cdot 11$ & $3$ \\
     $F_{13}$ &$2^2 \cdot 3 \cdot 13$ & $4$ \\
     $F_{16}$ & $2^4 \cdot 3 \cdot 5$ &  $5$ \\
     $F_{17}$ & $2^4 \cdot 17$ & $5$ \\
     $F_{19}$ & $2 \cdot 3^2 \cdot 19$ & $4$
 \end{tabular}
 \end{center}

\noindent We do not undertake an investigation of a more general formula to determine the width of $F_n$, or even the width of $F_p$ for $p$ prime, in this paper. These initial examples suggest that there is likely a closed formula for $w(F_p)$, and that it is dependent on the prime factorization of the group. We leave this as an open problem for the interested reader to explore.

\section{Hasse Diagrams and Future Work} \label{sec: lattices}

As discussed in the introduction, Hasse diagrams of $G$-transfer systems can be useful tools to not only study those $G$-transfer systems, but to formulate and check conjectures more broadly. We generated the following transfer systems using the code from \cite{balchincode} and proceeded to construct the transfer system lattices with indications of saturated (blue triangle), cosaturated (pink heart), and lesser simply paired transfer systems (purple diamond). If a $G$-transfer system $T$ is connected, meaning it contains the edge $e \to G$, then $\Hull(T)$ is the complete $G$-transfer system \cite{LSP}. Therefore connected transfer systems are trivially LSP as their hull is the complete transfer system; we denote these with a hollow diamond. The transfer systems which are LSP and not connected are more interesting and are denoted with a solid diamond.

To acquaint the reader with these decorated diagrams, we begin with the familiar, and small, example of the lattice of $C_{p^2}$-transfer systems in \Cref{transfersystemCp2}. 

\newcommand{\cppa}{
	\begin{tikzpicture}[scale=0.12,baseline=0.2mm]
    \node[cyan] (S) at (3,3){$\triangle$};              
    \node[white] (L) at (-3,2){$\diamondsuit$};              
    \node[magenta] (C) at (3,.5){$\heartsuit$};
        \node[inner sep=0cm] (0) at (0,0){$\cdot$};
        \node[inner sep=0cm] (1) at (0,2){$\cdot$};
        \node[inner sep=0cm] (2) at (0,4){$\cdot$};
        \end{tikzpicture}
}

\newcommand{\cppd}{
    \begin{tikzpicture}[scale=0.12,baseline=0.2mm]
    \node[cyan] (S) at (3,3){$\triangle$};              
    \node[violet] (L) at (-3,2){$\vardiamond$};              
    \node[magenta] (C) at (3,.5){$\heartsuit$};
    \node[inner sep=0cm] (0) at (0,0){$\cdot$};
    \node[inner sep=0cm] (1) at (0,2){$\cdot$};
    \node[inner sep=0cm] (2) at (0,4){$\cdot$};
    \draw[red,-] (1) edge (2);
    \end{tikzpicture}
}

\newcommand{\cppb}{
    \begin{tikzpicture}[scale=0.12,baseline=0.2mm]
    \node[cyan] (S) at (3,3){$\triangle$};              
    \node[white] (L) at (-3,2){$\diamondsuit$};              
    \node[white] (C) at (3,.5){$\heartsuit$};
    \node[inner sep=0cm] (0) at (0,0){$\cdot$};
    \node[inner sep=0cm] (1) at (0,2){$\cdot$};
    \node[inner sep=0cm] (2) at (0,4){$\cdot$};
    \draw[red,-] (0) edge (1);
    \end{tikzpicture}
}

\newcommand{\cppc}{
    \begin{tikzpicture}[scale=0.12,baseline=0.2mm]
    \node[white] (S) at (3,3){$\triangle$};              
    \node[violet] (L) at (-3,2){$\diamondsuit$};              
    \node[magenta] (C) at (3,.5){$\heartsuit$};
    \node[inner sep=0cm] (0) at (0,0){$\cdot$};
    \node[inner sep=0cm] (1) at (0,2){$\cdot$};
    \node[inner sep=0cm] (2) at (0,4){$\cdot$};
    \draw[red,-] (0) edge (1);
    \draw[red,-] (0) edge[bend right] (2);
    \end{tikzpicture}
}

\newcommand{\cppe}{
    \begin{tikzpicture}[scale=0.12,baseline=0.2mm]
    \node[cyan] (S) at (3,3){$\triangle$};              
    \node[violet] (L) at (-3,2){$\diamondsuit$};              
    \node[magenta] (C) at (3,.5){$\heartsuit$};
    \node[inner sep=0cm] (0) at (0,0){$\cdot$};
    \node[inner sep=0cm] (1) at (0,2){$\cdot$};
    \node[inner sep=0cm] (2) at (0,4){$\cdot$};
    \draw[red,-] (0) edge (1);
    \draw[red,-] (0) edge[bend right] (2);
    \draw[red,-] (1) edge (2);
        \end{tikzpicture}
}

\begin{figure}[H]
\begin{tikzpicture}[scale=1.25]		
		\node(02) at (-1,0.5) {$e$};
		\node(12) at (-1,1.5) {$C_p$};
		\node(22) at (-1,2.5) {$C_{p^2}$};

        \node[cyan] (S) at (-1.15,5){$\triangle$ = saturated}; 
        \node[magenta] (C) at (-1.05,4.5){$\heartsuit$ = cosaturated}; 
        \node[violet] (L) at (-.15,4){$\vardiamond$ = lesser simply paired (LSP)};      
        \node[violet] (L) at (-0.4,3.5){$\diamondsuit$ = connected ($\implies$ LSP)};

		\path[-]
		(02) edge node {} (12)
		(12) edge node {} (22)
		(02) edge [bend right] node {} (22)
		;
		
		\node(A2) at (4.5,0) {$\cppa$};
		\node(B2) at (4.5,1) {$\cppb$};
		\node(C2) at (4.5,2) {$\cppc$};
		\node(D2) at (5.5,1) {$\cppd$};
		\node(E2) at (5.5,3) {$\cppe$};
		
		\path[-]
		(A2) edge node {} (B2)
		(B2) edge node {} (C2)
		(A2) edge node {} (D2)
		(C2) edge node {} (E2)
		(D2) edge node {} (E2)
		;
	\end{tikzpicture}\centering
    \caption{The Hasse diagram of transfer systems for the group $C_{p^2}$.}
    \label{transfersystemCp2}
\end{figure}
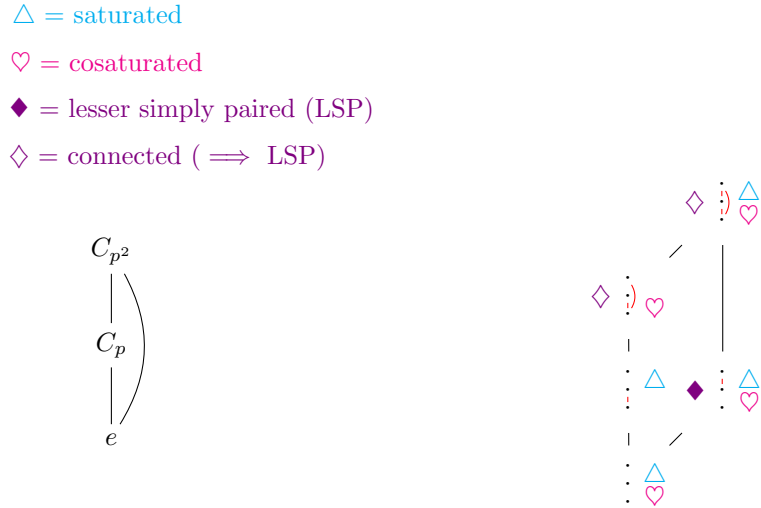

We next wish to consider a pair of transfer system lattices where the effects of the conjugacy axiom for non-abelian groups can be seen. The abelian group $C_{pq}$ and the non-abelian group $D_p$ for $p,q$ distinct primes and $p$ odd have very similar $\Sub(G)/G$ lattice structures. However, $D_p$ contains conjugate subgroups. Note that there is one fewer $D_p$-transfer system than there are $C_{pq}$-transfer systems because of this. Considering the decorations, one may note that the (co)saturation and LSP properties of the two lattices are quite similar. 


\newcommand{\cpqa}{
	\begin{tikzpicture}[scale=0.12,baseline=0.2mm]
    \node[cyan] (S) at (4,2.5){$\triangle$};              \node[white] (L) at (-4,1.5){$\diamondsuit$};              \node[magenta] (C) at (4,0){$\heartsuit$};
        \node[inner sep=0cm] (0) at (0,0){$\cdot$};
        \node[inner sep=0cm] (1) at (2,2){$\cdot$};
        \node[inner sep=0cm] (2) at (-2,2){$\cdot$};
        \node[inner sep=0cm] (3) at (0,4){$\cdot$};
        \end{tikzpicture}
}

\newcommand{\cpqb}{
    \begin{tikzpicture}[scale=0.12,baseline=0.2mm]
    \node[cyan] (S) at (4,2.5){$\triangle$};
    \node[white] (L) at (-4,1.5){$\diamondsuit$};     
    \node[white] (C) at (4,0){$\heartsuit$};
    
    \node[inner sep=0cm] (0) at (0,0){$\cdot$};
    \node[inner sep=0cm] (1) at (2,2){$\cdot$};
    \node[inner sep=0cm] (2) at (-2,2){$\cdot$};
    \node[inner sep=0cm] (3) at (0,4){$\cdot$};
    \draw[red,-] (0) edge (1);
    \end{tikzpicture}
}

\newcommand{\cpqc}{
    \begin{tikzpicture}[scale=0.12,baseline=0.2mm]
    \node[cyan] (S) at (4,2.5){$\triangle$};
    \node[white] (L) at (-4,1.5){$\diamondsuit$};     
    \node[white] (C) at (4,0){$\heartsuit$};
    \node[inner sep=0cm] (0) at (0,0){$\cdot$};
    \node[inner sep=0cm] (1) at (2,2){$\cdot$};
    \node[inner sep=0cm] (2) at (-2,2){$\cdot$};
    \node[inner sep=0cm] (3) at (0,4){$\cdot$};
    \draw[red,-] (0) edge (2);
    \end{tikzpicture}
}

\newcommand{\cpqd}{
    \begin{tikzpicture}[scale=0.12,baseline=0.2mm]
    \node[cyan] (S) at (4,2.5){$\triangle$};
    \node[white] (L) at (-4,1.5){$\diamondsuit$};     
    \node[white] (C) at (4,0){$\heartsuit$};
    \node[inner sep=0cm] (0) at (0,0){$\cdot$};
    \node[inner sep=0cm] (1) at (2,2){$\cdot$};
    \node[inner sep=0cm] (2) at (-2,2){$\cdot$};
    \node[inner sep=0cm] (3) at (0,4){$\cdot$};
    \draw[red,-] (0) edge (1);
    \draw[red,-] (0) edge (2);
    \end{tikzpicture}
}

\newcommand{\cpqe}{
    \begin{tikzpicture}[scale=0.12,baseline=0.2mm]
    \node[cyan] (S) at (4,2.5){$\triangle$};
    \node[violet] (L) at (-4,1.5){$\vardiamond$};     
    \node[magenta] (C) at (4,0){$\heartsuit$};
    \node[inner sep=0cm] (0) at (0,0){$\cdot$};
    \node[inner sep=0cm] (1) at (2,2){$\cdot$};
    \node[inner sep=0cm] (2) at (-2,2){$\cdot$};
    \node[inner sep=0cm] (3) at (0,4){$\cdot$};
    \draw[red,-] (0) edge (2);
    \draw[red,-] (1) edge (3);
    \end{tikzpicture}
}

\newcommand{\cpqf}{
    \begin{tikzpicture}[scale=0.12,baseline=0.2mm]
    \node[cyan] (S) at (4,2.5){$\triangle$};
    \node[violet] (L) at (-4,1.5){$\vardiamond$};     
    \node[magenta] (C) at (4,0){$\heartsuit$};
    \node[inner sep=0cm] (0) at (0,0){$\cdot$};
    \node[inner sep=0cm] (1) at (2,2){$\cdot$};
    \node[inner sep=0cm] (2) at (-2,2){$\cdot$};
    \node[inner sep=0cm] (3) at (0,4){$\cdot$};
    \draw[red,-] (0) edge (1);
    \draw[red,-] (2) edge (3);
    \end{tikzpicture}
}

\newcommand{\cpqg}{
    \begin{tikzpicture}[scale=0.12,baseline=0.2mm]
    \node[white] (S) at (4,2.5){$\triangle$};
    \node[violet] (L) at (-4,1.5){$\diamondsuit$};     
    \node[magenta] (C) at (4,0){$\heartsuit$};
    \node[inner sep=0cm] (0) at (0,0){$\cdot$};
    \node[inner sep=0cm] (1) at (2,2){$\cdot$};
    \node[inner sep=0cm] (2) at (-2,2){$\cdot$};
    \node[inner sep=0cm] (3) at (0,4){$\cdot$};
    \draw[red,-] (0) edge (1);
    \draw[red,-] (0) edge (2);
    \draw[red,-] (0) edge (3);
    \end{tikzpicture}
}

\newcommand{\cpqh}{
    \begin{tikzpicture}[scale=0.12,baseline=0.2mm]
    \node[white] (S) at (4,2.5){$\triangle$};
    \node[violet] (L) at (-4,1.5){$\diamondsuit$};     
    \node[magenta] (C) at (4,0){$\heartsuit$};
    \node[inner sep=0cm] (0) at (0,0){$\cdot$};
    \node[inner sep=0cm] (1) at (2,2){$\cdot$};
    \node[inner sep=0cm] (2) at (-2,2){$\cdot$};
    \node[inner sep=0cm] (3) at (0,4){$\cdot$};
    \draw[red,-] (0) edge (1);
    \draw[red,-] (0) edge (2);
    \draw[red,-] (0) edge (3);
    \draw[red,-] (2) edge (3);
    \end{tikzpicture}
}

\newcommand{\cpqi}{
    \begin{tikzpicture}[scale=0.12,baseline=0.2mm]
    \node[white] (S) at (4,2.5){$\triangle$};
    \node[violet] (L) at (-4,1.5){$\diamondsuit$};     
    \node[magenta] (C) at (4,0){$\heartsuit$};
    \node[inner sep=0cm] (0) at (0,0){$\cdot$};
    \node[inner sep=0cm] (1) at (2,2){$\cdot$};
    \node[inner sep=0cm] (2) at (-2,2){$\cdot$};
    \node[inner sep=0cm] (3) at (0,4){$\cdot$};
    \draw[red,-] (0) edge (1);
    \draw[red,-] (0) edge (2);
    \draw[red,-] (0) edge (3);
    \draw[red,-] (1) edge (3);
    \end{tikzpicture}
}

\newcommand{\cpqj}{
    \begin{tikzpicture}[scale=0.12,baseline=0.2mm]
    \node[cyan] (S) at (4,2.5){$\triangle$};
    \node[violet] (L) at (-4,1.5){$\diamondsuit$};     
    \node[magenta] (C) at (4,0){$\heartsuit$};
    \node[inner sep=0cm] (0) at (0,0){$\cdot$};
    \node[inner sep=0cm] (1) at (2,2){$\cdot$};
    \node[inner sep=0cm] (2) at (-2,2){$\cdot$};
    \node[inner sep=0cm] (3) at (0,4){$\cdot$};
    \draw[red,-] (0) edge (1);
    \draw[red,-] (0) edge (2);
    \draw[red,-] (0) edge (3);
    \draw[red,-] (1) edge (3);
    \draw[red,-] (2) edge (3);
    \end{tikzpicture}
}

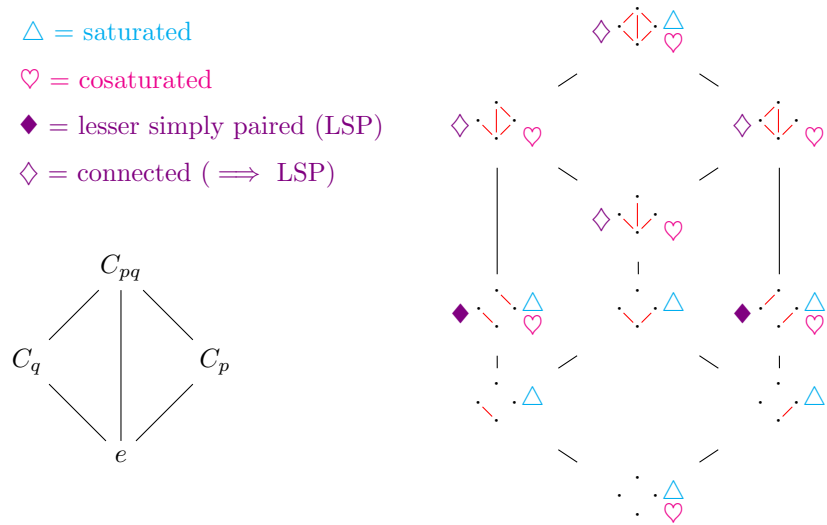
\begin{figure}[h]
\begin{tikzpicture}[scale=1.25]		
        \node[cyan] (S) at (-1.15,5){$\triangle$ = saturated}; 
        \node[magenta] (C) at (-1.05,4.5){$\heartsuit$ = cosaturated}; 
        \node[violet] (L) at (-.15,4){$\vardiamond$ = lesser simply paired (LSP)};      
        \node[violet] (L) at (-0.4,3.5){$\diamondsuit$ = connected ($\implies$ LSP)}; 

		\node(02) at (-1,0.5) {$e$};
		\node(12) at (0,1.5) {$C_p$};
        \node(11) at (-2,1.5) {$C_q$};
		\node(22) at (-1,2.5) {$C_{pq}$};
		
		\path[-]
		(02) edge node {} (12)
        (02) edge node {} (11)
		(12) edge node {} (22)
		(11) edge node {} (22)
        (02) edge node {} (22)
		;
		
		\node(A2) at (4.5,0) {$\cpqa$};
		\node(B2) at (6,1) {$\cpqb$};
		\node(C2) at (4.5,2) {$\cpqd$};
		\node(D2) at (3,1) {$\cpqc$};
		\node(E2) at (3,2) {$\cpqe$};
        \node(F2) at (6,2) {$\cpqf$};
        \node(G2) at (4.5,3) {$\cpqg$};
        \node(H2) at (6,4) {$\cpqh$};
        \node(I2) at (3,4) {$\cpqi$};
        \node(J2) at (4.5,5) {$\cpqj$};
		
		\path[-]
		(A2) edge node {} (B2)
		(B2) edge node {} (C2)
		(A2) edge node {} (D2)
		(C2) edge node {} (G2)
		(D2) edge node {} (E2)
        (D2) edge node {} (C2)
        (B2) edge node {} (F2)
        (I2) edge node {} (J2)
        (H2) edge node {} (J2)
        (E2) edge node {} (I2)
        (F2) edge node {} (H2)
        (G2) edge node {} (I2)
        (G2) edge node {} (H2)
		;
	\end{tikzpicture}
    \caption{The Hasse diagram of transfer systems for the group $C_{pq}$, for $p$ and $q$ distinct primes.}
    \label{fig: Hasse diagram Cpq}
\end{figure}


\newcommand{\ddpa}{
	\begin{tikzpicture}[scale=0.12,baseline=0.2mm]     
    \node[cyan] (S) at (4,3){$\triangle$};                   
    \node[white] (L) at (-4,2){$\diamondsuit$};                   
    \node[magenta] (C) at (4,.5){$\heartsuit$};
        \node[inner sep=0cm] (0) at (0,0){$\cdot$};
        \node[inner sep=0cm] (1) at (2,2){$\cdot$};
        \node[inner sep=0cm] (2) at (-2,2){$\cdot$};
        \node[inner sep=0cm] (3) at (0,4){$\cdot$};
        \end{tikzpicture}
}

\newcommand{\ddpb}{
	\begin{tikzpicture}[scale=0.12,baseline=0.2mm]     \node[cyan] (S) at (4,3){$\triangle$};                   \node[white] (L) at (-4,2){$\diamondsuit$};                   \node[white] (C) at (4,.5){$\heartsuit$};
        \node[inner sep=0cm] (0) at (0,0){$\cdot$};
        \node[inner sep=0cm] (1) at (2,2){$\cdot$};
        \node[inner sep=0cm] (2) at (-2,2){$\cdot$};
        \node[inner sep=0cm] (3) at (0,4){$\cdot$};
        \draw[red,-] (0) edge (2);
        \end{tikzpicture}
}

\newcommand{\ddpc}{
	\begin{tikzpicture}[scale=0.12,baseline=0.2mm]     \node[cyan] (S) at (4,3){$\triangle$};                   \node[white] (L) at (-4,2){$\diamondsuit$};                   \node[white] (C) at (4,.5){$\heartsuit$};
        \node[inner sep=0cm] (0) at (0,0){$\cdot$};
        \node[inner sep=0cm] (1) at (2,2){$\cdot$};
        \node[inner sep=0cm] (2) at (-2,2){$\cdot$};
        \node[inner sep=0cm] (3) at (0,4){$\cdot$};
        \draw[red,-] (0) edge (1);
        \end{tikzpicture}
}

\newcommand{\ddpd}{
	\begin{tikzpicture}[scale=0.12,baseline=0.2mm]     \node[cyan] (S) at (4,3){$\triangle$};                   \node[white] (L) at (-4,2){$\diamondsuit$};                   \node[white] (C) at (4,.5){$\heartsuit$};
        \node[inner sep=0cm] (0) at (0,0){$\cdot$};
        \node[inner sep=0cm] (1) at (2,2){$\cdot$};
        \node[inner sep=0cm] (2) at (-2,2){$\cdot$};
        \node[inner sep=0cm] (3) at (0,4){$\cdot$};
        \draw[red,-] (0) edge (1);
        \draw[red,-] (0) edge (2);
        \end{tikzpicture}
}

\newcommand{\ddpe}{
	\begin{tikzpicture}[scale=0.12,baseline=0.2mm]     \node[cyan] (S) at (4,3){$\triangle$};                   \node[violet] (L) at (-4,2){$\vardiamond$};                   \node[magenta] (C) at (4,.5){$\heartsuit$};
        \node[inner sep=0cm] (0) at (0,0){$\cdot$};
        \node[inner sep=0cm] (1) at (2,2){$\cdot$};
        \node[inner sep=0cm] (2) at (-2,2){$\cdot$};
        \node[inner sep=0cm] (3) at (0,4){$\cdot$};
        \draw[red,-] (0) edge (1);
        \draw[red,-] (2) edge (3);
        \end{tikzpicture}
}

\newcommand{\ddpf}{
	\begin{tikzpicture}[scale=0.12,baseline=0.2mm]     \node[white] (S) at (4,3){$\triangle$};                   \node[violet] (L) at (-4,2){$\diamondsuit$};                   \node[magenta] (C) at (4,.5){$\heartsuit$};
        \node[inner sep=0cm] (0) at (0,0){$\cdot$};
        \node[inner sep=0cm] (1) at (2,2){$\cdot$};
        \node[inner sep=0cm] (2) at (-2,2){$\cdot$};
        \node[inner sep=0cm] (3) at (0,4){$\cdot$};
        \draw[red,-] (0) edge (1);
        \draw[red,-] (0) edge (2);
        \draw[red,-] (0) edge (3);
        \end{tikzpicture}
}

\newcommand{\ddpg}{
	\begin{tikzpicture}[scale=0.12,baseline=0.2mm]     \node[white] (S) at (4,3){$\triangle$};                   \node[violet] (L) at (-4,2){$\diamondsuit$};                   \node[magenta] (C) at (4,.5){$\heartsuit$};
        \node[inner sep=0cm] (0) at (0,0){$\cdot$};
        \node[inner sep=0cm] (1) at (2,2){$\cdot$};
        \node[inner sep=0cm] (2) at (-2,2){$\cdot$};
        \node[inner sep=0cm] (3) at (0,4){$\cdot$};
        \draw[red,-] (0) edge (1);
        \draw[red,-] (0) edge (2);
        \draw[red,-] (0) edge (3);
        \draw[red,-] (2) edge (3);
        \end{tikzpicture}
}

\newcommand{\ddph}{
	\begin{tikzpicture}[scale=0.12,baseline=0.2mm]     \node[white] (S) at (4,3){$\triangle$};                   \node[violet] (L) at (-4,2){$\diamondsuit$};                   \node[magenta] (C) at (4,.5){$\heartsuit$};
        \node[inner sep=0cm] (0) at (0,0){$\cdot$};
        \node[inner sep=0cm] (1) at (2,2){$\cdot$};
        \node[inner sep=0cm] (2) at (-2,2){$\cdot$};
        \node[inner sep=0cm] (3) at (0,4){$\cdot$};
        \draw[red,-] (0) edge (1);
        \draw[red,-] (0) edge (2);
        \draw[red,-] (0) edge (3);
        \draw[red,-] (1) edge (3);
        \end{tikzpicture}
}

\newcommand{\ddpi}{
	\begin{tikzpicture}[scale=0.12,baseline=0.2mm]     
        \node[cyan] (S) at (4,3){$\triangle$};                   
        \node[violet] (L) at (-4,2){$\diamondsuit$};                   \node[magenta] (C) at (4,.5){$\heartsuit$};
        \node[inner sep=0cm] (0) at (0,0){$\cdot$};
        \node[inner sep=0cm] (1) at (2,2){$\cdot$};
        \node[inner sep=0cm] (2) at (-2,2){$\cdot$};
        \node[inner sep=0cm] (3) at (0,4){$\cdot$};
        \draw[red,-] (0) edge (1);
        \draw[red,-] (0) edge (2);
        \draw[red,-] (0) edge (3);
        \draw[red,-] (1) edge (3);
        \draw[red,-] (2) edge (3);
        \end{tikzpicture}
}

\begin{figure}[H]
\begin{tikzpicture}[scale=1.25]		
		\node(00) at (-1,.5) {$e$};
		\node(11) at (-2,1.5) {$C_p$};
            \node(12) at (0,1.5) {${}_p C_2$};
		\node(22) at (-1,2.5) {$D_{p}$};

        \node[cyan] (S) at (-1.15,5){$\triangle$ = saturated}; 
        \node[magenta] (C) at (-1.05,4.5){$\heartsuit$ = cosaturated}; 
        \node[violet] (L) at (-.15,4){$\vardiamond$ = lesser simply paired (LSP)};      
        \node[violet] (L) at (-0.4,3.5){$\diamondsuit$ = connected ($\implies$ LSP)};

		\path[-]
		(00) edge node {} (11)
            (00) edge node {} (12)
            (00) edge node {} (22)
		(12) edge node {} (22)
		(11) edge node {} (22)
		;
		
		\node(A1) at (4.5,0) {$\ddpa$};
		\node(B1) at (3,1) {$\ddpb$};
		\node(C1) at (6,1) {$\ddpc$};
		\node(D1) at (4.5,2) {$\ddpd$};
		\node(E1) at (6,2) {$\ddpe$};
        
            \node(F1) at (4.5,3) {$\ddpf$};
            \node(G1) at (6,4) {$\ddpg$};
		\node(H1) at (3,4) {$\ddph$};
            \node(I1) at (4.5,5) {$\ddpi$};
		
		\path[-]
		(A1) edge node {} (B1)
            (A1) edge node {} (C1)
		(B1) edge node {} (D1)
		(C1) edge node {} (D1)
		(C1) edge node {} (E1)
		(D1) edge node {} (F1)
            (E1) edge node {} (G1)
            (F1) edge node {} (H1)
            (F1) edge node {} (G1)
            (H1) edge node {} (I1)
            (G1) edge node {} (I1)
            
		;
	\end{tikzpicture}
    \caption{The Hasse diagram of transfer systems for the group $D_p$, for $p$ an odd prime.}
    \label{fig: Hasse diagram Dp}
    \label{transfersystemDp}
\end{figure}
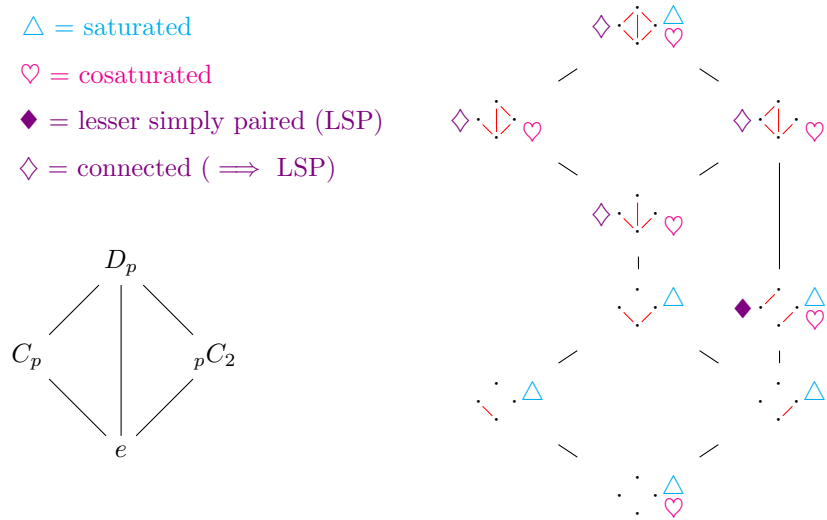

The remaining lattice examples in this section are significantly larger and more complex than the previous ones. 

In \Cref{fig: a4 lattice} we included a transfer system for the alternating group $A_4$ which has a non-modular $\Sub(G)/G$ lattice. \Cref{transfersystemA4} shows the entire decorated Hasse diagram of transfer systems for $A_4$.

The final three Hasse diagrams each depict a non-abelian group whose reduced subgroup lattices $\Sub(G)/G$ appear to have the same structure. However, there are vast differences in their Hasse diagrams, which demonstrates the impact of their different conjugacy classes. \Cref{transfersystemDp2} is a dihedral group with 56 transfer systems. \Cref{transfersystemDicp} is a dicyclic group with 62 transfer systems. \Cref{transfersystemF5} is a Frobenius group with 59 transfer systems. 

The creation and labeling of these Hasse diagrams spurred many interesting discussions about the structures in these lattices. Visualizing the transfer systems in this way led us to several conjectures about properties of non-abelian transfer systems and their Hasse diagrams. We intend to explore these conjectures in future work but provide two examples here with the hope that interested readers are inspired to generate their own hypotheses about the patterns appearing in these lattices. 

First, we discuss a conjecture whose development was aided by decorating the (co)saturation properties of the lattices. For the following conjecture, we define a path's \emph{length} to be the number of edges that the path contains, and recall that Hasse diagrams do not contain composition edges. A \emph{shortest path} between two transfer systems in the lattice is any path containing the minimal number of edges; note that shortest paths are not necessarily unique. 

\begin{conjecture}
    A shortest path between the trivial $G$-transfer system and the complete $G$-transfer system on a Hasse diagram will pass through as many bisaturated transfer systems as possible. 
\end{conjecture}

One can see the Hasse diagrams in this section support these conjectures. For example, the shortest path from the trivial transfer system to the complete transfer system in \Cref{transfersystemCp2} is of length two. Both of the shortest paths from the trivial transfer system to the complete transfer system in \Cref{fig: Hasse diagram Cpq} are of length four. This demonstrates that the shortest path is not always unique and does not always pass through all bisaturated transfer systems. The reader may verify that shortest paths from the trivial transfer system to the complete transfer system for $A_4$, $D_{p^2}, \Dic_p,$ and $F_5$ are six, seven, seven, and eight, respectively, using the Hasse diagrams that follow.

We also have conjectures and progress toward proofs for when a $G$-transfer system is LSP based on the connected component containing the vertex $G$.

\begin{conjecture}
    Let $T$ be a $G$-transfer system with two connected components, one containing $e$ and one containing $G$. If the connected component containing $G$ only has the edge $G \to G$, then $T$ is not LSP and is compatible with the transfer system generated by adding the edge $e \to G$ to $T$.
\end{conjecture}


\newcommand{\afoura}{
	\begin{tikzpicture}[scale=0.12,baseline=0.2mm]
        \node[cyan] (S) at (4,3){$\triangle$};                   
        \node[white] (L) at (-4,2){$\diamondsuit$};                   \node[magenta] (C) at (4,.5){$\heartsuit$};
        \node[inner sep=0cm] (0) at (0,0){$\cdot$};
        \node[inner sep=0cm] (1) at (1.5,1){$\cdot$};
        \node[inner sep=0cm] (2) at (-1.5, 1){$\cdot$};
        \node[inner sep=0cm] (3) at (1.5,2){$\cdot$};
        \node[inner sep=0cm] (4) at (0, 3){$\cdot$};
        \end{tikzpicture}
}

\newcommand{\afourb}{
	\begin{tikzpicture}[scale=0.12,baseline=0.2mm]
        \node[cyan] (S) at (4,3){$\triangle$};                   
        \node[white] (L) at (-4,2){$\diamondsuit$};                   \node[white] (C) at (4,.5){$\heartsuit$};
        \node[inner sep=0cm] (0) at (0,0){$\cdot$};
        \node[inner sep=0cm] (1) at (1.5,1){$\cdot$};
        \node[inner sep=0cm] (2) at (-1.5, 1){$\cdot$};
        \node[inner sep=0cm] (3) at (1.5,2){$\cdot$};
        \node[inner sep=0cm] (4) at (0, 3){$\cdot$};
        \draw[red,-] (0) edge (1);
    \end{tikzpicture}
}

\newcommand{\afourc}{
	\begin{tikzpicture}[scale=0.12,baseline=0.2mm]
        
        \node[inner sep=0cm] (0) at (0,0){$\cdot$};
        \node[inner sep=0cm] (1) at (1.5,1){$\cdot$};
        \node[inner sep=0cm] (2) at (-1.5, 1){$\cdot$};
        \node[inner sep=0cm] (3) at (1.5,2){$\cdot$};
        \node[inner sep=0cm] (4) at (0, 3){$\cdot$};
        \draw[red,-] (0) edge (1);
        \draw[red,-] (0) edge (3);
    \end{tikzpicture}
}

\newcommand{\afourd}{
	\begin{tikzpicture}[scale=0.12,baseline=0.2mm]
        \node[cyan] (S) at (4,3){$\triangle$};                   
        \node[white] (L) at (-4,2){$\diamondsuit$};                   \node[white] (C) at (4,.5){$\heartsuit$};
        \node[inner sep=0cm] (0) at (0,0){$\cdot$};
        \node[inner sep=0cm] (1) at (1.5,1){$\cdot$};
        \node[inner sep=0cm] (2) at (-1.5, 1){$\cdot$};
        \node[inner sep=0cm] (3) at (1.5,2){$\cdot$};
        \node[inner sep=0cm] (4) at (0, 3){$\cdot$};
        \draw[red,-] (0) edge (2);
    \end{tikzpicture}
}

\newcommand{\afoure}{
	\begin{tikzpicture}[scale=0.12,baseline=0.2mm]
        \node[cyan] (S) at (4,3){$\triangle$};                   
        \node[white] (L) at (-4,2){$\diamondsuit$};                   \node[magenta] (C) at (4,.5){$\heartsuit$};
        \node[inner sep=0cm] (0) at (0,0){$\cdot$};
        \node[inner sep=0cm] (1) at (1.5,1){$\cdot$};
        \node[inner sep=0cm] (2) at (-1.5, 1){$\cdot$};
        \node[inner sep=0cm] (3) at (1.5,2){$\cdot$};
        \node[inner sep=0cm] (4) at (0, 3){$\cdot$};
        \draw[red,-] (0) edge (2);
        \draw[red,-] (3) edge (4);
    \end{tikzpicture}
}

\newcommand{\afourf}{
	\begin{tikzpicture}[scale=0.12,baseline=0.2mm]
        \node[cyan] (S) at (4,3){$\triangle$};                   
        \node[white] (L) at (-4,2){$\diamondsuit$};                   \node[white] (C) at (4,.5){$\heartsuit$};
        \node[inner sep=0cm] (0) at (0,0){$\cdot$};
        \node[inner sep=0cm] (1) at (1.5,1){$\cdot$};
        \node[inner sep=0cm] (2) at (-1.5, 1){$\cdot$};
        \node[inner sep=0cm] (3) at (1.5,2){$\cdot$};
        \node[inner sep=0cm] (4) at (0, 3){$\cdot$};
        \draw[red,-] (0) edge (1);
        \draw[red,-] (0) edge (2);
    \end{tikzpicture}
}

\newcommand{\afourg}{
	\begin{tikzpicture}[scale=0.12,baseline=0.2mm]
        \node[cyan] (S) at (4,3){$\triangle$};                   
        \node[white] (L) at (-4,2){$\diamondsuit$};                   \node[white] (C) at (4,.5){$\heartsuit$};
        \node[inner sep=0cm] (0) at (0,0){$\cdot$};
        \node[inner sep=0cm] (1) at (1.5,1){$\cdot$};
        \node[inner sep=0cm] (2) at (-1.5, 1){$\cdot$};
        \node[inner sep=0cm] (3) at (1.5,2){$\cdot$};
        \node[inner sep=0cm] (4) at (0, 3){$\cdot$};
        \draw[red,-] (0) edge (1);
        \draw[red,-] (0) edge (3);
        \draw[red,-] (1) edge (3);
    \end{tikzpicture}
}

\newcommand{\afourh}{
	\begin{tikzpicture}[scale=0.12,baseline=0.2mm]
        \node[cyan] (S) at (4,3){$\triangle$};                   
        \node[white] (L) at (-4,2){$\diamondsuit$};                   \node[white] (C) at (4,.5){$\heartsuit$};
        \node[inner sep=0cm] (0) at (0,0){$\cdot$};
        \node[inner sep=0cm] (1) at (1.5,1){$\cdot$};
        \node[inner sep=0cm] (2) at (-1.5, 1){$\cdot$};
        \node[inner sep=0cm] (3) at (1.5,2){$\cdot$};
        \node[inner sep=0cm] (4) at (0, 3){$\cdot$};
        \draw[red,-] (0) edge (1);
        \draw[red,-] (0) edge (2);
        \draw[red,-] (3) edge (4);
    \end{tikzpicture}
}

\newcommand{\afouri}{
	\begin{tikzpicture}[scale=0.12,baseline=0.2mm]
        \node[inner sep=0cm] (0) at (0,0){$\cdot$};
        \node[inner sep=0cm] (1) at (1.5,1){$\cdot$};
        \node[inner sep=0cm] (2) at (-1.5, 1){$\cdot$};
        \node[inner sep=0cm] (3) at (1.5,2){$\cdot$};
        \node[inner sep=0cm] (4) at (0, 3){$\cdot$};
        \draw[red,-] (0) edge (1);
        \draw[red,-] (0) edge (2);
        \draw[red,-] (0) edge (3);
    \end{tikzpicture}
}

\newcommand{\afourj}{
	\begin{tikzpicture}[scale=0.12,baseline=0.2mm]
        \node[white] (S) at (4,3){$\triangle$};                   
        \node[violet] (L) at (-4,2){$\diamondsuit$};                   \node[magenta] (C) at (4,.5){$\heartsuit$};
        \node[inner sep=0cm] (0) at (0,0){$\cdot$};
        \node[inner sep=0cm] (1) at (1.5,1){$\cdot$};
        \node[inner sep=0cm] (2) at (-1.5, 1){$\cdot$};
        \node[inner sep=0cm] (3) at (1.5,2){$\cdot$};
        \node[inner sep=0cm] (4) at (0, 3){$\cdot$};
        \draw[red,-] (0) edge (1);
        \draw[red,-] (0) edge (2);
        \draw[red,-] (0) edge (3);
        \draw[red,-] (0) edge (4);
    \end{tikzpicture}
}

\newcommand{\afourk}{
	\begin{tikzpicture}[scale=0.12,baseline=0.2mm]
        \node[white] (S) at (4,3){$\triangle$};                   
        \node[violet] (L) at (-4,2){$\diamondsuit$};                   \node[magenta] (C) at (4,.5){$\heartsuit$};
        \node[inner sep=0cm] (0) at (0,0){$\cdot$};
        \node[inner sep=0cm] (1) at (1.5,1){$\cdot$};
        \node[inner sep=0cm] (2) at (-1.5, 1){$\cdot$};
        \node[inner sep=0cm] (3) at (1.5,2){$\cdot$};
        \node[inner sep=0cm] (4) at (0, 3){$\cdot$};
        \draw[red,-] (0) edge (1);
        \draw[red,-] (0) edge (2);
        \draw[red,-] (0) edge (3);
        \draw[red,-] (0) edge (4);
        \draw[red,-] (3) edge (4);
    \end{tikzpicture}
}

\newcommand{\afourl}{
	\begin{tikzpicture}[scale=0.12,baseline=0.2mm]
        \node[cyan] (S) at (4,3){$\triangle$};                   
        \node[white] (L) at (-4,2){$\diamondsuit$};                   \node[white] (C) at (4,.5){$\heartsuit$};
        \node[inner sep=0cm] (0) at (0,0){$\cdot$};
        \node[inner sep=0cm] (1) at (1.5,1){$\cdot$};
        \node[inner sep=0cm] (2) at (-1.5, 1){$\cdot$};
        \node[inner sep=0cm] (3) at (1.5,2){$\cdot$};
        \node[inner sep=0cm] (4) at (0, 3){$\cdot$};
        \draw[red,-] (0) edge (1);
        \draw[red,-] (0) edge (2);
        \draw[red,-] (0) edge (3);
        \draw[red,-] (1) edge (3);
    \end{tikzpicture}
}

\newcommand{\afourm}{
	\begin{tikzpicture}[scale=0.12,baseline=0.2mm]
        \node[white] (S) at (4,3){$\triangle$};                   
        \node[violet] (L) at (-4,2){$\diamondsuit$};                   \node[white] (C) at (4,.5){$\heartsuit$};
        \node[inner sep=0cm] (0) at (0,0){$\cdot$};
        \node[inner sep=0cm] (1) at (1.5,1){$\cdot$};
        \node[inner sep=0cm] (2) at (-1.5, 1){$\cdot$};
        \node[inner sep=0cm] (3) at (1.5,2){$\cdot$};
        \node[inner sep=0cm] (4) at (0, 3){$\cdot$};
        \draw[red,-] (0) edge (1);
        \draw[red,-] (0) edge (2);
        \draw[red,-] (0) edge (3);
        \draw[red,-] (0) edge (4);
        \draw[red,-] (1) edge (3);
    \end{tikzpicture}
}

\newcommand{\afourn}{
	\begin{tikzpicture}[scale=0.12,baseline=0.2mm]
        \node[white] (S) at (4,3){$\triangle$};                   
        \node[violet] (L) at (-4,2){$\diamondsuit$};                   \node[magenta] (C) at (4,.5){$\heartsuit$};
        \node[inner sep=0cm] (0) at (0,0){$\cdot$};
        \node[inner sep=0cm] (1) at (1.5,1){$\cdot$};
        \node[inner sep=0cm] (2) at (-1.5, 1){$\cdot$};
        \node[inner sep=0cm] (3) at (1.5,2){$\cdot$};
        \node[inner sep=0cm] (4) at (0, 3){$\cdot$};
        \draw[red,-] (0) edge (1);
        \draw[red,-] (0) edge (2);
        \draw[red,-] (0) edge (3);
        \draw[red,-] (0) edge (4);
        \draw[red,-] (2) edge (4);
    \end{tikzpicture}
}

\newcommand{\afouro}{
	\begin{tikzpicture}[scale=0.12,baseline=0.2mm]
        \node[white] (S) at (4,3){$\triangle$};                   
        \node[violet] (L) at (-4,2){$\diamondsuit$};                   \node[magenta] (C) at (4,.5){$\heartsuit$};
        \node[inner sep=0cm] (0) at (0,0){$\cdot$};
        \node[inner sep=0cm] (1) at (1.5,1){$\cdot$};
        \node[inner sep=0cm] (2) at (-1.5, 1){$\cdot$};
        \node[inner sep=0cm] (3) at (1.5,2){$\cdot$};
        \node[inner sep=0cm] (4) at (0, 3){$\cdot$};
        \draw[red,-] (0) edge (1);
        \draw[red,-] (0) edge (2);
        \draw[red,-] (0) edge (3);
        \draw[red,-] (0) edge (4);
        \draw[red,-] (2) edge (4);
        \draw[red,-] (3) edge (4);
    \end{tikzpicture}
}

\newcommand{\afourp}{
	\begin{tikzpicture}[scale=0.12,baseline=0.2mm]
        \node[white] (S) at (4,3){$\triangle$};                   
        \node[violet] (L) at (-4,2){$\diamondsuit$};                   \node[magenta] (C) at (4,.5){$\heartsuit$};
        \node[inner sep=0cm] (0) at (0,0){$\cdot$};
        \node[inner sep=0cm] (1) at (1.5,1){$\cdot$};
        \node[inner sep=0cm] (2) at (-1.5, 1){$\cdot$};
        \node[inner sep=0cm] (3) at (1.5,2){$\cdot$};
        \node[inner sep=0cm] (4) at (0, 3){$\cdot$};
        \draw[red,-] (0) edge (1);
        \draw[red,-] (0) edge (2);
        \draw[red,-] (0) edge (3);
        \draw[red,-] (0) edge (4);
        \draw[red,-] (1) edge (3);
        \draw[red,-] (1) edge (4);
    \end{tikzpicture}
}

\newcommand{\afourq}{
	\begin{tikzpicture}[scale=0.12,baseline=0.2mm]
        \node[white] (S) at (4,3){$\triangle$};                   
        \node[violet] (L) at (-4,2){$\diamondsuit$};                   \node[white] (C) at (4,.5){$\heartsuit$};
        \node[inner sep=0cm] (0) at (0,0){$\cdot$};
        \node[inner sep=0cm] (1) at (1.5,1){$\cdot$};
        \node[inner sep=0cm] (2) at (-1.5, 1){$\cdot$};
        \node[inner sep=0cm] (3) at (1.5,2){$\cdot$};
        \node[inner sep=0cm] (4) at (0, 3){$\cdot$};
        \draw[red,-] (0) edge (1);
        \draw[red,-] (0) edge (2);
        \draw[red,-] (0) edge (3);
        \draw[red,-] (0) edge (4);
        \draw[red,-] (1) edge (3);
        \draw[red,-] (2) edge (4);
    \end{tikzpicture}
}

\newcommand{\afourr}{
	\begin{tikzpicture}[scale=0.12,baseline=0.2mm]
        \node[white] (S) at (4,3){$\triangle$};                   
        \node[violet] (L) at (-4,2){$\diamondsuit$};                   \node[magenta] (C) at (4,.5){$\heartsuit$};
        \node[inner sep=0cm] (0) at (0,0){$\cdot$};
        \node[inner sep=0cm] (1) at (1.5,1){$\cdot$};
        \node[inner sep=0cm] (2) at (-1.5, 1){$\cdot$};
        \node[inner sep=0cm] (3) at (1.5,2){$\cdot$};
        \node[inner sep=0cm] (4) at (0, 3){$\cdot$};
        \draw[red,-] (0) edge (1);
        \draw[red,-] (0) edge (2);
        \draw[red,-] (0) edge (3);
        \draw[red,-] (0) edge (4);
        \draw[red,-] (1) edge (3);
        \draw[red,-] (1) edge (4);
        \draw[red,-] (3) edge (4);
    \end{tikzpicture}
}

\newcommand{\afours}{
	\begin{tikzpicture}[scale=0.12,baseline=0.2mm]
        \node[white] (S) at (4,3){$\triangle$};                   
        \node[violet] (L) at (-4,2){$\diamondsuit$};                   \node[magenta] (C) at (4,.5){$\heartsuit$};
        \node[inner sep=0cm] (0) at (0,0){$\cdot$};
        \node[inner sep=0cm] (1) at (1.5,1){$\cdot$};
        \node[inner sep=0cm] (2) at (-1.5, 1){$\cdot$};
        \node[inner sep=0cm] (3) at (1.5,2){$\cdot$};
        \node[inner sep=0cm] (4) at (0, 3){$\cdot$};
        \draw[red,-] (0) edge (1);
        \draw[red,-] (0) edge (2);
        \draw[red,-] (0) edge (3);
        \draw[red,-] (0) edge (4);
        \draw[red,-] (1) edge (3);
        \draw[red,-] (1) edge (4);
        \draw[red,-] (2) edge (4);
    \end{tikzpicture}
}

\newcommand{\afourt}{
	\begin{tikzpicture}[scale=0.12,baseline=0.2mm]
        \node[cyan] (S) at (4,3){$\triangle$};                   
        \node[violet] (L) at (-4,2){$\diamondsuit$};                   \node[magenta] (C) at (4,.5){$\heartsuit$};
        \node[inner sep=0cm] (0) at (0,0){$\cdot$};
        \node[inner sep=0cm] (1) at (1.5,1){$\cdot$};
        \node[inner sep=0cm] (2) at (-1.5, 1){$\cdot$};
        \node[inner sep=0cm] (3) at (1.5,2){$\cdot$};
        \node[inner sep=0cm] (4) at (0, 3){$\cdot$};
        \draw[red,-] (0) edge (1);
        \draw[red,-] (0) edge (2);
        \draw[red,-] (0) edge (3);
        \draw[red,-] (0) edge (4);
        \draw[red,-] (1) edge (3);
        \draw[red,-] (1) edge (4);
        \draw[red,-] (2) edge (4);
        \draw[red,-] (3) edge (4);
    \end{tikzpicture}
}

\begin{figure}[H]
	\caption{The Hasse diagram of transfer systems for the group $A_{4}$.}
	\label{transfersystemA4}	
	\[
	\begin{array}{c}
	\begin{tikzpicture}[scale=1.1]
        \node[cyan] (S) at (-6.15,12){$\triangle$ = saturated}; 
        \node[magenta] (C) at (-6.05,11.5){$\heartsuit$ = cosaturated}; 
        \node[violet] (L) at (-5,11){$\vardiamond$ = lesser simply paired (LSP)};      
        \node[violet] (L) at (-5.3,10.5){$\diamondsuit$ = connected ($\implies$ LSP)};    

        \node[inner sep=0cm] (60) at (-5.5,7){$e$};
        \node[inner sep=0cm] (61) at (-4,8){${}_3 C_2$};
        \node[inner sep=0cm] (62) at (-7, 8){${}_4 C_3$};
        \node[inner sep=0cm] (63) at (-4,9){$C_2^2$};
        \node[inner sep=0cm] (64) at (-5.5, 10){$A_4$};
        \draw[black,-] (60) edge (61);
        \draw[black,-] (60) edge (62);
        \draw[black,-] (60) edge (63);
        \draw[black,-] (60) edge (64);
        \draw[black,-] (61) edge (63);
        \draw[black,-] (61) edge (64);
        \draw[black,-] (62) edge (64);
        \draw[black,-] (63) edge (64);

		\node(00) at (0,0) {$\afoura$};
		\node(-11) at (-2,1.5) {$\afourd$};
		\node(11) at (2,1.5) {$\afourb$};
		\node(-12) at (-2,3) {$\afoure$};
		\node(02) at (0,3) {$\afourf$};
		\node(12) at (2,3) {$\afourc$};
		\node(-13) at (-2,4.5) {$\afourh$};
		\node(03) at (0,4.5) {$\afouri$};
		\node(13) at (2,4.5) {$\afourg$};
		\node(-14) at (-2,7.5) {$\afourk$};
		\node(04) at (0,6) {$\afourj$};
		\node(14) at (2,6) {$\afourl$};
		\node(-15) at (-2,9) {$\afouro$};
		\node(05) at (0,7.5) {$\afourn$};
		\node(15) at (2,7.5) {$\afourm$};
		\node(06) at (0,9) {$\afourq$};
            \node(16) at (2,9) {$\afourp$};
            \node(07) at (0,10.5) {$\afourr$};
            \node(17) at (2,10.5) {$\afours$};
            \node(08) at (0,12) {$\afourt$};
		
		\path[-]
		(00) edge node {} (-11)
		(00) edge node {} (11)
		(-11) edge node {} (-12)
		(-11) edge node {} (02)
		(03) edge node {} (14)
		(11) edge node {} (02)
		(11) edge node {} (12)
		(02) edge node {} (-13)
		(12) edge node {} (03)
		(-12) edge node {} (-13)
		(02) edge node {} (03)
		(12) edge node {} (13)
		(-13) edge node {} (-14)
		(03) edge node {} (04)
		(13) edge node {} (14)
		(-14) edge node {} (-15)
		(04) edge node {} (-14)
		(04) edge node {} (15)
		(04) edge node {} (05)
		(14) edge node {} (15)
		(-15) edge node {} (08)
		(05) edge node {} (06)
		(15) edge node {} (06)
            (15) edge node {} (16)
            (16) edge node {} (07)
            (16) edge node {} (17)
            (06) edge node {} (17)
            (-14) edge node {} (07)
            (07) edge node {} (08)
            (17) edge node {} (08)
            (05) edge node {} (-15)
		;
	\end{tikzpicture}
	\end{array}
	\]
\end{figure}
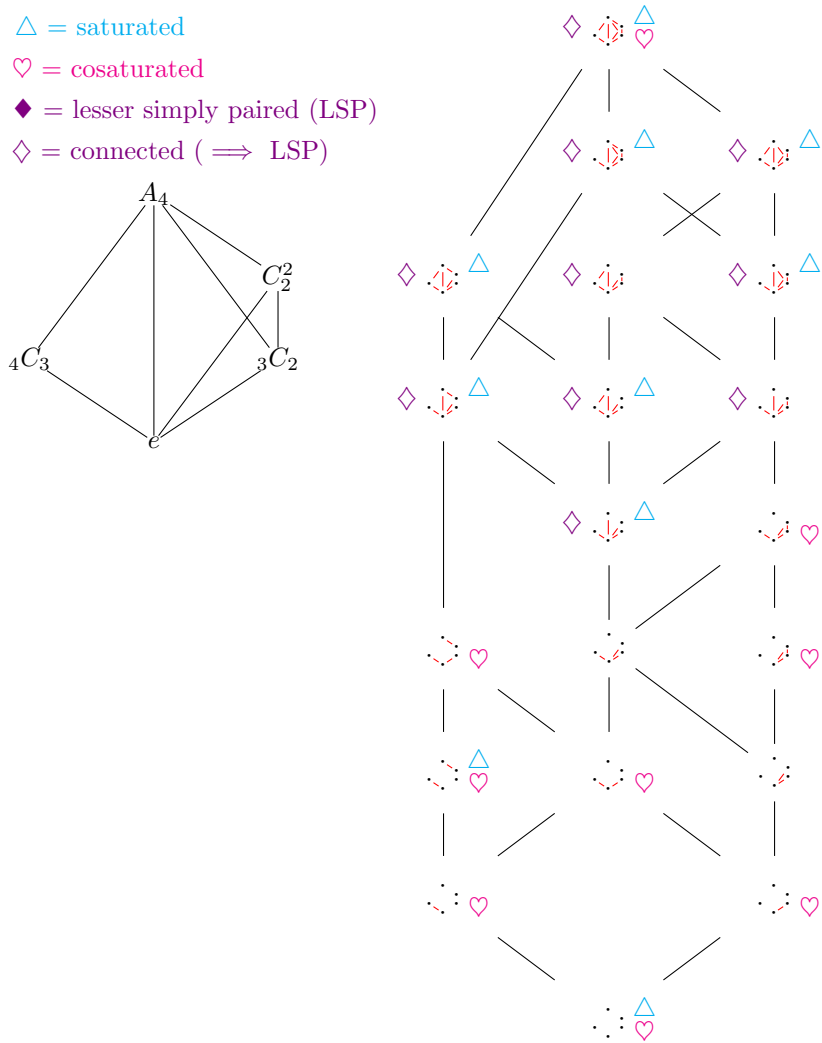


\newcommand{\dpa}{
    \begin{tikzpicture}[scale=0.12,baseline=0.2mm]	\node[cyan] (S) at (4,2.5){$\triangle$};         \node[white] (L) at (-4,1.5){$\diamondsuit$};         \node[magenta] (C) at (4,0){$\heartsuit$};
	\node[inner sep=0cm] (0) at (0,0){$\cdot$};
        \node[inner sep=0cm] (1) at (1.5,1){$\cdot$};
        \node[inner sep=0cm] (2) at (-1.5,1){$\cdot$};
        \node[inner sep=0cm] (3) at (1.5,2){$\cdot$};
        \node[inner sep=0cm] (4) at (-1.5,2){$\cdot$};
        \node[inner sep=0cm] (5) at (0,3){$\cdot$};
    \end{tikzpicture}
}

\newcommand{\dpb}{
    \begin{tikzpicture}[scale=0.12,baseline=0.2mm]	
    \node[cyan] (S) at (4,2.5){$\triangle$};  
    \node[white] (F) at (-4,2.5){$\triangle$}; 
	\node[inner sep=0cm] (0) at (0,0){$\cdot$};
        \node[inner sep=0cm] (1) at (1.5,1){$\cdot$};
        \node[inner sep=0cm] (2) at (-1.5,1){$\cdot$};
        \node[inner sep=0cm] (3) at (1.5,2){$\cdot$};
        \node[inner sep=0cm] (4) at (-1.5,2){$\cdot$};
        \node[inner sep=0cm] (5) at (0,3){$\cdot$};
        \draw[red,-] (0) edge (2);
    \end{tikzpicture}
}

\newcommand{\dpc}{
    \begin{tikzpicture}[scale=0.12,baseline=0.2mm]	
    \node[cyan] (S) at (4,2.5){$\triangle$};  
    \node[white] (F) at (-4,2.5){$\triangle$}; 
	\node[inner sep=0cm] (0) at (0,0){$\cdot$};
        \node[inner sep=0cm] (1) at (1.5,1){$\cdot$};
        \node[inner sep=0cm] (2) at (-1.5,1){$\cdot$};
        \node[inner sep=0cm] (3) at (1.5,2){$\cdot$};
        \node[inner sep=0cm] (4) at (-1.5,2){$\cdot$};
        \node[inner sep=0cm] (5) at (0,3){$\cdot$};
        \draw[red,-] (2) edge (4);
    \end{tikzpicture}
}

\newcommand{\dpd}{
    \begin{tikzpicture}[scale=0.12,baseline=0.2mm]	
    
    
	\node[inner sep=0cm] (0) at (0,0){$\cdot$};
        \node[inner sep=0cm] (1) at (1.5,1){$\cdot$};
        \node[inner sep=0cm] (2) at (-1.5,1){$\cdot$};
        \node[inner sep=0cm] (3) at (1.5,2){$\cdot$};
        \node[inner sep=0cm] (4) at (-1.5,2){$\cdot$};
        \node[inner sep=0cm] (5) at (0,3){$\cdot$};
        \draw[red,-] (0) edge (2);
        \draw[red,-] (0) edge (4);
    \end{tikzpicture}
}

\newcommand{\dpe}{
    \begin{tikzpicture}[scale=0.12,baseline=0.2mm]	
    \node[cyan] (S) at (4,2.5){$\triangle$}; 
    \node[white] (F) at (-4,2.5){$\triangle$};
	\node[inner sep=0cm] (0) at (0,0){$\cdot$};
        \node[inner sep=0cm] (1) at (1.5,1){$\cdot$};
        \node[inner sep=0cm] (2) at (-1.5,1){$\cdot$};
        \node[inner sep=0cm] (3) at (1.5,2){$\cdot$};
        \node[inner sep=0cm] (4) at (-1.5,2){$\cdot$};
        \node[inner sep=0cm] (5) at (0,3){$\cdot$};
        \draw[red,-] (0) edge (2);
        \draw[red,-] (0) edge (4);
        \draw[red,-] (2) edge (4);
    \end{tikzpicture}
}

\newcommand{\dpf}{
    \begin{tikzpicture}[scale=0.12,baseline=0.2mm]	
    \node[cyan] (S) at (4,2.5){$\triangle$};    
    \node[white] (F) at (-4,2.5){$\triangle$}; 
	\node[inner sep=0cm] (0) at (0,0){$\cdot$};
        \node[inner sep=0cm] (1) at (1.5,1){$\cdot$};
        \node[inner sep=0cm] (2) at (-1.5,1){$\cdot$};
        \node[inner sep=0cm] (3) at (1.5,2){$\cdot$};
        \node[inner sep=0cm] (4) at (-1.5,2){$\cdot$};
        \node[inner sep=0cm] (5) at (0,3){$\cdot$};
        \draw[red,-] (0) edge (1);
    \end{tikzpicture}
}

\newcommand{\dpg}{
    \begin{tikzpicture}[scale=0.12,baseline=0.2mm]	
    \node[cyan] (S) at (4,2.5){$\triangle$};        
    \node[white] (F) at (-4,2.5){$\triangle$}; 
	\node[inner sep=0cm] (0) at (0,0){$\cdot$};
        \node[inner sep=0cm] (1) at (1.5,1){$\cdot$};
        \node[inner sep=0cm] (2) at (-1.5,1){$\cdot$};
        \node[inner sep=0cm] (3) at (1.5,2){$\cdot$};
        \node[inner sep=0cm] (4) at (-1.5,2){$\cdot$};
        \node[inner sep=0cm] (5) at (0,3){$\cdot$};
        \draw[red,-] (0) edge (1);
        \draw[red,-] (0) edge (2);
    \end{tikzpicture}
}

\newcommand{\dph}{
    \begin{tikzpicture}[scale=0.12,baseline=0.2mm]	
    \node[cyan] (S) at (4,2.5){$\triangle$}; 
    \node[white] (F) at (-4,2.5){$\triangle$};
	\node[inner sep=0cm] (0) at (0,0){$\cdot$};
        \node[inner sep=0cm] (1) at (1.5,1){$\cdot$};
        \node[inner sep=0cm] (2) at (-1.5,1){$\cdot$};
        \node[inner sep=0cm] (3) at (1.5,2){$\cdot$};
        \node[inner sep=0cm] (4) at (-1.5,2){$\cdot$};
        \node[inner sep=0cm] (5) at (0,3){$\cdot$};
        \draw[red,-] (0) edge (1);
        \draw[red,-] (2) edge (4);
    \end{tikzpicture}
}

\newcommand{\dpi}{
    \begin{tikzpicture}[scale=0.12,baseline=0.2mm]
	\node[inner sep=0cm] (0) at (0,0){$\cdot$};
        \node[inner sep=0cm] (1) at (1.5,1){$\cdot$};
        \node[inner sep=0cm] (2) at (-1.5,1){$\cdot$};
        \node[inner sep=0cm] (3) at (1.5,2){$\cdot$};
        \node[inner sep=0cm] (4) at (-1.5,2){$\cdot$};
        \node[inner sep=0cm] (5) at (0,3){$\cdot$};
        \draw[red,-] (0) edge (1);
        \draw[red,-] (0) edge (2);
        \draw[red,-] (0) edge (4);
    \end{tikzpicture}
}

\newcommand{\dpj}{
    \begin{tikzpicture}[scale=0.12,baseline=0.2mm]	
    \node[cyan] (S) at (4,2.5){$\triangle$};  
    \node[white] (F) at (-4,2.5){$\triangle$};  
	\node[inner sep=0cm] (0) at (0,0){$\cdot$};
        \node[inner sep=0cm] (1) at (1.5,1){$\cdot$};
        \node[inner sep=0cm] (2) at (-1.5,1){$\cdot$};
        \node[inner sep=0cm] (3) at (1.5,2){$\cdot$};
        \node[inner sep=0cm] (4) at (-1.5,2){$\cdot$};
        \node[inner sep=0cm] (5) at (0,3){$\cdot$};
        \draw[red,-] (0) edge (1);
        \draw[red,-] (0) edge (2);
        \draw[red,-] (0) edge (4);
        \draw[red,-] (2) edge (4);
    \end{tikzpicture}
}

\newcommand{\dpba}{
    \begin{tikzpicture}[scale=0.12,baseline=0.2mm]	
    \node[cyan] (S) at (4,2.5){$\triangle$}; 
    \node[white] (F) at (-4,2.5){$\triangle$};
	\node[inner sep=0cm] (0) at (0,0){$\cdot$};
        \node[inner sep=0cm] (1) at (1.5,1){$\cdot$};
        \node[inner sep=0cm] (2) at (-1.5,1){$\cdot$};
        \node[inner sep=0cm] (3) at (1.5,2){$\cdot$};
        \node[inner sep=0cm] (4) at (-1.5,2){$\cdot$};
        \node[inner sep=0cm] (5) at (0,3){$\cdot$};
        \draw[red,-] (0) edge (1);
        \draw[red,-] (2) edge (3);
    \end{tikzpicture}
}

\newcommand{\dpbb}{
    \begin{tikzpicture}[scale=0.12,baseline=0.2mm]	
	\node[inner sep=0cm] (0) at (0,0){$\cdot$};
        \node[inner sep=0cm] (1) at (1.5,1){$\cdot$};
        \node[inner sep=0cm] (2) at (-1.5,1){$\cdot$};
        \node[inner sep=0cm] (3) at (1.5,2){$\cdot$};
        \node[inner sep=0cm] (4) at (-1.5,2){$\cdot$};
        \node[inner sep=0cm] (5) at (0,3){$\cdot$};
        \draw[red,-] (0) edge (1);
        \draw[red,-] (0) edge (2);
        \draw[red,-] (0) edge (3);
    \end{tikzpicture}
}
\newcommand{\dpbc}{
    \begin{tikzpicture}[scale=0.12,baseline=0.2mm]	
    \node[cyan] (S) at (4,2.5){$\triangle$};         
    \node[white] (L) at (-4,1.5){$\diamondsuit$};         
    \node[magenta] (C) at (4,0){$\heartsuit$};
	\node[inner sep=0cm] (0) at (0,0){$\cdot$};
        \node[inner sep=0cm] (1) at (1.5,1){$\cdot$};
        \node[inner sep=0cm] (2) at (-1.5,1){$\cdot$};
        \node[inner sep=0cm] (3) at (1.5,2){$\cdot$};
        \node[inner sep=0cm] (4) at (-1.5,2){$\cdot$};
        \node[inner sep=0cm] (5) at (0,3){$\cdot$};
        \draw[red,-] (0) edge (1);
        \draw[red,-] (2) edge (3);
        \draw[red,-] (4) edge (5);
    \end{tikzpicture}
}

\newcommand{\dpbd}{
    \begin{tikzpicture}[scale=0.12,baseline=0.2mm]	
    \node[cyan] (S) at (4,2.5){$\triangle$}; 
    \node[white] (F) at (-4,2.5){$\triangle$};
	\node[inner sep=0cm] (0) at (0,0){$\cdot$};
        \node[inner sep=0cm] (1) at (1.5,1){$\cdot$};
        \node[inner sep=0cm] (2) at (-1.5,1){$\cdot$};
        \node[inner sep=0cm] (3) at (1.5,2){$\cdot$};
        \node[inner sep=0cm] (4) at (-1.5,2){$\cdot$};
        \node[inner sep=0cm] (5) at (0,3){$\cdot$};
        \draw[red,-] (0) edge (1);
        \draw[red,-] (2) edge (3);
        \draw[red,-] (2) edge (4);
    \end{tikzpicture}
}

\newcommand{\dpbe}{
    \begin{tikzpicture}[scale=0.12,baseline=0.2mm]	
    \node[violet] (L) at (-4,1.5){$\vardiamond$};         
    \node[magenta] (C) at (4,0){$\heartsuit$};
	\node[inner sep=0cm] (0) at (0,0){$\cdot$};
        \node[inner sep=0cm] (1) at (1.5,1){$\cdot$};
        \node[inner sep=0cm] (2) at (-1.5,1){$\cdot$};
        \node[inner sep=0cm] (3) at (1.5,2){$\cdot$};
        \node[inner sep=0cm] (4) at (-1.5,2){$\cdot$};
        \node[inner sep=0cm] (5) at (0,3){$\cdot$};
        \draw[red,-] (0) edge (1);
        \draw[red,-] (2) edge (3);
        \draw[red,-] (2) edge (4);
        \draw[red,-] (2) edge (5);
    \end{tikzpicture}
}

\newcommand{\dpbf}{
    \begin{tikzpicture}[scale=0.12,baseline=0.2mm]	
	\node[inner sep=0cm] (0) at (0,0){$\cdot$};
        \node[inner sep=0cm] (1) at (1.5,1){$\cdot$};
        \node[inner sep=0cm] (2) at (-1.5,1){$\cdot$};
        \node[inner sep=0cm] (3) at (1.5,2){$\cdot$};
        \node[inner sep=0cm] (4) at (-1.5,2){$\cdot$};
        \node[inner sep=0cm] (5) at (0,3){$\cdot$};
        \draw[red,-] (0) edge (1);
        \draw[red,-] (0) edge (2);
        \draw[red,-] (0) edge (3);
        \draw[red,-] (0) edge (4);
    \end{tikzpicture}
}

\newcommand{\dpbg}{
    \begin{tikzpicture}[scale=0.12,baseline=0.2mm]	
    \node[violet] (L) at (-4,1.5){$\diamondsuit$};         
    \node[magenta] (C) at (4,0){$\heartsuit$};
	\node[inner sep=0cm] (0) at (0,0){$\cdot$};
        \node[inner sep=0cm] (1) at (1.5,1){$\cdot$};
        \node[inner sep=0cm] (2) at (-1.5,1){$\cdot$};
        \node[inner sep=0cm] (3) at (1.5,2){$\cdot$};
        \node[inner sep=0cm] (4) at (-1.5,2){$\cdot$};
        \node[inner sep=0cm] (5) at (0,3){$\cdot$};
        \draw[red,-] (0) edge (1);
        \draw[red,-] (0) edge (2);
        \draw[red,-] (0) edge (3);
        \draw[red,-] (0) edge (4);
        \draw[red,-] (0) edge (5);
    \end{tikzpicture}
}

\newcommand{\dpbh}{
    \begin{tikzpicture}[scale=0.12,baseline=0.2mm]	
    \node[violet] (L) at (-4,1.5){$\vardiamond$};         
    \node[magenta] (C) at (4,0){$\heartsuit$};
	\node[inner sep=0cm] (0) at (0,0){$\cdot$};
        \node[inner sep=0cm] (1) at (1.5,1){$\cdot$};
        \node[inner sep=0cm] (2) at (-1.5,1){$\cdot$};
        \node[inner sep=0cm] (3) at (1.5,2){$\cdot$};
        \node[inner sep=0cm] (4) at (-1.5,2){$\cdot$};
        \node[inner sep=0cm] (5) at (0,3){$\cdot$};
        \draw[red,-] (0) edge (1);
        \draw[red,-] (2) edge (3);
        \draw[red,-] (2) edge (4);
        \draw[red,-] (2) edge (5);
        \draw[red,-] (4) edge (5);
    \end{tikzpicture}
}

\newcommand{\dpbi}{
    \begin{tikzpicture}[scale=0.12,baseline=0.2mm]	
	\node[inner sep=0cm] (0) at (0,0){$\cdot$};
        \node[inner sep=0cm] (1) at (1.5,1){$\cdot$};
        \node[inner sep=0cm] (2) at (-1.5,1){$\cdot$};
        \node[inner sep=0cm] (3) at (1.5,2){$\cdot$};
        \node[inner sep=0cm] (4) at (-1.5,2){$\cdot$};
        \node[inner sep=0cm] (5) at (0,3){$\cdot$};
        \draw[red,-] (0) edge (1);
        \draw[red,-] (0) edge (2);
        \draw[red,-] (0) edge (3);
        \draw[red,-] (0) edge (4);
        \draw[red,-] (2) edge (4);
    \end{tikzpicture}
}

\newcommand{\dpbj}{
    \begin{tikzpicture}[scale=0.12,baseline=0.2mm]	

	\node[inner sep=0cm] (0) at (0,0){$\cdot$};
        \node[inner sep=0cm] (1) at (1.5,1){$\cdot$};
        \node[inner sep=0cm] (2) at (-1.5,1){$\cdot$};
        \node[inner sep=0cm] (3) at (1.5,2){$\cdot$};
        \node[inner sep=0cm] (4) at (-1.5,2){$\cdot$};
        \node[inner sep=0cm] (5) at (0,3){$\cdot$};
        \draw[red,-] (0) edge (1);
        \draw[red,-] (0) edge (2);
        \draw[red,-] (0) edge (3);
        \draw[red,-] (2) edge (3);
    \end{tikzpicture}
}

\newcommand{\dpca}{
    \begin{tikzpicture}[scale=0.12,baseline=0.2mm]	            
    \node[violet] (L) at (-4,1.5){$\diamondsuit$};         
    \node[white] (F) at (4,1.5){$\diamondsuit$};       
   	\node[inner sep=0cm] (0) at (0,0){$\cdot$};
        \node[inner sep=0cm] (1) at (1.5,1){$\cdot$};
        \node[inner sep=0cm] (2) at (-1.5,1){$\cdot$};
        \node[inner sep=0cm] (3) at (1.5,2){$\cdot$};
        \node[inner sep=0cm] (4) at (-1.5,2){$\cdot$};
        \node[inner sep=0cm] (5) at (0,3){$\cdot$};
        \draw[red,-] (0) edge (1);
        \draw[red,-] (0) edge (2);
        \draw[red,-] (0) edge (3);
        \draw[red,-] (0) edge (4);
        \draw[red,-] (0) edge (5);
        \draw[red,-] (2) edge (4);
    \end{tikzpicture}
}

\newcommand{\dpcb}{
    \begin{tikzpicture}[scale=0.12,baseline=0.2mm]	
    \node[violet] (L) at (-4,1.5){$\vardiamond$};         
    \node[magenta] (C) at (4,0){$\heartsuit$};
	\node[inner sep=0cm] (0) at (0,0){$\cdot$};
        \node[inner sep=0cm] (1) at (1.5,1){$\cdot$};
        \node[inner sep=0cm] (2) at (-1.5,1){$\cdot$};
        \node[inner sep=0cm] (3) at (1.5,2){$\cdot$};
        \node[inner sep=0cm] (4) at (-1.5,2){$\cdot$};
        \node[inner sep=0cm] (5) at (0,3){$\cdot$};
        \draw[red,-] (0) edge (1);
        \draw[red,-] (2) edge (3);
        \draw[red,-] (2) edge (4);
        \draw[red,-] (2) edge (5);
        \draw[red,-] (3) edge (5);
    \end{tikzpicture}
}

\newcommand{\dpcc}{
    \begin{tikzpicture}[scale=0.12,baseline=0.2mm]	
	\node[inner sep=0cm] (0) at (0,0){$\cdot$};
        \node[inner sep=0cm] (1) at (1.5,1){$\cdot$};
        \node[inner sep=0cm] (2) at (-1.5,1){$\cdot$};
        \node[inner sep=0cm] (3) at (1.5,2){$\cdot$};
        \node[inner sep=0cm] (4) at (-1.5,2){$\cdot$};
        \node[inner sep=0cm] (5) at (0,3){$\cdot$};
        \draw[red,-] (0) edge (1);
        \draw[red,-] (0) edge (2);
        \draw[red,-] (0) edge (3);
        \draw[red,-] (2) edge (3);
        \draw[red,-] (4) edge (5);
    \end{tikzpicture}
}

\newcommand{\dpcd}{
    \begin{tikzpicture}[scale=0.12,baseline=0.2mm]	
	\node[inner sep=0cm] (0) at (0,0){$\cdot$};
        \node[inner sep=0cm] (1) at (1.5,1){$\cdot$};
        \node[inner sep=0cm] (2) at (-1.5,1){$\cdot$};
        \node[inner sep=0cm] (3) at (1.5,2){$\cdot$};
        \node[inner sep=0cm] (4) at (-1.5,2){$\cdot$};
        \node[inner sep=0cm] (5) at (0,3){$\cdot$};
        \draw[red,-] (0) edge (1);
        \draw[red,-] (0) edge (2);
        \draw[red,-] (0) edge (3);
        \draw[red,-] (0) edge (4);
        \draw[red,-] (2) edge (3);
    \end{tikzpicture}
}

\newcommand{\dpce}{
    \begin{tikzpicture}[scale=0.12,baseline=0.2mm]	
	\node[inner sep=0cm] (0) at (0,0){$\cdot$};
        \node[inner sep=0cm] (1) at (1.5,1){$\cdot$};
        \node[inner sep=0cm] (2) at (-1.5,1){$\cdot$};
        \node[inner sep=0cm] (3) at (1.5,2){$\cdot$};
        \node[inner sep=0cm] (4) at (-1.5,2){$\cdot$};
        \node[inner sep=0cm] (5) at (0,3){$\cdot$};
        \draw[red,-] (0) edge (1);
        \draw[red,-] (0) edge (2);
        \draw[red,-] (0) edge (3);
        \draw[red,-] (0) edge (4);
        \draw[red,-] (2) edge (3);
        \draw[red,-] (2) edge (4);
    \end{tikzpicture}
}

\newcommand{\dpcf}{
    \begin{tikzpicture}[scale=0.12,baseline=0.2mm]	
    \node[cyan] (S) at (4,2.5){$\triangle$};         
    \node[violet] (L) at (-4,1.5){$\vardiamond$};         
    \node[magenta] (C) at (4,0){$\heartsuit$};
	\node[inner sep=0cm] (0) at (0,0){$\cdot$};
        \node[inner sep=0cm] (1) at (1.5,1){$\cdot$};
        \node[inner sep=0cm] (2) at (-1.5,1){$\cdot$};
        \node[inner sep=0cm] (3) at (1.5,2){$\cdot$};
        \node[inner sep=0cm] (4) at (-1.5,2){$\cdot$};
        \node[inner sep=0cm] (5) at (0,3){$\cdot$};
        \draw[red,-] (0) edge (1);
        \draw[red,-] (2) edge (3);
        \draw[red,-] (2) edge (4);
        \draw[red,-] (2) edge (5);
        \draw[red,-] (3) edge (5);
        \draw[red,-] (4) edge (5);
    \end{tikzpicture}
}

\newcommand{\dpcg}{
    \begin{tikzpicture}[scale=0.12,baseline=0.2mm]	       
    \node[violet] (L) at (-4,1.5){$\diamondsuit$}; 
    \node[white] (F) at (4,1.5){$\diamondsuit$};    
	\node[inner sep=0cm] (0) at (0,0){$\cdot$};
        \node[inner sep=0cm] (1) at (1.5,1){$\cdot$};
        \node[inner sep=0cm] (2) at (-1.5,1){$\cdot$};
        \node[inner sep=0cm] (3) at (1.5,2){$\cdot$};
        \node[inner sep=0cm] (4) at (-1.5,2){$\cdot$};
        \node[inner sep=0cm] (5) at (0,3){$\cdot$};
        \draw[red,-] (0) edge (1);
        \draw[red,-] (0) edge (2);
        \draw[red,-] (0) edge (3);
        \draw[red,-] (0) edge (4);
        \draw[red,-] (0) edge (5);
        \draw[red,-] (2) edge (3);
    \end{tikzpicture}
}

\newcommand{\dpch}{
    \begin{tikzpicture}[scale=0.12,baseline=0.2mm]	       
    \node[violet] (L) at (-4,1.5){$\diamondsuit$};         
    \node[white] (F) at (4,1.5){$\diamondsuit$};    
	\node[inner sep=0cm] (0) at (0,0){$\cdot$};
        \node[inner sep=0cm] (1) at (1.5,1){$\cdot$};
        \node[inner sep=0cm] (2) at (-1.5,1){$\cdot$};
        \node[inner sep=0cm] (3) at (1.5,2){$\cdot$};
        \node[inner sep=0cm] (4) at (-1.5,2){$\cdot$};
        \node[inner sep=0cm] (5) at (0,3){$\cdot$};
        \draw[red,-] (0) edge (1);
        \draw[red,-] (0) edge (2);
        \draw[red,-] (0) edge (3);
        \draw[red,-] (0) edge (4);
        \draw[red,-] (0) edge (5);
        \draw[red,-] (2) edge (3);
        \draw[red,-] (2) edge (4);
    \end{tikzpicture}
}

\newcommand{\dpci}{
    \begin{tikzpicture}[scale=0.12,baseline=0.2mm]	
    \node[violet] (L) at (-4,1.5){$\diamondsuit$};         
    \node[magenta] (C) at (4,0){$\heartsuit$};
	\node[inner sep=0cm] (0) at (0,0){$\cdot$};
        \node[inner sep=0cm] (1) at (1.5,1){$\cdot$};
        \node[inner sep=0cm] (2) at (-1.5,1){$\cdot$};
        \node[inner sep=0cm] (3) at (1.5,2){$\cdot$};
        \node[inner sep=0cm] (4) at (-1.5,2){$\cdot$};
        \node[inner sep=0cm] (5) at (0,3){$\cdot$};
        \draw[red,-] (0) edge (1);
        \draw[red,-] (0) edge (2);
        \draw[red,-] (0) edge (3);
        \draw[red,-] (0) edge (4);
        \draw[red,-] (0) edge (5);
        \draw[red,-] (2) edge (3);
        \draw[red,-] (4) edge (5);
    \end{tikzpicture}
}

\newcommand{\dpcj}{
    \begin{tikzpicture}[scale=0.12,baseline=0.2mm]	      
    \node[violet] (L) at (-4,1.5){$\diamondsuit$};         
    \node[magenta] (C) at (4,0){$\heartsuit$};
	\node[inner sep=0cm] (0) at (0,0){$\cdot$};
        \node[inner sep=0cm] (1) at (1.5,1){$\cdot$};
        \node[inner sep=0cm] (2) at (-1.5,1){$\cdot$};
        \node[inner sep=0cm] (3) at (1.5,2){$\cdot$};
        \node[inner sep=0cm] (4) at (-1.5,2){$\cdot$};
        \node[inner sep=0cm] (5) at (0,3){$\cdot$};
        \draw[red,-] (0) edge (1);
        \draw[red,-] (0) edge (2);
        \draw[red,-] (0) edge (3);
        \draw[red,-] (0) edge (4);
        \draw[red,-] (0) edge (5);
        \draw[red,-] (2) edge (3);
        \draw[red,-] (2) edge (4);
        \draw[red,-] (2) edge (5);
    \end{tikzpicture}
}

\newcommand{\dpda}{
    \begin{tikzpicture}[scale=0.12,baseline=0.2mm]	       
    \node[violet] (L) at (-4,1.5){$\diamondsuit$};         
    \node[magenta] (C) at (4,0){$\heartsuit$};
	\node[inner sep=0cm] (0) at (0,0){$\cdot$};
        \node[inner sep=0cm] (1) at (1.5,1){$\cdot$};
        \node[inner sep=0cm] (2) at (-1.5,1){$\cdot$};
        \node[inner sep=0cm] (3) at (1.5,2){$\cdot$};
        \node[inner sep=0cm] (4) at (-1.5,2){$\cdot$};
        \node[inner sep=0cm] (5) at (0,3){$\cdot$};
        \draw[red,-] (0) edge (1);
        \draw[red,-] (0) edge (2);
        \draw[red,-] (0) edge (3);
        \draw[red,-] (0) edge (4);
        \draw[red,-] (0) edge (5);
        \draw[red,-] (2) edge (3);
        \draw[red,-] (2) edge (4);
        \draw[red,-] (2) edge (5);
        \draw[red,-] (4) edge (5);
    \end{tikzpicture}
}

\newcommand{\dpdb}{
    \begin{tikzpicture}[scale=0.12,baseline=0.2mm]	       
	\node[inner sep=0cm] (0) at (0,0){$\cdot$};
        \node[inner sep=0cm] (1) at (1.5,1){$\cdot$};
        \node[inner sep=0cm] (2) at (-1.5,1){$\cdot$};
        \node[inner sep=0cm] (3) at (1.5,2){$\cdot$};
        \node[inner sep=0cm] (4) at (-1.5,2){$\cdot$};
        \node[inner sep=0cm] (5) at (0,3){$\cdot$};
        \draw[red,-] (0) edge (1);
        \draw[red,-] (0) edge (2);
        \draw[red,-] (0) edge (3);
        \draw[red,-] (1) edge (3);
    \end{tikzpicture}
}

\newcommand{\dpdc}{
    \begin{tikzpicture}[scale=0.12,baseline=0.2mm]	        
    \node[violet] (L) at (-4,1.5){$\diamondsuit$};         
    \node[magenta] (C) at (4,0){$\heartsuit$};
	\node[inner sep=0cm] (0) at (0,0){$\cdot$};
        \node[inner sep=0cm] (1) at (1.5,1){$\cdot$};
        \node[inner sep=0cm] (2) at (-1.5,1){$\cdot$};
        \node[inner sep=0cm] (3) at (1.5,2){$\cdot$};
        \node[inner sep=0cm] (4) at (-1.5,2){$\cdot$};
        \node[inner sep=0cm] (5) at (0,3){$\cdot$};
        \draw[red,-] (0) edge (1);
        \draw[red,-] (0) edge (2);
        \draw[red,-] (0) edge (3);
        \draw[red,-] (0) edge (4);
        \draw[red,-] (0) edge (5);
        \draw[red,-] (2) edge (3);
        \draw[red,-] (2) edge (4);
        \draw[red,-] (2) edge (5);
        \draw[red,-] (3) edge (5);
    \end{tikzpicture}
}

\newcommand{\dpdd}{
    \begin{tikzpicture}[scale=0.12,baseline=0.2mm]
	\node[inner sep=0cm] (0) at (0,0){$\cdot$};
        \node[inner sep=0cm] (1) at (1.5,1){$\cdot$};
        \node[inner sep=0cm] (2) at (-1.5,1){$\cdot$};
        \node[inner sep=0cm] (3) at (1.5,2){$\cdot$};
        \node[inner sep=0cm] (4) at (-1.5,2){$\cdot$};
        \node[inner sep=0cm] (5) at (0,3){$\cdot$};
        \draw[red,-] (0) edge (1);
        \draw[red,-] (0) edge (2);
        \draw[red,-] (0) edge (3);
        \draw[red,-] (0) edge (4);
        \draw[red,-] (1) edge (3);
    \end{tikzpicture}
}

\newcommand{\dpde}{
    \begin{tikzpicture}[scale=0.12,baseline=0.2mm]	
    \node[violet] (L) at (-4,1.5){$\diamondsuit$};         
    \node[magenta] (C) at (4,0){$\heartsuit$};
	\node[inner sep=0cm] (0) at (0,0){$\cdot$};
        \node[inner sep=0cm] (1) at (1.5,1){$\cdot$};
        \node[inner sep=0cm] (2) at (-1.5,1){$\cdot$};
        \node[inner sep=0cm] (3) at (1.5,2){$\cdot$};
        \node[inner sep=0cm] (4) at (-1.5,2){$\cdot$};
        \node[inner sep=0cm] (5) at (0,3){$\cdot$};
        \draw[red,-] (0) edge (1);
        \draw[red,-] (0) edge (2);
        \draw[red,-] (0) edge (3);
        \draw[red,-] (0) edge (4);
        \draw[red,-] (0) edge (5);
        \draw[red,-] (2) edge (3);
        \draw[red,-] (2) edge (4);
        \draw[red,-] (2) edge (5);
        \draw[red,-] (3) edge (5);
        \draw[red,-] (4) edge (5);
    \end{tikzpicture}
}

\newcommand{\dpdf}{
    \begin{tikzpicture}[scale=0.12,baseline=0.2mm]	         
    \node[violet] (L) at (-4,1.5){$\diamondsuit$};         
    \node[white] (F) at (4,1.5){$\diamondsuit$};
	\node[inner sep=0cm] (0) at (0,0){$\cdot$};
        \node[inner sep=0cm] (1) at (1.5,1){$\cdot$};
        \node[inner sep=0cm] (2) at (-1.5,1){$\cdot$};
        \node[inner sep=0cm] (3) at (1.5,2){$\cdot$};
        \node[inner sep=0cm] (4) at (-1.5,2){$\cdot$};
        \node[inner sep=0cm] (5) at (0,3){$\cdot$};
        \draw[red,-] (0) edge (1);
        \draw[red,-] (0) edge (2);
        \draw[red,-] (0) edge (3);
        \draw[red,-] (0) edge (4);
        \draw[red,-] (0) edge (5);
        \draw[red,-] (1) edge (3);
    \end{tikzpicture}
}

\newcommand{\dpdg}{
    \begin{tikzpicture}[scale=0.12,baseline=0.2mm]	
    \node[white] (S) at (4,2.5){$\triangle$};         
    \node[white] (F) at (-4,2.5){$\triangle$};  
	\node[inner sep=0cm] (0) at (0,0){$\cdot$};
        \node[inner sep=0cm] (1) at (1.5,1){$\cdot$};
        \node[inner sep=0cm] (2) at (-1.5,1){$\cdot$};
        \node[inner sep=0cm] (3) at (1.5,2){$\cdot$};
        \node[inner sep=0cm] (4) at (-1.5,2){$\cdot$};
        \node[inner sep=0cm] (5) at (0,3){$\cdot$};
        \draw[red,-] (0) edge (1);
        \draw[red,-] (0) edge (2);
        \draw[red,-] (0) edge (3);
        \draw[red,-] (0) edge (4);
        \draw[red,-] (1) edge (3);
        \draw[red,-] (2) edge (4);
    \end{tikzpicture}
}

\newcommand{\dpdh}{
    \begin{tikzpicture}[scale=0.12,baseline=0.2mm]	
    \node[violet] (L) at (-4,1.5){$\diamondsuit$};
    \node[white] (F) at (4,1.5){$\diamondsuit$};
	\node[inner sep=0cm] (0) at (0,0){$\cdot$};
        \node[inner sep=0cm] (1) at (1.5,1){$\cdot$};
        \node[inner sep=0cm] (2) at (-1.5,1){$\cdot$};
        \node[inner sep=0cm] (3) at (1.5,2){$\cdot$};
        \node[inner sep=0cm] (4) at (-1.5,2){$\cdot$};
        \node[inner sep=0cm] (5) at (0,3){$\cdot$};
        \draw[red,-] (0) edge (1);
        \draw[red,-] (0) edge (2);
        \draw[red,-] (0) edge (3);
        \draw[red,-] (0) edge (4);
        \draw[red,-] (0) edge (5);
        \draw[red,-] (1) edge (3);
        \draw[red,-] (2) edge (4);
    \end{tikzpicture}
}

\newcommand{\dpdi}{
    \begin{tikzpicture}[scale=0.12,baseline=0.2mm]	
    \node[cyan] (S) at (4,2.5){$\triangle$};  
    \node[white] (F) at (-4,2.5){$\triangle$};  
	\node[inner sep=0cm] (0) at (0,0){$\cdot$};
        \node[inner sep=0cm] (1) at (1.5,1){$\cdot$};
        \node[inner sep=0cm] (2) at (-1.5,1){$\cdot$};
        \node[inner sep=0cm] (3) at (1.5,2){$\cdot$};
        \node[inner sep=0cm] (4) at (-1.5,2){$\cdot$};
        \node[inner sep=0cm] (5) at (0,3){$\cdot$};
        \draw[red,-] (0) edge (1);
        \draw[red,-] (0) edge (2);
        \draw[red,-] (0) edge (3);
        \draw[red,-] (1) edge (3);
        \draw[red,-] (2) edge (3);
    \end{tikzpicture}
}

\newcommand{\dpdj}{
    \begin{tikzpicture}[scale=0.12,baseline=0.2mm]	
    \node[cyan] (S) at (4,2.5){$\triangle$};         
    \node[white] (L) at (-4,1.5){$\vardiamond$};         
	\node[inner sep=0cm] (0) at (0,0){$\cdot$};
        \node[inner sep=0cm] (1) at (1.5,1){$\cdot$};
        \node[inner sep=0cm] (2) at (-1.5,1){$\cdot$};
        \node[inner sep=0cm] (3) at (1.5,2){$\cdot$};
        \node[inner sep=0cm] (4) at (-1.5,2){$\cdot$};
        \node[inner sep=0cm] (5) at (0,3){$\cdot$};
        \draw[red,-] (0) edge (1);
        \draw[red,-] (0) edge (2);
        \draw[red,-] (0) edge (3);
        \draw[red,-] (1) edge (3);
        \draw[red,-] (2) edge (3);
        \draw[red,-] (4) edge (5);
    \end{tikzpicture}
}

\newcommand{\dpea}{
    \begin{tikzpicture}[scale=0.12,baseline=0.2mm]
	\node[inner sep=0cm] (0) at (0,0){$\cdot$};
        \node[inner sep=0cm] (1) at (1.5,1){$\cdot$};
        \node[inner sep=0cm] (2) at (-1.5,1){$\cdot$};
        \node[inner sep=0cm] (3) at (1.5,2){$\cdot$};
        \node[inner sep=0cm] (4) at (-1.5,2){$\cdot$};
        \node[inner sep=0cm] (5) at (0,3){$\cdot$};
        \draw[red,-] (0) edge (1);
        \draw[red,-] (0) edge (2);
        \draw[red,-] (0) edge (3);
        \draw[red,-] (0) edge (4);
        \draw[red,-] (1) edge (3);
        \draw[red,-] (2) edge (3);
    \end{tikzpicture}
}

\newcommand{\dpeb}{
    \begin{tikzpicture}[scale=0.12,baseline=0.2mm]	      
    \node[violet] (L) at (-4,1.5){$\diamondsuit$};         
    \node[white] (F) at (4,1.5){$\diamondsuit$};    
	\node[inner sep=0cm] (0) at (0,0){$\cdot$};
        \node[inner sep=0cm] (1) at (1.5,1){$\cdot$};
        \node[inner sep=0cm] (2) at (-1.5,1){$\cdot$};
        \node[inner sep=0cm] (3) at (1.5,2){$\cdot$};
        \node[inner sep=0cm] (4) at (-1.5,2){$\cdot$};
        \node[inner sep=0cm] (5) at (0,3){$\cdot$};
        \draw[red,-] (0) edge (1);
        \draw[red,-] (0) edge (2);
        \draw[red,-] (0) edge (3);
        \draw[red,-] (0) edge (4);
        \draw[red,-] (0) edge (5);
        \draw[red,-] (1) edge (3);
        \draw[red,-] (2) edge (3);
            \end{tikzpicture}
}

\newcommand{\dpec}{
    \begin{tikzpicture}[scale=0.12,baseline=0.2mm]	
    \node[cyan] (S) at (4,2.5){$\triangle$};  
    \node[white] (F) at (-4,2.5){$\triangle$};  
	\node[inner sep=0cm] (0) at (0,0){$\cdot$};
        \node[inner sep=0cm] (1) at (1.5,1){$\cdot$};
        \node[inner sep=0cm] (2) at (-1.5,1){$\cdot$};
        \node[inner sep=0cm] (3) at (1.5,2){$\cdot$};
        \node[inner sep=0cm] (4) at (-1.5,2){$\cdot$};
        \node[inner sep=0cm] (5) at (0,3){$\cdot$};
        \draw[red,-] (0) edge (1);
        \draw[red,-] (0) edge (2);
        \draw[red,-] (0) edge (3);
        \draw[red,-] (0) edge (4);
        \draw[red,-] (1) edge (3);
        \draw[red,-] (2) edge (3);
        \draw[red,-] (2) edge (4);
    \end{tikzpicture}
}

\newcommand{\dped}{
    \begin{tikzpicture}[scale=0.12,baseline=0.2mm]	      
    \node[violet] (L) at (-4,1.5){$\diamondsuit$};
    \node[white] (F) at (4,1.5){$\diamondsuit$};
	\node[inner sep=0cm] (0) at (0,0){$\cdot$};
        \node[inner sep=0cm] (1) at (1.5,1){$\cdot$};
        \node[inner sep=0cm] (2) at (-1.5,1){$\cdot$};
        \node[inner sep=0cm] (3) at (1.5,2){$\cdot$};
        \node[inner sep=0cm] (4) at (-1.5,2){$\cdot$};
        \node[inner sep=0cm] (5) at (0,3){$\cdot$};
        \draw[red,-] (0) edge (1);
        \draw[red,-] (0) edge (2);
        \draw[red,-] (0) edge (3);
        \draw[red,-] (0) edge (4);
        \draw[red,-] (0) edge (5);
        \draw[red,-] (1) edge (3);
        \draw[red,-] (2) edge (3);
        \draw[red,-] (2) edge (4);
    \end{tikzpicture}
}

\newcommand{\dpee}{
    \begin{tikzpicture}[scale=0.12,baseline=0.2mm]	
    \node[violet] (L) at (-4,1.5){$\diamondsuit$};         
    \node[white] (F) at (4,1.5){$\diamondsuit$};
	\node[inner sep=0cm] (0) at (0,0){$\cdot$};
        \node[inner sep=0cm] (1) at (1.5,1){$\cdot$};
        \node[inner sep=0cm] (2) at (-1.5,1){$\cdot$};
        \node[inner sep=0cm] (3) at (1.5,2){$\cdot$};
        \node[inner sep=0cm] (4) at (-1.5,2){$\cdot$};
        \node[inner sep=0cm] (5) at (0,3){$\cdot$};
        \draw[red,-] (0) edge (1);
        \draw[red,-] (0) edge (2);
        \draw[red,-] (0) edge (3);
        \draw[red,-] (0) edge (4);
        \draw[red,-] (0) edge (5);
        \draw[red,-] (1) edge (3);
        \draw[red,-] (2) edge (3);
        \draw[red,-] (4) edge (5);
    \end{tikzpicture}
}

\newcommand{\dpef}{
    \begin{tikzpicture}[scale=0.12,baseline=0.2mm]	
    \node[violet] (L) at (-4,1.5){$\diamondsuit$};         
    \node[white] (F) at (4,1.5){$\diamondsuit$};
	\node[inner sep=0cm] (0) at (0,0){$\cdot$};
        \node[inner sep=0cm] (1) at (1.5,1){$\cdot$};
        \node[inner sep=0cm] (2) at (-1.5,1){$\cdot$};
        \node[inner sep=0cm] (3) at (1.5,2){$\cdot$};
        \node[inner sep=0cm] (4) at (-1.5,2){$\cdot$};
        \node[inner sep=0cm] (5) at (0,3){$\cdot$};
        \draw[red,-] (0) edge (1);
        \draw[red,-] (0) edge (2);
        \draw[red,-] (0) edge (3);
        \draw[red,-] (0) edge (4);
        \draw[red,-] (0) edge (5);
        \draw[red,-] (1) edge (3);
        \draw[red,-] (2) edge (3);
        \draw[red,-] (2) edge (4);
        \draw[red,-] (2) edge (5);
    \end{tikzpicture}
}

\newcommand{\dpeg}{
    \begin{tikzpicture}[scale=0.12,baseline=0.2mm]	
    \node[violet] (L) at (-4,1.5){$\diamondsuit$};         
    \node[white] (F) at (4,1.5){$\diamondsuit$};
	\node[inner sep=0cm] (0) at (0,0){$\cdot$};
        \node[inner sep=0cm] (1) at (1.5,1){$\cdot$};
        \node[inner sep=0cm] (2) at (-1.5,1){$\cdot$};
        \node[inner sep=0cm] (3) at (1.5,2){$\cdot$};
        \node[inner sep=0cm] (4) at (-1.5,2){$\cdot$};
        \node[inner sep=0cm] (5) at (0,3){$\cdot$};
        \draw[red,-] (0) edge (1);
        \draw[red,-] (0) edge (2);
        \draw[red,-] (0) edge (3);
        \draw[red,-] (0) edge (4);
        \draw[red,-] (0) edge (5);
        \draw[red,-] (1) edge (3);
        \draw[red,-] (2) edge (3);
        \draw[red,-] (2) edge (4);
        \draw[red,-] (2) edge (5);
        \draw[red,-] (4) edge (5);
    \end{tikzpicture}
}

\newcommand{\dpeh}{
    \begin{tikzpicture}[scale=0.12,baseline=0.2mm]	         
    \node[violet] (L) at (-4,1.5){$\diamondsuit$};         
    \node[magenta] (C) at (4,0){$\heartsuit$};
	\node[inner sep=0cm] (0) at (0,0){$\cdot$};
        \node[inner sep=0cm] (1) at (1.5,1){$\cdot$};
        \node[inner sep=0cm] (2) at (-1.5,1){$\cdot$};
        \node[inner sep=0cm] (3) at (1.5,2){$\cdot$};
        \node[inner sep=0cm] (4) at (-1.5,2){$\cdot$};
        \node[inner sep=0cm] (5) at (0,3){$\cdot$};
        \draw[red,-] (0) edge (1);
        \draw[red,-] (0) edge (2);
        \draw[red,-] (0) edge (3);
        \draw[red,-] (0) edge (4);
        \draw[red,-] (0) edge (5);
        \draw[red,-] (1) edge (3);
        \draw[red,-] (1) edge (5);
    \end{tikzpicture}
}

\newcommand{\dpei}{
    \begin{tikzpicture}[scale=0.12,baseline=0.2mm]	       
    \node[violet] (L) at (-4,1.5){$\diamondsuit$};
    \node[white] (F) at (4,1.5){$\diamondsuit$};
	\node[inner sep=0cm] (0) at (0,0){$\cdot$};
        \node[inner sep=0cm] (1) at (1.5,1){$\cdot$};
        \node[inner sep=0cm] (2) at (-1.5,1){$\cdot$};
        \node[inner sep=0cm] (3) at (1.5,2){$\cdot$};
        \node[inner sep=0cm] (4) at (-1.5,2){$\cdot$};
        \node[inner sep=0cm] (5) at (0,3){$\cdot$};
        \draw[red,-] (0) edge (1);
        \draw[red,-] (0) edge (2);
        \draw[red,-] (0) edge (3);
        \draw[red,-] (0) edge (4);
        \draw[red,-] (0) edge (5);
        \draw[red,-] (1) edge (3);
        \draw[red,-] (1) edge (5);
        \draw[red,-] (2) edge (4);
    \end{tikzpicture}
}

\newcommand{\dpej}{
    \begin{tikzpicture}[scale=0.12,baseline=0.2mm]	       
    \node[violet] (L) at (-4,1.5){$\diamondsuit$};         
    \node[white] (F) at (4,1.5){$\diamondsuit$};
	\node[inner sep=0cm] (0) at (0,0){$\cdot$};
        \node[inner sep=0cm] (1) at (1.5,1){$\cdot$};
        \node[inner sep=0cm] (2) at (-1.5,1){$\cdot$};
        \node[inner sep=0cm] (3) at (1.5,2){$\cdot$};
        \node[inner sep=0cm] (4) at (-1.5,2){$\cdot$};
        \node[inner sep=0cm] (5) at (0,3){$\cdot$};
        \draw[red,-] (0) edge (1);
        \draw[red,-] (0) edge (2);
        \draw[red,-] (0) edge (3);
        \draw[red,-] (0) edge (4);
        \draw[red,-] (0) edge (5);
        \draw[red,-] (1) edge (3);
        \draw[red,-] (1) edge (5);
        \draw[red,-] (2) edge (3);
    \end{tikzpicture}
}

\newcommand{\dpfa}{
    \begin{tikzpicture}[scale=0.12,baseline=0.2mm]	
    \node[violet] (L) at (-4,1.5){$\diamondsuit$};         
    \node[magenta] (C) at (4,0){$\heartsuit$};
	\node[inner sep=0cm] (0) at (0,0){$\cdot$};
        \node[inner sep=0cm] (1) at (1.5,1){$\cdot$};
        \node[inner sep=0cm] (2) at (-1.5,1){$\cdot$};
        \node[inner sep=0cm] (3) at (1.5,2){$\cdot$};
        \node[inner sep=0cm] (4) at (-1.5,2){$\cdot$};
        \node[inner sep=0cm] (5) at (0,3){$\cdot$};
        \draw[red,-] (0) edge (1);
        \draw[red,-] (0) edge (2);
        \draw[red,-] (0) edge (3);
        \draw[red,-] (0) edge (4);
        \draw[red,-] (0) edge (5);
        \draw[red,-] (1) edge (3);
        \draw[red,-] (1) edge (5);
        \draw[red,-] (2) edge (3);
        \draw[red,-] (4) edge (5);
    \end{tikzpicture}
}

\newcommand{\dpfb}{
    \begin{tikzpicture}[scale=0.12,baseline=0.2mm]	
    \node[violet] (L) at (-4,1.5){$\diamondsuit$};         
    \node[white] (F) at (4,1.5){$\diamondsuit$};
	\node[inner sep=0cm] (0) at (0,0){$\cdot$};
        \node[inner sep=0cm] (1) at (1.5,1){$\cdot$};
        \node[inner sep=0cm] (2) at (-1.5,1){$\cdot$};
        \node[inner sep=0cm] (3) at (1.5,2){$\cdot$};
        \node[inner sep=0cm] (4) at (-1.5,2){$\cdot$};
        \node[inner sep=0cm] (5) at (0,3){$\cdot$};
        \draw[red,-] (0) edge (1);
        \draw[red,-] (0) edge (2);
        \draw[red,-] (0) edge (3);
        \draw[red,-] (0) edge (4);
        \draw[red,-] (0) edge (5);
        \draw[red,-] (1) edge (3);
        \draw[red,-] (1) edge (5);
        \draw[red,-] (2) edge (3);
        \draw[red,-] (2) edge (4);
    \end{tikzpicture}
}

\newcommand{\dpfc}{
    \begin{tikzpicture}[scale=0.12,baseline=0.2mm]	
    \node[violet] (L) at (-4,1.5){$\diamondsuit$};         
    \node[magenta] (C) at (4,0){$\heartsuit$};
	\node[inner sep=0cm] (0) at (0,0){$\cdot$};
        \node[inner sep=0cm] (1) at (1.5,1){$\cdot$};
        \node[inner sep=0cm] (2) at (-1.5,1){$\cdot$};
        \node[inner sep=0cm] (3) at (1.5,2){$\cdot$};
        \node[inner sep=0cm] (4) at (-1.5,2){$\cdot$};
        \node[inner sep=0cm] (5) at (0,3){$\cdot$};
        \draw[red,-] (0) edge (1);
        \draw[red,-] (0) edge (2);
        \draw[red,-] (0) edge (3);
        \draw[red,-] (0) edge (4);
        \draw[red,-] (0) edge (5);
        \draw[red,-] (1) edge (3);
        \draw[red,-] (1) edge (5);
        \draw[red,-] (2) edge (3);
        \draw[red,-] (2) edge (4);
        \draw[red,-] (2) edge (5);
    \end{tikzpicture}
}

\newcommand{\dpfd}{
    \begin{tikzpicture}[scale=0.12,baseline=0.2mm]	
    \node[violet] (L) at (-4,1.5){$\diamondsuit$};         
    \node[magenta] (C) at (4,0){$\heartsuit$};
	\node[inner sep=0cm] (0) at (0,0){$\cdot$};
        \node[inner sep=0cm] (1) at (1.5,1){$\cdot$};
        \node[inner sep=0cm] (2) at (-1.5,1){$\cdot$};
        \node[inner sep=0cm] (3) at (1.5,2){$\cdot$};
        \node[inner sep=0cm] (4) at (-1.5,2){$\cdot$};
        \node[inner sep=0cm] (5) at (0,3){$\cdot$};
        \draw[red,-] (0) edge (1);
        \draw[red,-] (0) edge (2);
        \draw[red,-] (0) edge (3);
        \draw[red,-] (0) edge (4);
        \draw[red,-] (0) edge (5);
        \draw[red,-] (1) edge (3);
        \draw[red,-] (1) edge (5);
        \draw[red,-] (2) edge (3);
        \draw[red,-] (2) edge (4);
        \draw[red,-] (2) edge (5);
        \draw[red,-] (4) edge (5);
    \end{tikzpicture}
}

\newcommand{\dpfe}{
    \begin{tikzpicture}[scale=0.12,baseline=0.2mm]	
    \node[violet] (L) at (-4,1.5){$\diamondsuit$};         
    \node[magenta] (C) at (4,0){$\heartsuit$};
	\node[inner sep=0cm] (0) at (0,0){$\cdot$};
        \node[inner sep=0cm] (1) at (1.5,1){$\cdot$};
        \node[inner sep=0cm] (2) at (-1.5,1){$\cdot$};
        \node[inner sep=0cm] (3) at (1.5,2){$\cdot$};
        \node[inner sep=0cm] (4) at (-1.5,2){$\cdot$};
        \node[inner sep=0cm] (5) at (0,3){$\cdot$};
        \draw[red,-] (0) edge (1);
        \draw[red,-] (0) edge (2);
        \draw[red,-] (0) edge (3);
        \draw[red,-] (0) edge (4);
        \draw[red,-] (0) edge (5);
        \draw[red,-] (1) edge (3);
        \draw[red,-] (1) edge (5);
        \draw[red,-] (2) edge (3);
        \draw[red,-] (2) edge (4);
        \draw[red,-] (2) edge (5);
        \draw[red,-] (3) edge (5);
    \end{tikzpicture}
}

\newcommand{\dpff}{
    \begin{tikzpicture}[scale=0.12,baseline=0.2mm]	\node[cyan] (S) at (4,2.5){$\triangle$};         \node[violet] (L) at (-4,1.5){$\diamondsuit$};         \node[magenta] (C) at (4,0){$\heartsuit$};
	\node[cyan] (S) at (4,2.5){$\triangle$};
        \node[violet] (L) at (-4,1.5){$\diamondsuit$};
        \node[magenta] (C) at (4,0){$\heartsuit$};
        \node[inner sep=0cm] (0) at (0,0){$\cdot$};
        \node[inner sep=0cm] (1) at (1.5,1){$\cdot$};
        \node[inner sep=0cm] (2) at (-1.5,1){$\cdot$};
        \node[inner sep=0cm] (3) at (1.5,2){$\cdot$};
        \node[inner sep=0cm] (4) at (-1.5,2){$\cdot$};
        \node[inner sep=0cm] (5) at (0,3){$\cdot$};
        \draw[red,-] (0) edge (1);
        \draw[red,-] (0) edge (2);
        \draw[red,-] (0) edge (3);
        \draw[red,-] (0) edge (4);
        \draw[red,-] (0) edge (5);
        \draw[red,-] (1) edge (3);
        \draw[red,-] (1) edge (5);
        \draw[red,-] (2) edge (3);
        \draw[red,-] (2) edge (4);
        \draw[red,-] (2) edge (5);
        \draw[red,-] (3) edge (5);
        \draw[red,-] (4) edge (5);
    \end{tikzpicture}
}

\vfill
\newpage

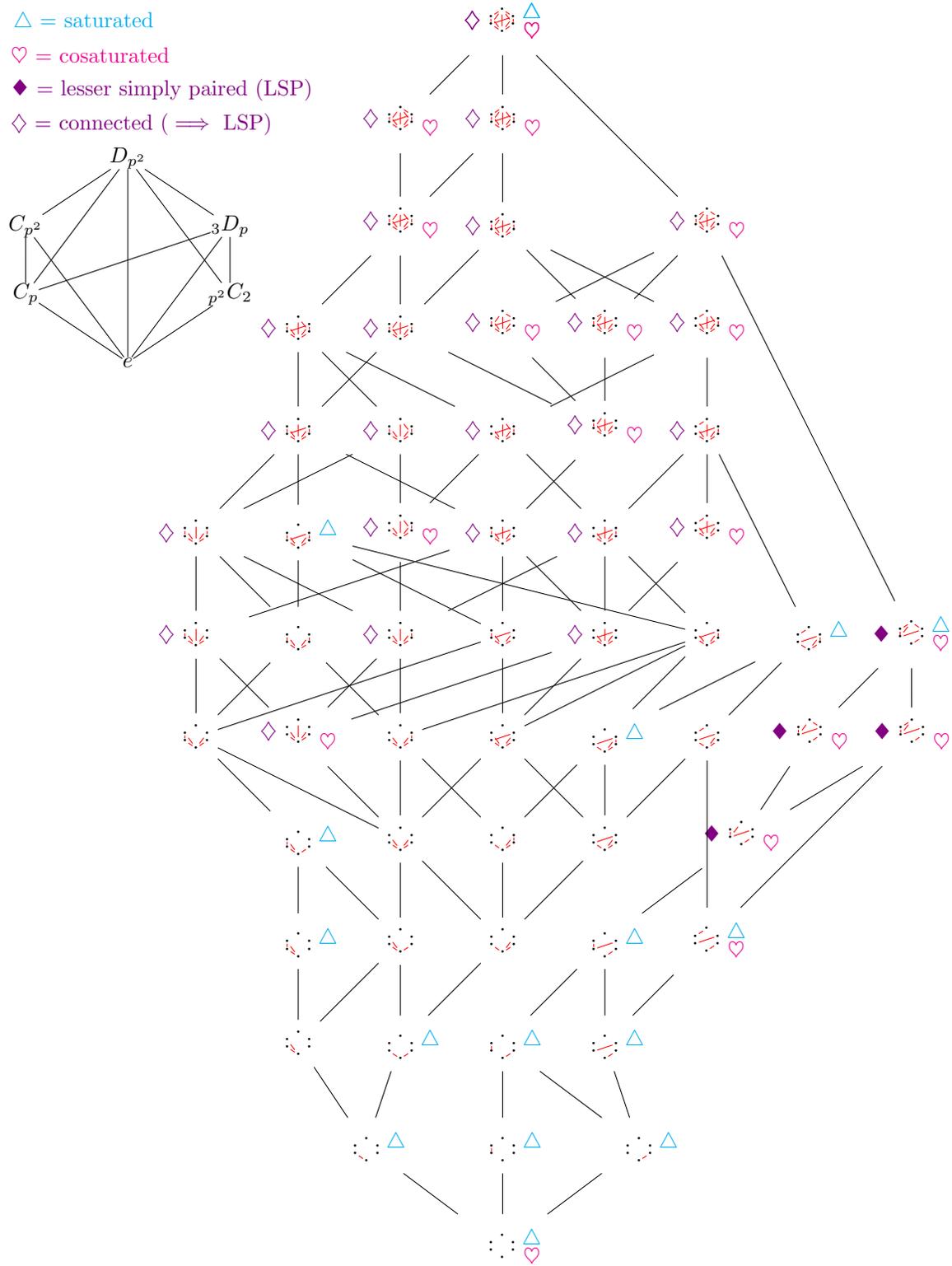
\begin{figure}[H]
	\caption{The Hasse diagram of transfer systems for the dihedral group $D_{p^2}$, for $p$ an odd prime.}
	\label{transfersystemDp2}	
	\[
	\begin{array}{c}
	\begin{tikzpicture}[scale=1.1]
		\node(00) at (0,0) {$\dpa$};
		\node(05) at (2,1.5) {$\dpf$};
            \node(01) at (-2,1.5) {$\dpb$};
		\node(02) at (0,1.5) {$\dpc$};
            \node(07) at (0,3) {$\dph$};
            \node(03) at (-3,3) {$\dpd$};
		\node(06) at (-1.5,3) {$\dpg$};
		\node(10) at (1.5,3) {$\dpba$};
            \node(04) at (-3,4.5) {$\dpe$};
		\node(08) at (-1.5,4.5) {$\dpi$};
            \node(11) at (0,4.5) {$\dpbb$};
            \node(13) at (1.5,4.5) {$\dpbd$};
            \node(12) at (3,4.5) {$\dpbc$};
            \node(09) at (-3,6) {$\dpj$};
		\node(15) at (-1.5,6) {$\dpbf$};
            \node(19) at (1.5,6) {$\dpbj$};
            \node(31) at (0,6) {$\dpdb$};
            \node(14) at (3.5,6) {$\dpbe$};
            \node(18) at (-4.5,7.5) {$\dpbi$};
            \node(33) at (-1.5,7.5) {$\dpdd$};
            \node(16) at (-3,7.5) {$\dpbg$};
		\node(23) at (0,7.5) {$\dpcd$};
            \node(38) at (1.5,7.5) {$\dpdi$};
            \node(21) at (4.5,7.5) {$\dpcb$};
            \node(22) at (3,7.5) {$\dpcc$};
            \node(17) at (6,7.5) {$\dpbh$};
            \node(36) at (-3,9) {$\dpdg$};
            \node(20) at (-4.5,9) {$\dpca$};
            \node(35) at (-1.5,9) {$\dpdf$};
		\node(24) at (0,9) {$\dpce$};
            \node(26) at (1.5,9) {$\dpcg$};
            \node(40) at (3,9) {$\dpea$};
            \node(39) at (4.5,9) {$\dpdj$};
            \node(25) at (6,9) {$\dpcf$};
            \node(47) at (-1.5,10.5) {$\dpeh$};
            \node(37) at (-4.5,10.5) {$\dpdh$};
		\node(41) at (1.5,10.5) {$\dpeb$};
            \node(27) at (0,10.5) {$\dpch$};
            \node(42) at (-3,10.5) {$\dpec$};
            \node(28) at (3,10.5) {$\dpci$};
            \node(49) at (0,12) {$\dpej$};
		\node(43) at (-3,12) {$\dped$};
            \node(29) at (1.5,12) {$\dpcj$};
            \node(48) at (-1.5,12) {$\dpei$};
            \node(44) at (3,12) {$\dpee$};
            \node(32) at (0,13.5) {$\dpdc$};
		\node(30) at (1.5,13.5) {$\dpda$};
            \node(51) at (-3,13.5) {$\dpfb$};
            \node(45) at (-1.5,13.5) {$\dpef$};
            \node(50) at (3,13.5) {$\dpfa$};    
		\node(46) at (0,15) {$\dpeg$};
            \node(52) at (-1.5,15) {$\dpfc$};
            \node(34) at (3,15) {$\dpde$};
		\node(53) at (0,16.5) {$\dpfd$};
            \node(54) at (-1.5,16.5) {$\dpfe$};
		\node(55) at (0,18) {$\dpff$};

        \node[cyan] (S) at (-6.15,18){$\triangle$ = saturated}; 
        \node[magenta] (C) at (-6.05,17.5){$\heartsuit$ = cosaturated}; 
        \node[violet] (L) at (-5,17){$\vardiamond$ = lesser simply paired (LSP)};      
        \node[violet] (L) at (-5.3,16.5){$\diamondsuit$ = connected ($\implies$ LSP)}; 

        \node[inner sep=0cm] (60) at (-5.5,13){$e$};
        \node[inner sep=0cm] (61) at (-4,14){${}_{p^2}C_2$};
        \node[inner sep=0cm] (62) at (-7,14){$C_p$};
        \node[inner sep=0cm] (63) at (-4,15){${}_3 D_p$};
        \node[inner sep=0cm] (64) at (-7,15){$C_{p^2}$};
        \node[inner sep=0cm] (65) at (-5.5,16){$D_{p^2}$};
        \draw[black,-] (60) edge (61);
        \draw[black,-] (60) edge (62);
        \draw[black,-] (60) edge (63);
        \draw[black,-] (60) edge (64);
        \draw[black,-] (60) edge (65);
        \draw[black,-] (61) edge (63);
        \draw[black,-] (61) edge (65);
        \draw[black,-] (62) edge (63);
        \draw[black,-] (62) edge (64);
        \draw[black,-] (62) edge (65);
        \draw[black,-] (63) edge (65);
        \draw[black,-] (64) edge (65);
        
		\path[-]
		(00) edge node {} (02)
		(00) edge node {} (01)
		(00) edge node {} (05)
		(02) edge node {} (07)
            (01) edge node {} (03)
		(01) edge node {} (06)
		(05) edge node {} (07)
		(05) edge node {} (10)
		(07) edge node {} (13)
            (03) edge node {} (04)
		(03) edge node {} (08)
		(06) edge node {} (08)
		(06) edge node {} (11)
		(10) edge node {} (13)
		(10) edge node {} (12)
		(04) edge node {} (09)
		(08) edge node {} (09)
		(08) edge node {} (15)
            (11) edge node {} (19)
		(11) edge node {} (31)
            (11) edge node {} (15)
		(13) edge node {} (14)
		(12) edge node {} (22)
		(12) edge node {} (17)
		(09) edge node {} (18)
		(15) edge node {} (18)
		(15) edge node {} (33)
		(15) edge node {} (16)
		(15) edge node {} (23)
		(19) edge node {} (23)
		(19) edge node {} (38)
		(19) edge node {} (22)
            (31) edge node {} (33)
            (31) edge node {} (38)
            (14) edge node {} (21)
            (14) edge node {} (17)
            (18) edge node {} (36)
            (18) edge node {} (20)
            (18) edge node {} (24)
            (33) edge node {} (36)
            (33) edge node {} (35)
            (33) edge node {} (40)
            (16) edge node {} (20)
            (16) edge node {} (35)
            (16) edge node {} (26)
            (23) edge node {} (24)
            (23) edge node {} (26)
            (23) edge node {} (40)
            (38) edge node {} (40)
            (38) edge node {} (39)
            (21) edge node {} (25)
            (22) edge node {} (39)
            (17) edge node {} (25)
            (36) edge node {} (37)
            (36) edge node {} (42)
            (20) edge node {} (37)
            (20) edge node {} (27)
            (35) edge node {} (47)
            (35) edge node {} (37)
            (35) edge node {} (41)
            (24) edge node {} (27)
            (24) edge node {} (42)
            (26) edge node {} (41)
            (26) edge node {} (27)
            (26) edge node {} (28)
            (40) edge node {} (41)
            (40) edge node {} (42)
            (39) edge node {} (44)
            (25) edge node {} (34)
            (47) edge node {} (49)
            (47) edge node {} (48)
            (37) edge node {} (43)
            (37) edge node {} (48)
            (41) edge node {} (49)
            (41) edge node {} (44)
            (27) edge node {} (43)
            (27) edge node {} (29)
            (42) edge node {} (43)
            (28) edge node {} (44)
            (49) edge node {} (51)
            (49) edge node {} (50)
            (43) edge node {} (51)
            (43) edge node {} (45)
            (29) edge node {} (32)
            (29) edge node {} (30)
            (29) edge node {} (45)
            (48) edge node {} (51)
            (44) edge node {} (50)
            (32) edge node {} (34)
            (30) edge node {} (34)
            (30) edge node {} (46)
            (51) edge node {} (52)
            (45) edge node {} (46)
            (45) edge node {} (52)
            (50) edge node {} (46)
            (34) edge node {} (55)
            (46) edge node {} (53)
            (52) edge node {} (53)
            (52) edge node {} (54)
            (53) edge node {} (55)
            (54) edge node {} (55)
		;
	\end{tikzpicture}
	\end{array}
	\]
\end{figure}


\newcommand{\dic}{
	\begin{tikzpicture}[scale=0.12,baseline=0.2mm]         \node[cyan] (S) at (4,2.5){$\triangle$};              \node[white] (L) at (-4,1.5){$\diamondsuit$};              \node[magenta] (C) at (4,0){$\heartsuit$};
	\node[inner sep=0cm] (0) at (0,0){$\cdot$};
        \node[inner sep=0cm] (1) at (-1.5,1){$\cdot$};
        \node[inner sep=0cm] (2) at (1.5,1){$\cdot$};
        \node[inner sep=0cm] (3) at (-1.5,2){$\cdot$};
        \node[inner sep=0cm] (4) at (1.5,2){$\cdot$};
        \node[inner sep=0cm] (5) at (0,3){$\cdot$};
	\end{tikzpicture}
}

\newcommand{\dicb}{
	\begin{tikzpicture}[scale=0.12,baseline=0.2mm]     
    \node[cyan] (S) at (4,2.5){$\triangle$};              
    \node[white] (L) at (-4,1.5){$\diamondsuit$};              
    \node[white] (C) at (4,0){$\heartsuit$};
	\node[inner sep=0cm] (0) at (0,0){$\cdot$};
        \node[inner sep=0cm] (1) at (-1.5,1){$\cdot$};
        \node[inner sep=0cm] (2) at (1.5,1){$\cdot$};
        \node[inner sep=0cm] (3) at (-1.5,2){$\cdot$};
        \node[inner sep=0cm] (4) at (1.5,2){$\cdot$};
        \node[inner sep=0cm] (5) at (0,3){$\cdot$};
        \draw[red,-] (0) edge (1);
	\end{tikzpicture}
}

\newcommand{\dicc}{
	\begin{tikzpicture}[scale=0.12,baseline=0.2mm]     \node[cyan] (S) at (4,2.5){$\triangle$};              \node[white] (L) at (-4,1.5){$\diamondsuit$};              \node[white] (C) at (4,0){$\heartsuit$};
	\node[inner sep=0cm] (0) at (0,0){$\cdot$};
        \node[inner sep=0cm] (1) at (-1.5,1){$\cdot$};
        \node[inner sep=0cm] (2) at (1.5,1){$\cdot$};
        \node[inner sep=0cm] (3) at (-1.5,2){$\cdot$};
        \node[inner sep=0cm] (4) at (1.5,2){$\cdot$};
        \node[inner sep=0cm] (5) at (0,3){$\cdot$};
        \draw[red,-] (0) edge (2);
	\end{tikzpicture}
}

\newcommand{\dicd}{
	\begin{tikzpicture}[scale=0.12,baseline=0.2mm]     \node[cyan] (S) at (4,2.5){$\triangle$};              \node[white] (L) at (-4,1.5){$\diamondsuit$};              \node[white] (C) at (4,0){$\heartsuit$};
	    \node[inner sep=0cm] (0) at (0,0){$\cdot$};
            \node[inner sep=0cm] (1) at (-1.5,1){$\cdot$};
            \node[inner sep=0cm] (2) at (1.5,1){$\cdot$};
            \node[inner sep=0cm] (3) at (-1.5,2){$\cdot$};
            \node[inner sep=0cm] (4) at (1.5,2){$\cdot$};
            \node[inner sep=0cm] (5) at (0,3){$\cdot$};
            \draw[red,-] (0) edge (1);
            \draw[red,-] (2) edge (4);
	\end{tikzpicture}
}

\newcommand{\dice}{
	\begin{tikzpicture}[scale=0.12,baseline=0.2mm]     \node[cyan] (S) at (4,2.5){$\triangle$};              \node[white] (L) at (-4,1.5){$\diamondsuit$};              \node[white] (C) at (4,0){$\heartsuit$};
	\node[inner sep=0cm] (0) at (0,0){$\cdot$};
        \node[inner sep=0cm] (1) at (-1.5,1){$\cdot$};
        \node[inner sep=0cm] (2) at (1.5,1){$\cdot$};
        \node[inner sep=0cm] (3) at (-1.5,2){$\cdot$};
        \node[inner sep=0cm] (4) at (1.5,2){$\cdot$};
        \node[inner sep=0cm] (5) at (0,3){$\cdot$};
        \draw[red,-] (0) edge (1);
        \draw[red,-] (0) edge (2);
	\end{tikzpicture}
}

\newcommand{\dicf}{
	\begin{tikzpicture}[scale=0.12,baseline=0.2mm]     \node[cyan] (S) at (4,2.5){$\triangle$};              \node[white] (L) at (-4,1.5){$\diamondsuit$};              \node[white] (C) at (4,0){$\heartsuit$};
	\node[inner sep=0cm] (0) at (0,0){$\cdot$};
        \node[inner sep=0cm] (1) at (-1.5,1){$\cdot$};
        \node[inner sep=0cm] (2) at (1.5,1){$\cdot$};
        \node[inner sep=0cm] (3) at (-1.5,2){$\cdot$};
        \node[inner sep=0cm] (4) at (1.5,2){$\cdot$};
        \node[inner sep=0cm] (5) at (0,3){$\cdot$};
        \draw[red,-] (0) edge (2);
        \draw[red,-] (1) edge (4);
	\end{tikzpicture}
}

\newcommand{\dicg}{
	\begin{tikzpicture}[scale=0.12,baseline=0.2mm]     
	\node[inner sep=0cm] (0) at (0,0){$\cdot$};
        \node[inner sep=0cm] (1) at (-1.5,1){$\cdot$};
        \node[inner sep=0cm] (2) at (1.5,1){$\cdot$};
        \node[inner sep=0cm] (3) at (-1.5,2){$\cdot$};
        \node[inner sep=0cm] (4) at (1.5,2){$\cdot$};
        \node[inner sep=0cm] (5) at (0,3){$\cdot$};
        \draw[red,-] (0) edge (1);
        \draw[red,-] (0) edge (2);
        \draw[red,-] (0) edge (4);
	\end{tikzpicture}
}

\newcommand{\dich}{
	\begin{tikzpicture}[scale=0.12,baseline=0.2mm]     \node[cyan] (S) at (4,2.5){$\triangle$};              \node[white] (L) at (-4,1.5){$\diamondsuit$};              \node[white] (C) at (4,0){$\heartsuit$};
	\node[inner sep=0cm] (0) at (0,0){$\cdot$};
        \node[inner sep=0cm] (1) at (-1.5,1){$\cdot$};
        \node[inner sep=0cm] (2) at (1.5,1){$\cdot$};
        \node[inner sep=0cm] (3) at (-1.5,2){$\cdot$};
        \node[inner sep=0cm] (4) at (1.5,2){$\cdot$};
        \node[inner sep=0cm] (5) at (0,3){$\cdot$};
        \draw[red,-] (1) edge (3);
	\end{tikzpicture}
}

\newcommand{\dici}{
	\begin{tikzpicture}[scale=0.12,baseline=0.2mm]     \node[cyan] (S) at (4,2.5){$\triangle$};              \node[white] (L) at (-4,1.5){$\diamondsuit$};              \node[white] (C) at (4,0){$\heartsuit$};
	\node[inner sep=0cm] (0) at (0,0){$\cdot$};
        \node[inner sep=0cm] (1) at (-1.5,1){$\cdot$};
        \node[inner sep=0cm] (2) at (1.5,1){$\cdot$};
        \node[inner sep=0cm] (3) at (-1.5,2){$\cdot$};
        \node[inner sep=0cm] (4) at (1.5,2){$\cdot$};
        \node[inner sep=0cm] (5) at (0,3){$\cdot$};
        \draw[red,-] (0) edge (2);
        \draw[red,-] (1) edge (3);
	\end{tikzpicture}
}

\newcommand{\dicj}{
	\begin{tikzpicture}[scale=0.12,baseline=0.2mm]     
	\node[inner sep=0cm] (0) at (0,0){$\cdot$};
        \node[inner sep=0cm] (1) at (-1.5,1){$\cdot$};
        \node[inner sep=0cm] (2) at (1.5,1){$\cdot$};
        \node[inner sep=0cm] (3) at (-1.5,2){$\cdot$};
        \node[inner sep=0cm] (4) at (1.5,2){$\cdot$};
        \node[inner sep=0cm] (5) at (0,3){$\cdot$};
        \draw[red,-] (0) edge (1);
        \draw[red,-] (0) edge (2);
        \draw[red,-] (0) edge (4);
        \draw[red,-] (1) edge (4);
	\end{tikzpicture}
}

\newcommand{\dicba}{
	\begin{tikzpicture}[scale=0.12,baseline=0.2mm]     
	\node[inner sep=0cm] (0) at (0,0){$\cdot$};
        \node[inner sep=0cm] (1) at (-1.5,1){$\cdot$};
        \node[inner sep=0cm] (2) at (1.5,1){$\cdot$};
        \node[inner sep=0cm] (3) at (-1.5,2){$\cdot$};
        \node[inner sep=0cm] (4) at (1.5,2){$\cdot$};
        \node[inner sep=0cm] (5) at (0,3){$\cdot$};
        \draw[red,-] (0) edge (1);
        \draw[red,-] (0) edge (2);
        \draw[red,-] (0) edge (4);
        \draw[red,-] (1) edge (4);
	\end{tikzpicture}
}

\newcommand{\dicbb}{
	\begin{tikzpicture}[scale=0.12,baseline=0.2mm]     
	\node[inner sep=0cm] (0) at (0,0){$\cdot$};
        \node[inner sep=0cm] (1) at (-1.5,1){$\cdot$};
        \node[inner sep=0cm] (2) at (1.5,1){$\cdot$};
        \node[inner sep=0cm] (3) at (-1.5,2){$\cdot$};
        \node[inner sep=0cm] (4) at (1.5,2){$\cdot$};
        \node[inner sep=0cm] (5) at (0,3){$\cdot$};
        \draw[red,-] (0) edge (1);
        \draw[red,-] (0) edge (3);
	\end{tikzpicture}
}

\newcommand{\dicbc}{
	\begin{tikzpicture}[scale=0.12,baseline=0.2mm]     \node[cyan] (S) at (4,2.5){$\triangle$};              \node[white] (L) at (-4,1.5){$\diamondsuit$};              \node[magenta] (C) at (4,0){$\heartsuit$};
	\node[inner sep=0cm] (0) at (0,0){$\cdot$};
        \node[inner sep=0cm] (1) at (-1.5,1){$\cdot$};
        \node[inner sep=0cm] (2) at (1.5,1){$\cdot$};
        \node[inner sep=0cm] (3) at (-1.5,2){$\cdot$};
        \node[inner sep=0cm] (4) at (1.5,2){$\cdot$};
        \node[inner sep=0cm] (5) at (0,3){$\cdot$};
        \draw[red,-] (1) edge (3);
        \draw[red,-] (4) edge (5);
	\end{tikzpicture}
}

\newcommand{\dicbd}{
	\begin{tikzpicture}[scale=0.12,baseline=0.2mm]     \node[cyan] (S) at (4,2.5){$\triangle$};              \node[white] (L) at (-4,1.5){$\diamondsuit$};              \node[white] (C) at (4,0){$\heartsuit$};
	\node[inner sep=0cm] (0) at (0,0){$\cdot$};
        \node[inner sep=0cm] (1) at (-1.5,1){$\cdot$};
        \node[inner sep=0cm] (2) at (1.5,1){$\cdot$};
        \node[inner sep=0cm] (3) at (-1.5,2){$\cdot$};
        \node[inner sep=0cm] (4) at (1.5,2){$\cdot$};
        \node[inner sep=0cm] (5) at (0,3){$\cdot$};
        \draw[red,-] (0) edge (2);
        \draw[red,-] (1) edge (3);
        \draw[red,-] (4) edge (5);
	\end{tikzpicture}
}

\newcommand{\dicbe}{
	\begin{tikzpicture}[scale=0.12,baseline=0.2mm]     
	 \node[inner sep=0cm] (0) at (0,0){$\cdot$};
        \node[inner sep=0cm] (1) at (-1.5,1){$\cdot$};
        \node[inner sep=0cm] (2) at (1.5,1){$\cdot$};
        \node[inner sep=0cm] (3) at (-1.5,2){$\cdot$};
        \node[inner sep=0cm] (4) at (1.5,2){$\cdot$};
        \node[inner sep=0cm] (5) at (0,3){$\cdot$};
        \draw[red,-] (0) edge (1);
        \draw[red,-] (0) edge (3);
        \draw[red,-] (2) edge (4);
	\end{tikzpicture}
}

\newcommand{\dicbf}{
	\begin{tikzpicture}[scale=0.12,baseline=0.2mm]     \node[cyan] (S) at (4,2.5){$\triangle$};              \node[white] (L) at (-4,1.5){$\diamondsuit$};              \node[white] (C) at (4,0){$\heartsuit$};
	\node[inner sep=0cm] (0) at (0,0){$\cdot$};
        \node[inner sep=0cm] (1) at (-1.5,1){$\cdot$};
        \node[inner sep=0cm] (2) at (1.5,1){$\cdot$};
        \node[inner sep=0cm] (3) at (-1.5,2){$\cdot$};
        \node[inner sep=0cm] (4) at (1.5,2){$\cdot$};
        \node[inner sep=0cm] (5) at (0,3){$\cdot$};
        \draw[red,-] (0) edge (1);
        \draw[red,-] (0) edge (2);
        \draw[red,-] (0) edge (4);
        \draw[red,-] (1) edge (4);
        \draw[red,-] (2) edge (4);
	\end{tikzpicture}
}

\newcommand{\dicbg}{
	\begin{tikzpicture}[scale=0.12,baseline=0.2mm]     
	\node[inner sep=0cm] (0) at (0,0){$\cdot$};
        \node[inner sep=0cm] (1) at (-1.5,1){$\cdot$};
        \node[inner sep=0cm] (2) at (1.5,1){$\cdot$};
        \node[inner sep=0cm] (3) at (-1.5,2){$\cdot$};
        \node[inner sep=0cm] (4) at (1.5,2){$\cdot$};
        \node[inner sep=0cm] (5) at (0,3){$\cdot$};
        \draw[red,-] (0) edge (1);
        \draw[red,-] (0) edge (2);
        \draw[red,-] (0) edge (3);
	\end{tikzpicture}
}

\newcommand{\dicbh}{
	\begin{tikzpicture}[scale=0.12,baseline=0.2mm]     \node[cyan] (S) at (4,2.5){$\triangle$};              \node[white] (L) at (-4,1.5){$\diamondsuit$};              \node[white] (C) at (4,0){$\heartsuit$};
	\node[inner sep=0cm] (0) at (0,0){$\cdot$};
        \node[inner sep=0cm] (1) at (-1.5,1){$\cdot$};
        \node[inner sep=0cm] (2) at (1.5,1){$\cdot$};
        \node[inner sep=0cm] (3) at (-1.5,2){$\cdot$};
        \node[inner sep=0cm] (4) at (1.5,2){$\cdot$};
        \node[inner sep=0cm] (5) at (0,3){$\cdot$};
        \draw[red,-] (0) edge (2);
        \draw[red,-] (1) edge (3);
        \draw[red,-] (1) edge (4);
	\end{tikzpicture}
}

\newcommand{\dicbi}{
	\begin{tikzpicture}[scale=0.12,baseline=0.2mm]     \node[white] (S) at (4,2.5){$\triangle$};              \node[violet] (L) at (-4,1.5){$\vardiamond$};              \node[magenta] (C) at (4,0){$\heartsuit$};
	\node[inner sep=0cm] (0) at (0,0){$\cdot$};
        \node[inner sep=0cm] (1) at (-1.5,1){$\cdot$};
        \node[inner sep=0cm] (2) at (1.5,1){$\cdot$};
        \node[inner sep=0cm] (3) at (-1.5,2){$\cdot$};
        \node[inner sep=0cm] (4) at (1.5,2){$\cdot$};
        \node[inner sep=0cm] (5) at (0,3){$\cdot$};
        \draw[red,-] (0) edge (1);
        \draw[red,-] (0) edge (3);
        \draw[red,-] (2) edge (4);
        \draw[red,-] (2) edge (5);
	\end{tikzpicture}
}

\newcommand{\dicbj}{
	\begin{tikzpicture}[scale=0.12,baseline=0.2mm]     \node[white] (S) at (4,2.5){$\triangle$};              \node[violet] (L) at (-4,1.5){$\vardiamond$};              \node[magenta] (C) at (4,0){$\heartsuit$};
	\node[inner sep=0cm] (0) at (0,0){$\cdot$};
        \node[inner sep=0cm] (1) at (-1.5,1){$\cdot$};
        \node[inner sep=0cm] (2) at (1.5,1){$\cdot$};
        \node[inner sep=0cm] (3) at (-1.5,2){$\cdot$};
        \node[inner sep=0cm] (4) at (1.5,2){$\cdot$};
        \node[inner sep=0cm] (5) at (0,3){$\cdot$};
        \draw[red,-] (0) edge (2);
        \draw[red,-] (1) edge (3);
        \draw[red,-] (1) edge (4);
        \draw[red,-] (1) edge (5);
	\end{tikzpicture}
}

\newcommand{\dicca}{
	\begin{tikzpicture}[scale=0.12,baseline=0.2mm]     
	\node[inner sep=0cm] (0) at (0,0){$\cdot$};
        \node[inner sep=0cm] (1) at (-1.5,1){$\cdot$};
        \node[inner sep=0cm] (2) at (1.5,1){$\cdot$};
        \node[inner sep=0cm] (3) at (-1.5,2){$\cdot$};
        \node[inner sep=0cm] (4) at (1.5,2){$\cdot$};
        \node[inner sep=0cm] (5) at (0,3){$\cdot$};
        \draw[red,-] (0) edge (1);
        \draw[red,-] (0) edge (2);
        \draw[red,-] (0) edge (3);
        \draw[red,-] (0) edge (4);
	\end{tikzpicture}
}

\newcommand{\diccb}{
	\begin{tikzpicture}[scale=0.12,baseline=0.2mm]     \node[white] (S) at (4,2.5){$\triangle$};              \node[violet] (L) at (-4,1.5){$\diamondsuit$};              \node[magenta] (C) at (4,0){$\heartsuit$};
	\node[inner sep=0cm] (0) at (0,0){$\cdot$};
        \node[inner sep=0cm] (1) at (-1.5,1){$\cdot$};
        \node[inner sep=0cm] (2) at (1.5,1){$\cdot$};
        \node[inner sep=0cm] (3) at (-1.5,2){$\cdot$};
        \node[inner sep=0cm] (4) at (1.5,2){$\cdot$};
        \node[inner sep=0cm] (5) at (0,3){$\cdot$};
        \draw[red,-] (0) edge (1);
        \draw[red,-] (0) edge (2);
        \draw[red,-] (0) edge (3);
        \draw[red,-] (0) edge (4);
        \draw[red,-] (0) edge (5);
	\end{tikzpicture}
}

\newcommand{\diccc}{
	\begin{tikzpicture}[scale=0.12,baseline=0.2mm]     \node[cyan] (S) at (4,2.5){$\triangle$};              \node[white] (L) at (-4,1.5){$\diamondsuit$};              \node[white] (C) at (4,0){$\heartsuit$};
	\node[inner sep=0cm] (0) at (0,0){$\cdot$};
        \node[inner sep=0cm] (1) at (-1.5,1){$\cdot$};
        \node[inner sep=0cm] (2) at (1.5,1){$\cdot$};
        \node[inner sep=0cm] (3) at (-1.5,2){$\cdot$};
        \node[inner sep=0cm] (4) at (1.5,2){$\cdot$};
        \node[inner sep=0cm] (5) at (0,3){$\cdot$};
        \draw[red,-] (0) edge (1);
        \draw[red,-] (0) edge (3);
        \draw[red,-] (1) edge (3);
	\end{tikzpicture}
}

\newcommand{\diccd}{
	\begin{tikzpicture}[scale=0.12,baseline=0.2mm]     
	\node[inner sep=0cm] (0) at (0,0){$\cdot$};
        \node[inner sep=0cm] (1) at (-1.5,1){$\cdot$};
        \node[inner sep=0cm] (2) at (1.5,1){$\cdot$};
        \node[inner sep=0cm] (3) at (-1.5,2){$\cdot$};
        \node[inner sep=0cm] (4) at (1.5,2){$\cdot$};
        \node[inner sep=0cm] (5) at (0,3){$\cdot$};
        \draw[red,-] (0) edge (1);
        \draw[red,-] (0) edge (2);
        \draw[red,-] (0) edge (3);
        \draw[red,-] (0) edge (4);
        \draw[red,-] (1) edge (4);
	\end{tikzpicture}
}

\newcommand{\dicce}{
	\begin{tikzpicture}[scale=0.12,baseline=0.2mm]     \node[white] (S) at (4,2.5){$\triangle$};              \node[violet] (L) at (-4,1.5){$\vardiamond$};              \node[magenta] (C) at (4,0){$\heartsuit$};
	\node[inner sep=0cm] (0) at (0,0){$\cdot$};
        \node[inner sep=0cm] (1) at (-1.5,1){$\cdot$};
        \node[inner sep=0cm] (2) at (1.5,1){$\cdot$};
        \node[inner sep=0cm] (3) at (-1.5,2){$\cdot$};
        \node[inner sep=0cm] (4) at (1.5,2){$\cdot$};
        \node[inner sep=0cm] (5) at (0,3){$\cdot$};
        \draw[red,-] (0) edge (2);
        \draw[red,-] (1) edge (3);
        \draw[red,-] (1) edge (4);
        \draw[red,-] (1) edge (5);
        \draw[red,-] (4) edge (5);
	\end{tikzpicture}
}

\newcommand{\diccf}{
	\begin{tikzpicture}[scale=0.12,baseline=0.2mm]     
	\node[inner sep=0cm] (0) at (0,0){$\cdot$};
        \node[inner sep=0cm] (1) at (-1.5,1){$\cdot$};
        \node[inner sep=0cm] (2) at (1.5,1){$\cdot$};
        \node[inner sep=0cm] (3) at (-1.5,2){$\cdot$};
        \node[inner sep=0cm] (4) at (1.5,2){$\cdot$};
        \node[inner sep=0cm] (5) at (0,3){$\cdot$};
        \draw[red,-] (0) edge (1);
        \draw[red,-] (0) edge (2);
        \draw[red,-] (0) edge (3);
        \draw[red,-] (0) edge (4);
        \draw[red,-] (2) edge (4);
	\end{tikzpicture}
}

\newcommand{\diccg}{
	\begin{tikzpicture}[scale=0.12,baseline=0.2mm]     \node[cyan] (S) at (4,2.5){$\triangle$};              \node[white] (L) at (-4,1.5){$\diamondsuit$};              \node[white] (C) at (4,0){$\heartsuit$};
	\node[inner sep=0cm] (0) at (0,0){$\cdot$};
        \node[inner sep=0cm] (1) at (-1.5,1){$\cdot$};
        \node[inner sep=0cm] (2) at (1.5,1){$\cdot$};
        \node[inner sep=0cm] (3) at (-1.5,2){$\cdot$};
        \node[inner sep=0cm] (4) at (1.5,2){$\cdot$};
        \node[inner sep=0cm] (5) at (0,3){$\cdot$};
        \draw[red,-] (0) edge (1);
        \draw[red,-] (0) edge (3);
        \draw[red,-] (1) edge (3);
        \draw[red,-] (2) edge (4);
	\end{tikzpicture}
}

\newcommand{\dicch}{
	\begin{tikzpicture}[scale=0.12,baseline=0.2mm]     
	\node[inner sep=0cm] (0) at (0,0){$\cdot$};
        \node[inner sep=0cm] (1) at (-1.5,1){$\cdot$};
        \node[inner sep=0cm] (2) at (1.5,1){$\cdot$};
        \node[inner sep=0cm] (3) at (-1.5,2){$\cdot$};
        \node[inner sep=0cm] (4) at (1.5,2){$\cdot$};
        \node[inner sep=0cm] (5) at (0,3){$\cdot$};
        \draw[red,-] (0) edge (1);
        \draw[red,-] (0) edge (2);
        \draw[red,-] (0) edge (3);
        \draw[red,-] (0) edge (4);
        \draw[red,-] (1) edge (4);
        \draw[red,-] (2) edge (4);
	\end{tikzpicture}
}

\newcommand{\dicci}{
	\begin{tikzpicture}[scale=0.12,baseline=0.2mm]     \node[cyan] (S) at (4,2.5){$\triangle$};              \node[white] (L) at (-4,1.5){$\diamondsuit$};              \node[white] (C) at (4,0){$\heartsuit$};
	\node[inner sep=0cm] (0) at (0,0){$\cdot$};
        \node[inner sep=0cm] (1) at (-1.5,1){$\cdot$};
        \node[inner sep=0cm] (2) at (1.5,1){$\cdot$};
        \node[inner sep=0cm] (3) at (-1.5,2){$\cdot$};
        \node[inner sep=0cm] (4) at (1.5,2){$\cdot$};
        \node[inner sep=0cm] (5) at (0,3){$\cdot$};
        \draw[red,-] (0) edge (1);
        \draw[red,-] (0) edge (3);
        \draw[red,-] (1) edge (3);
        \draw[red,-] (4) edge (5);
	\end{tikzpicture}
}

\newcommand{\diccj}{
	\begin{tikzpicture}[scale=0.12,baseline=0.2mm]     
    \node[violet] (L) at (-4,1.5){$\diamondsuit$};
    \node[white] (W) at (4,1.5){$\diamondsuit$};
	\node[inner sep=0cm] (0) at (0,0){$\cdot$};
        \node[inner sep=0cm] (1) at (-1.5,1){$\cdot$};
        \node[inner sep=0cm] (2) at (1.5,1){$\cdot$};
        \node[inner sep=0cm] (3) at (-1.5,2){$\cdot$};
        \node[inner sep=0cm] (4) at (1.5,2){$\cdot$};
        \node[inner sep=0cm] (5) at (0,3){$\cdot$};
        \draw[red,-] (0) edge (1);
        \draw[red,-] (0) edge (2);
        \draw[red,-] (0) edge (3);
        \draw[red,-] (0) edge (4);
        \draw[red,-] (0) edge (5);
        \draw[red,-] (2) edge (4);
	\end{tikzpicture}
}

\newcommand{\dicda}{
	\begin{tikzpicture}[scale=0.12,baseline=0.2mm]     \node[cyan] (S) at (4,2.5){$\triangle$};              \node[white] (L) at (-4,1.5){$\diamondsuit$};              \node[white] (C) at (4,0){$\heartsuit$};
	\node[inner sep=0cm] (0) at (0,0){$\cdot$};
        \node[inner sep=0cm] (1) at (-1.5,1){$\cdot$};
        \node[inner sep=0cm] (2) at (1.5,1){$\cdot$};
        \node[inner sep=0cm] (3) at (-1.5,2){$\cdot$};
        \node[inner sep=0cm] (4) at (1.5,2){$\cdot$};
        \node[inner sep=0cm] (5) at (0,3){$\cdot$};
        \draw[red,-] (0) edge (1);
        \draw[red,-] (0) edge (2);
        \draw[red,-] (0) edge (3);
        \draw[red,-] (1) edge (3);
	\end{tikzpicture}
}

\newcommand{\dicdb}{
	\begin{tikzpicture}[scale=0.12,baseline=0.2mm]     
    \node[violet] (L) at (-4,1.5){$\diamondsuit$};
    \node[white] (W) at (4,1.5){$\diamondsuit$};
	\node[inner sep=0cm] (0) at (0,0){$\cdot$};
        \node[inner sep=0cm] (1) at (-1.5,1){$\cdot$};
        \node[inner sep=0cm] (2) at (1.5,1){$\cdot$};
        \node[inner sep=0cm] (3) at (-1.5,2){$\cdot$};
        \node[inner sep=0cm] (4) at (1.5,2){$\cdot$};
        \node[inner sep=0cm] (5) at (0,3){$\cdot$};
        \draw[red,-] (0) edge (1);
        \draw[red,-] (0) edge (2);
        \draw[red,-] (0) edge (3);
        \draw[red,-] (0) edge (4);
        \draw[red,-] (0) edge (5);
        \draw[red,-] (1) edge (4);
	\end{tikzpicture}
}

\newcommand{\dicdc}{
	\begin{tikzpicture}[scale=0.12,baseline=0.2mm]     
    \node[violet] (L) at (-4,1.5){$\diamondsuit$};
    \node[white] (W) at (4,1.5){$\diamondsuit$};
	\node[inner sep=0cm] (0) at (0,0){$\cdot$};
        \node[inner sep=0cm] (1) at (-1.5,1){$\cdot$};
        \node[inner sep=0cm] (2) at (1.5,1){$\cdot$};
        \node[inner sep=0cm] (3) at (-1.5,2){$\cdot$};
        \node[inner sep=0cm] (4) at (1.5,2){$\cdot$};
        \node[inner sep=0cm] (5) at (0,3){$\cdot$};
        \draw[red,-] (0) edge (1);
        \draw[red,-] (0) edge (2);
        \draw[red,-] (0) edge (3);
        \draw[red,-] (0) edge (4);
        \draw[red,-] (0) edge (5);
        \draw[red,-] (1) edge (4);
        \draw[red,-] (2) edge (4);
	\end{tikzpicture}
}

\newcommand{\dicdd}{
	\begin{tikzpicture}[scale=0.12,baseline=0.2mm]     
    \node[violet] (L) at (-4,1.5){$\vardiamond$};
    \node[white] (W) at (4,1.5){$\diamondsuit$};
	\node[inner sep=0cm] (0) at (0,0){$\cdot$};
        \node[inner sep=0cm] (1) at (-1.5,1){$\cdot$};
        \node[inner sep=0cm] (2) at (1.5,1){$\cdot$};
        \node[inner sep=0cm] (3) at (-1.5,2){$\cdot$};
        \node[inner sep=0cm] (4) at (1.5,2){$\cdot$};
        \node[inner sep=0cm] (5) at (0,3){$\cdot$};
        \draw[red,-] (0) edge (1);
        \draw[red,-] (0) edge (3);
        \draw[red,-] (1) edge (3);
        \draw[red,-] (2) edge (4);
        \draw[red,-] (2) edge (5);
	\end{tikzpicture}
}

\newcommand{\dicde}{
	\begin{tikzpicture}[scale=0.12,baseline=0.2mm]     \node[white] (S) at (4,2.5){$\triangle$};              \node[violet] (L) at (-4,1.5){$\vardiamond$};              \node[magenta] (C) at (4,0){$\heartsuit$};
	\node[inner sep=0cm] (0) at (0,0){$\cdot$};
        \node[inner sep=0cm] (1) at (-1.5,1){$\cdot$};
        \node[inner sep=0cm] (2) at (1.5,1){$\cdot$};
        \node[inner sep=0cm] (3) at (-1.5,2){$\cdot$};
        \node[inner sep=0cm] (4) at (1.5,2){$\cdot$};
        \node[inner sep=0cm] (5) at (0,3){$\cdot$};
        \draw[red,-] (0) edge (2);
        \draw[red,-] (1) edge (3);
        \draw[red,-] (1) edge (4);
        \draw[red,-] (1) edge (5);
        \draw[red,-] (3) edge (5);
	\end{tikzpicture}
}

\newcommand{\dicdf}{
	\begin{tikzpicture}[scale=0.12,baseline=0.2mm]     \node[cyan] (S) at (4,2.5){$\triangle$};              \node[white] (L) at (-4,1.5){$\diamondsuit$};              \node[white] (C) at (4,0){$\heartsuit$};
	\node[inner sep=0cm] (0) at (0,0){$\cdot$};
        \node[inner sep=0cm] (1) at (-1.5,1){$\cdot$};
        \node[inner sep=0cm] (2) at (1.5,1){$\cdot$};
        \node[inner sep=0cm] (3) at (-1.5,2){$\cdot$};
        \node[inner sep=0cm] (4) at (1.5,2){$\cdot$};
        \node[inner sep=0cm] (5) at (0,3){$\cdot$};
        \draw[red,-] (0) edge (1);
        \draw[red,-] (0) edge (2);
        \draw[red,-] (0) edge (3);
        \draw[red,-] (1) edge (3);
        \draw[red,-] (4) edge (5);
	\end{tikzpicture}
}

\newcommand{\dicdg}{
	\begin{tikzpicture}[scale=0.12,baseline=0.2mm]     \node[white] (S) at (4,2.5){$\triangle$};              \node[violet] (L) at (-4,1.5){$\diamondsuit$};              \node[magenta] (C) at (4,0){$\heartsuit$};
	\node[inner sep=0cm] (0) at (0,0){$\cdot$};
        \node[inner sep=0cm] (1) at (-1.5,1){$\cdot$};
        \node[inner sep=0cm] (2) at (1.5,1){$\cdot$};
        \node[inner sep=0cm] (3) at (-1.5,2){$\cdot$};
        \node[inner sep=0cm] (4) at (1.5,2){$\cdot$};
        \node[inner sep=0cm] (5) at (0,3){$\cdot$};
        \draw[red,-] (0) edge (1);
        \draw[red,-] (0) edge (2);
        \draw[red,-] (0) edge (3);
        \draw[red,-] (0) edge (4);
        \draw[red,-] (0) edge (5);
        \draw[red,-] (2) edge (4);
        \draw[red,-] (2) edge (5);
	\end{tikzpicture}
}

\newcommand{\dicdh}{
	\begin{tikzpicture}[scale=0.12,baseline=0.2mm]     
	\node[inner sep=0cm] (0) at (0,0){$\cdot$};
        \node[inner sep=0cm] (1) at (-1.5,1){$\cdot$};
        \node[inner sep=0cm] (2) at (1.5,1){$\cdot$};
        \node[inner sep=0cm] (3) at (-1.5,2){$\cdot$};
        \node[inner sep=0cm] (4) at (1.5,2){$\cdot$};
        \node[inner sep=0cm] (5) at (0,3){$\cdot$};
        \draw[red,-] (0) edge (1);
        \draw[red,-] (0) edge (2);
        \draw[red,-] (0) edge (3);
        \draw[red,-] (0) edge (4);
        \draw[red,-] (1) edge (3);
	\end{tikzpicture}
}

\newcommand{\dicdi}{
	\begin{tikzpicture}[scale=0.12,baseline=0.2mm]     \node[cyan] (S) at (4,2.5){$\triangle$};              \node[violet] (L) at (-4,1.5){$\vardiamond$};              \node[magenta] (C) at (4,0){$\heartsuit$};
	\node[inner sep=0cm] (0) at (0,0){$\cdot$};
        \node[inner sep=0cm] (1) at (-1.5,1){$\cdot$};
        \node[inner sep=0cm] (2) at (1.5,1){$\cdot$};
        \node[inner sep=0cm] (3) at (-1.5,2){$\cdot$};
        \node[inner sep=0cm] (4) at (1.5,2){$\cdot$};
        \node[inner sep=0cm] (5) at (0,3){$\cdot$};
        \draw[red,-] (0) edge (2);
        \draw[red,-] (1) edge (3);
        \draw[red,-] (1) edge (4);
        \draw[red,-] (1) edge (5);
        \draw[red,-] (3) edge (5);
        \draw[red,-] (4) edge (5);
	\end{tikzpicture}
}

\newcommand{\dicdj}{
	\begin{tikzpicture}[scale=0.12,baseline=0.2mm]     
    \node[violet] (L) at (-4,1.5){$\diamondsuit$};
    \node[white] (W) at (4,1.5){$\diamondsuit$};
	\node[inner sep=0cm] (0) at (0,0){$\cdot$};
        \node[inner sep=0cm] (1) at (-1.5,1){$\cdot$};
        \node[inner sep=0cm] (2) at (1.5,1){$\cdot$};
        \node[inner sep=0cm] (3) at (-1.5,2){$\cdot$};
        \node[inner sep=0cm] (4) at (1.5,2){$\cdot$};
        \node[inner sep=0cm] (5) at (0,3){$\cdot$};
        \draw[red,-] (0) edge (1);
        \draw[red,-] (0) edge (2);
        \draw[red,-] (0) edge (3);
        \draw[red,-] (0) edge (4);
        \draw[red,-] (0) edge (5);
        \draw[red,-] (1) edge (4);
        \draw[red,-] (2) edge (4);
        \draw[red,-] (2) edge (5);
	\end{tikzpicture}
}

\newcommand{\dicea}{
	\begin{tikzpicture}[scale=0.12,baseline=0.2mm]     
	\node[inner sep=0cm] (0) at (0,0){$\cdot$};
        \node[inner sep=0cm] (1) at (-1.5,1){$\cdot$};
        \node[inner sep=0cm] (2) at (1.5,1){$\cdot$};
        \node[inner sep=0cm] (3) at (-1.5,2){$\cdot$};
        \node[inner sep=0cm] (4) at (1.5,2){$\cdot$};
        \node[inner sep=0cm] (5) at (0,3){$\cdot$};
        \draw[red,-] (0) edge (1);
        \draw[red,-] (0) edge (2);
        \draw[red,-] (0) edge (3);
        \draw[red,-] (0) edge (4);
        \draw[red,-] (1) edge (3);
        \draw[red,-] (1) edge (4);
	\end{tikzpicture}
}

\newcommand{\diceb}{
	\begin{tikzpicture}[scale=0.12,baseline=0.2mm]     
	\node[inner sep=0cm] (0) at (0,0){$\cdot$};
        \node[inner sep=0cm] (1) at (-1.5,1){$\cdot$};
        \node[inner sep=0cm] (2) at (1.5,1){$\cdot$};
        \node[inner sep=0cm] (3) at (-1.5,2){$\cdot$};
        \node[inner sep=0cm] (4) at (1.5,2){$\cdot$};
        \node[inner sep=0cm] (5) at (0,3){$\cdot$};
        \draw[red,-] (0) edge (1);
        \draw[red,-] (0) edge (2);
        \draw[red,-] (0) edge (3);
        \draw[red,-] (0) edge (4);
        \draw[red,-] (1) edge (3);
        \draw[red,-] (2) edge (4);
	\end{tikzpicture}
}

\newcommand{\dicec}{
	\begin{tikzpicture}[scale=0.12,baseline=0.2mm]     
    \node[violet] (L) at (-4,1.5){$\diamondsuit$};
    \node[white] (W) at (4,1.5){$\diamondsuit$};
	\node[inner sep=0cm] (0) at (0,0){$\cdot$};
        \node[inner sep=0cm] (1) at (-1.5,1){$\cdot$};
        \node[inner sep=0cm] (2) at (1.5,1){$\cdot$};
        \node[inner sep=0cm] (3) at (-1.5,2){$\cdot$};
        \node[inner sep=0cm] (4) at (1.5,2){$\cdot$};
        \node[inner sep=0cm] (5) at (0,3){$\cdot$};
        \draw[red,-] (0) edge (1);
        \draw[red,-] (0) edge (2);
        \draw[red,-] (0) edge (3);
        \draw[red,-] (0) edge (4);
        \draw[red,-] (0) edge (5);
        \draw[red,-] (1) edge (3);
	\end{tikzpicture}
}

\newcommand{\diced}{
	\begin{tikzpicture}[scale=0.12,baseline=0.2mm]     \node[cyan] (S) at (4,2.5){$\triangle$};              \node[violet] (L) at (-4,1.5){$\vardiamond$};              \node[magenta] (C) at (4,0){$\heartsuit$};
	\node[inner sep=0cm] (0) at (0,0){$\cdot$};
        \node[inner sep=0cm] (1) at (-1.5,1){$\cdot$};
        \node[inner sep=0cm] (2) at (1.5,1){$\cdot$};
        \node[inner sep=0cm] (3) at (-1.5,2){$\cdot$};
        \node[inner sep=0cm] (4) at (1.5,2){$\cdot$};
        \node[inner sep=0cm] (5) at (0,3){$\cdot$};
        \draw[red,-] (0) edge (1);
        \draw[red,-] (0) edge (3);
        \draw[red,-] (1) edge (3);
        \draw[red,-] (2) edge (4);
        \draw[red,-] (2) edge (5);
        \draw[red,-] (4) edge (5);
	\end{tikzpicture}
}

\newcommand{\dicee}{
	\begin{tikzpicture}[scale=0.12,baseline=0.2mm]     \node[cyan] (S) at (4,2.5){$\triangle$};              \node[white] (L) at (-4,1.5){$\diamondsuit$};              \node[white] (C) at (4,0){$\heartsuit$};
	\node[inner sep=0cm] (0) at (0,0){$\cdot$};
        \node[inner sep=0cm] (1) at (-1.5,1){$\cdot$};
        \node[inner sep=0cm] (2) at (1.5,1){$\cdot$};
        \node[inner sep=0cm] (3) at (-1.5,2){$\cdot$};
        \node[inner sep=0cm] (4) at (1.5,2){$\cdot$};
        \node[inner sep=0cm] (5) at (0,3){$\cdot$};
        \draw[red,-] (0) edge (1);
        \draw[red,-] (0) edge (2);
        \draw[red,-] (0) edge (3);
        \draw[red,-] (0) edge (4);
        \draw[red,-] (1) edge (3);
        \draw[red,-] (1) edge (4);
        \draw[red,-] (2) edge (4);
	\end{tikzpicture}
}

\newcommand{\dicef}{
	\begin{tikzpicture}[scale=0.12,baseline=0.2mm]     \node[white] (S) at (4,2.5){$\triangle$};              \node[violet] (L) at (-4,1.5){$\diamondsuit$};              \node[magenta] (C) at (4,0){$\heartsuit$};
	\node[inner sep=0cm] (0) at (0,0){$\cdot$};
        \node[inner sep=0cm] (1) at (-1.5,1){$\cdot$};
        \node[inner sep=0cm] (2) at (1.5,1){$\cdot$};
        \node[inner sep=0cm] (3) at (-1.5,2){$\cdot$};
        \node[inner sep=0cm] (4) at (1.5,2){$\cdot$};
        \node[inner sep=0cm] (5) at (0,3){$\cdot$};
        \draw[red,-] (0) edge (1);
        \draw[red,-] (0) edge (2);
        \draw[red,-] (0) edge (3);
        \draw[red,-] (0) edge (4);
        \draw[red,-] (0) edge (5);
        \draw[red,-] (1) edge (3);
        \draw[red,-] (4) edge (5);
	\end{tikzpicture}
}

\newcommand{\diceg}{
	\begin{tikzpicture}[scale=0.12,baseline=0.2mm]     
    \node[violet] (L) at (-4,1.5){$\diamondsuit$};
    \node[white] (W) at (4,1.5){$\diamondsuit$};
	\node[inner sep=0cm] (0) at (0,0){$\cdot$};
        \node[inner sep=0cm] (1) at (-1.5,1){$\cdot$};
        \node[inner sep=0cm] (2) at (1.5,1){$\cdot$};
        \node[inner sep=0cm] (3) at (-1.5,2){$\cdot$};
        \node[inner sep=0cm] (4) at (1.5,2){$\cdot$};
        \node[inner sep=0cm] (5) at (0,3){$\cdot$};
        \draw[red,-] (0) edge (1);
        \draw[red,-] (0) edge (2);
        \draw[red,-] (0) edge (3);
        \draw[red,-] (0) edge (4);
        \draw[red,-] (0) edge (5);
        \draw[red,-] (1) edge (3);
        \draw[red,-] (1) edge (4);
	\end{tikzpicture}
}

\newcommand{\diceh}{
	\begin{tikzpicture}[scale=0.12,baseline=0.2mm]     
    \node[violet] (L) at (-4,1.5){$\diamondsuit$};
    \node[white] (W) at (4,1.5){$\diamondsuit$};
	\node[inner sep=0cm] (0) at (0,0){$\cdot$};
        \node[inner sep=0cm] (1) at (-1.5,1){$\cdot$};
        \node[inner sep=0cm] (2) at (1.5,1){$\cdot$};
        \node[inner sep=0cm] (3) at (-1.5,2){$\cdot$};
        \node[inner sep=0cm] (4) at (1.5,2){$\cdot$};
        \node[inner sep=0cm] (5) at (0,3){$\cdot$};
        \draw[red,-] (0) edge (1);
        \draw[red,-] (0) edge (2);
        \draw[red,-] (0) edge (3);
        \draw[red,-] (0) edge (4);
        \draw[red,-] (0) edge (5);
        \draw[red,-] (1) edge (3);
        \draw[red,-] (2) edge (4);
	\end{tikzpicture}
}

\newcommand{\dicei}{
	\begin{tikzpicture}[scale=0.12,baseline=0.2mm]     \node[white] (S) at (4,2.5){$\triangle$};              \node[violet] (L) at (-4,1.5){$\diamondsuit$};              \node[magenta] (C) at (4,0){$\heartsuit$};
	\node[inner sep=0cm] (0) at (0,0){$\cdot$};
        \node[inner sep=0cm] (1) at (-1.5,1){$\cdot$};
        \node[inner sep=0cm] (2) at (1.5,1){$\cdot$};
        \node[inner sep=0cm] (3) at (-1.5,2){$\cdot$};
        \node[inner sep=0cm] (4) at (1.5,2){$\cdot$};
        \node[inner sep=0cm] (5) at (0,3){$\cdot$};
        \draw[red,-] (0) edge (1);
        \draw[red,-] (0) edge (2);
        \draw[red,-] (0) edge (3);
        \draw[red,-] (0) edge (4);
        \draw[red,-] (0) edge (5);
        \draw[red,-] (1) edge (3);
        \draw[red,-] (1) edge (4);
        \draw[red,-] (1) edge (5);
	\end{tikzpicture}
}

\newcommand{\dicej}{
	\begin{tikzpicture}[scale=0.12,baseline=0.2mm]     
    \node[violet] (L) at (-4,1.5){$\diamondsuit$};
    \node[white] (W) at (4,1.5){$\diamondsuit$};
    
	\node[inner sep=0cm] (0) at (0,0){$\cdot$};
        \node[inner sep=0cm] (1) at (-1.5,1){$\cdot$};
        \node[inner sep=0cm] (2) at (1.5,1){$\cdot$};
        \node[inner sep=0cm] (3) at (-1.5,2){$\cdot$};
        \node[inner sep=0cm] (4) at (1.5,2){$\cdot$};
        \node[inner sep=0cm] (5) at (0,3){$\cdot$};
        \draw[red,-] (0) edge (1);
        \draw[red,-] (0) edge (2);
        \draw[red,-] (0) edge (3);
        \draw[red,-] (0) edge (4);
        \draw[red,-] (0) edge (5);
        \draw[red,-] (1) edge (3);
        \draw[red,-] (1) edge (4);
        \draw[red,-] (2) edge (4);
	\end{tikzpicture}
}

\newcommand{\dicfa}{
	\begin{tikzpicture}[scale=0.12,baseline=0.2mm]     
    \node[violet] (L) at (-4,1.5){$\diamondsuit$};
    \node[white] (W) at (4,1.5){$\diamondsuit$};
	\node[inner sep=0cm] (0) at (0,0){$\cdot$};
        \node[inner sep=0cm] (1) at (-1.5,1){$\cdot$};
        \node[inner sep=0cm] (2) at (1.5,1){$\cdot$};
        \node[inner sep=0cm] (3) at (-1.5,2){$\cdot$};
        \node[inner sep=0cm] (4) at (1.5,2){$\cdot$};
        \node[inner sep=0cm] (5) at (0,3){$\cdot$};
        \draw[red,-] (0) edge (1);
        \draw[red,-] (0) edge (2);
        \draw[red,-] (0) edge (3);
        \draw[red,-] (0) edge (4);
        \draw[red,-] (0) edge (5);
        \draw[red,-] (1) edge (3);
        \draw[red,-] (2) edge (4);
        \draw[red,-] (2) edge (5);
	\end{tikzpicture}
}

\newcommand{\dicfb}{
	\begin{tikzpicture}[scale=0.12,baseline=0.2mm]     \node[white] (S) at (4,2.5){$\triangle$};              \node[violet] (L) at (-4,1.5){$\diamondsuit$};              \node[magenta] (C) at (4,0){$\heartsuit$};
	\node[inner sep=0cm] (0) at (0,0){$\cdot$};
        \node[inner sep=0cm] (1) at (-1.5,1){$\cdot$};
        \node[inner sep=0cm] (2) at (1.5,1){$\cdot$};
        \node[inner sep=0cm] (3) at (-1.5,2){$\cdot$};
        \node[inner sep=0cm] (4) at (1.5,2){$\cdot$};
        \node[inner sep=0cm] (5) at (0,3){$\cdot$};
        \draw[red,-] (0) edge (1);
        \draw[red,-] (0) edge (2);
        \draw[red,-] (0) edge (3);
        \draw[red,-] (0) edge (4);
        \draw[red,-] (0) edge (5);
        \draw[red,-] (1) edge (3);
        \draw[red,-] (1) edge (4);
        \draw[red,-] (1) edge (5);
        \draw[red,-] (4) edge (5);
	\end{tikzpicture}
}

\newcommand{\dicfc}{
	\begin{tikzpicture}[scale=0.12,baseline=0.2mm]     \node[white] (S) at (4,2.5){$\triangle$};              \node[violet] (L) at (-4,1.5){$\diamondsuit$};              \node[white] (C) at (4,0){$\heartsuit$};
	\node[inner sep=0cm] (0) at (0,0){$\cdot$};
        \node[inner sep=0cm] (1) at (-1.5,1){$\cdot$};
        \node[inner sep=0cm] (2) at (1.5,1){$\cdot$};
        \node[inner sep=0cm] (3) at (-1.5,2){$\cdot$};
        \node[inner sep=0cm] (4) at (1.5,2){$\cdot$};
        \node[inner sep=0cm] (5) at (0,3){$\cdot$};
        \draw[red,-] (0) edge (1);
        \draw[red,-] (0) edge (2);
        \draw[red,-] (0) edge (3);
        \draw[red,-] (0) edge (4);
        \draw[red,-] (0) edge (5);
        \draw[red,-] (1) edge (3);
        \draw[red,-] (1) edge (4);
        \draw[red,-] (2) edge (4);
        \draw[red,-] (2) edge (5);
	\end{tikzpicture}
}

\newcommand{\dicfd}{
	\begin{tikzpicture}[scale=0.12,baseline=0.2mm]     \node[white] (S) at (4,2.5){$\triangle$};              \node[violet] (L) at (-4,1.5){$\diamondsuit$};              \node[white] (C) at (4,0){$\heartsuit$};
	\node[inner sep=0cm] (0) at (0,0){$\cdot$};
        \node[inner sep=0cm] (1) at (-1.5,1){$\cdot$};
        \node[inner sep=0cm] (2) at (1.5,1){$\cdot$};
        \node[inner sep=0cm] (3) at (-1.5,2){$\cdot$};
        \node[inner sep=0cm] (4) at (1.5,2){$\cdot$};
        \node[inner sep=0cm] (5) at (0,3){$\cdot$};
        \draw[red,-] (0) edge (1);
        \draw[red,-] (0) edge (2);
        \draw[red,-] (0) edge (3);
        \draw[red,-] (0) edge (4);
        \draw[red,-] (0) edge (5);
        \draw[red,-] (1) edge (3);
        \draw[red,-] (1) edge (4);
        \draw[red,-] (1) edge (5);
        \draw[red,-] (2) edge (4);
	\end{tikzpicture}
}

\newcommand{\dicfe}{
	\begin{tikzpicture}[scale=0.12,baseline=0.2mm]     \node[white] (S) at (4,2.5){$\triangle$};              \node[violet] (L) at (-4,1.5){$\diamondsuit$};              \node[magenta] (C) at (4,0){$\heartsuit$};
	\node[inner sep=0cm] (0) at (0,0){$\cdot$};
        \node[inner sep=0cm] (1) at (-1.5,1){$\cdot$};
        \node[inner sep=0cm] (2) at (1.5,1){$\cdot$};
        \node[inner sep=0cm] (3) at (-1.5,2){$\cdot$};
        \node[inner sep=0cm] (4) at (1.5,2){$\cdot$};
        \node[inner sep=0cm] (5) at (0,3){$\cdot$};
        \draw[red,-] (0) edge (1);
        \draw[red,-] (0) edge (2);
        \draw[red,-] (0) edge (3);
        \draw[red,-] (0) edge (4);
        \draw[red,-] (0) edge (5);
        \draw[red,-] (1) edge (3);
        \draw[red,-] (2) edge (4);
        \draw[red,-] (2) edge (5);
        \draw[red,-] (4) edge (5);
	\end{tikzpicture}
}

\newcommand{\dicff}{
	\begin{tikzpicture}[scale=0.12,baseline=0.2mm]     \node[white] (S) at (4,2.5){$\triangle$};              \node[violet] (L) at (-4,1.5){$\diamondsuit$};              \node[magenta] (C) at (4,0){$\heartsuit$};
	\node[inner sep=0cm] (0) at (0,0){$\cdot$};
        \node[inner sep=0cm] (1) at (-1.5,1){$\cdot$};
        \node[inner sep=0cm] (2) at (1.5,1){$\cdot$};
        \node[inner sep=0cm] (3) at (-1.5,2){$\cdot$};
        \node[inner sep=0cm] (4) at (1.5,2){$\cdot$};
        \node[inner sep=0cm] (5) at (0,3){$\cdot$};
        \draw[red,-] (0) edge (1);
        \draw[red,-] (0) edge (2);
        \draw[red,-] (0) edge (3);
        \draw[red,-] (0) edge (4);
        \draw[red,-] (0) edge (5);
        \draw[red,-] (1) edge (3);
        \draw[red,-] (1) edge (4);
        \draw[red,-] (1) edge (5);
        \draw[red,-] (2) edge (4);
        \draw[red,-] (2) edge (5);
	\end{tikzpicture}
}

\newcommand{\dicfg}{
	\begin{tikzpicture}[scale=0.12,baseline=0.2mm]     \node[white] (S) at (4,2.5){$\triangle$};              \node[violet] (L) at (-4,1.5){$\diamondsuit$};              \node[magenta] (C) at (4,0){$\heartsuit$};
	\node[inner sep=0cm] (0) at (0,0){$\cdot$};
        \node[inner sep=0cm] (1) at (-1.5,1){$\cdot$};
        \node[inner sep=0cm] (2) at (1.5,1){$\cdot$};
        \node[inner sep=0cm] (3) at (-1.5,2){$\cdot$};
        \node[inner sep=0cm] (4) at (1.5,2){$\cdot$};
        \node[inner sep=0cm] (5) at (0,3){$\cdot$};
        \draw[red,-] (0) edge (1);
        \draw[red,-] (0) edge (2);
        \draw[red,-] (0) edge (3);
        \draw[red,-] (0) edge (4);
        \draw[red,-] (0) edge (5);
        \draw[red,-] (1) edge (3);
        \draw[red,-] (1) edge (4);
        \draw[red,-] (1) edge (5);
        \draw[red,-] (3) edge (5);
	\end{tikzpicture}
}

\newcommand{\dicfh}{
	\begin{tikzpicture}[scale=0.12,baseline=0.2mm]     \node[white] (S) at (4,2.5){$\triangle$};              \node[violet] (L) at (-4,1.5){$\diamondsuit$};              \node[magenta] (C) at (4,0){$\heartsuit$};
	\node[inner sep=0cm] (0) at (0,0){$\cdot$};
        \node[inner sep=0cm] (1) at (-1.5,1){$\cdot$};
        \node[inner sep=0cm] (2) at (1.5,1){$\cdot$};
        \node[inner sep=0cm] (3) at (-1.5,2){$\cdot$};
        \node[inner sep=0cm] (4) at (1.5,2){$\cdot$};
        \node[inner sep=0cm] (5) at (0,3){$\cdot$};
        \draw[red,-] (0) edge (1);
        \draw[red,-] (0) edge (2);
        \draw[red,-] (0) edge (3);
        \draw[red,-] (0) edge (4);
        \draw[red,-] (0) edge (5);
        \draw[red,-] (1) edge (3);
        \draw[red,-] (1) edge (4);
        \draw[red,-] (1) edge (5);
        \draw[red,-] (2) edge (4);
        \draw[red,-] (2) edge (5);
        \draw[red,-] (4) edge (5);
	\end{tikzpicture}
}

\newcommand{\dicfi}{
	\begin{tikzpicture}[scale=0.12,baseline=0.2mm]     \node[white] (S) at (4,2.5){$\triangle$};              \node[violet] (L) at (-4,1.5){$\diamondsuit$};              \node[white] (C) at (4,0){$\heartsuit$};
	\node[inner sep=0cm] (0) at (0,0){$\cdot$};
        \node[inner sep=0cm] (1) at (-1.5,1){$\cdot$};
        \node[inner sep=0cm] (2) at (1.5,1){$\cdot$};
        \node[inner sep=0cm] (3) at (-1.5,2){$\cdot$};
        \node[inner sep=0cm] (4) at (1.5,2){$\cdot$};
        \node[inner sep=0cm] (5) at (0,3){$\cdot$};
        \draw[red,-] (0) edge (1);
        \draw[red,-] (0) edge (2);
        \draw[red,-] (0) edge (3);
        \draw[red,-] (0) edge (4);
        \draw[red,-] (0) edge (5);
        \draw[red,-] (1) edge (3);
        \draw[red,-] (1) edge (4);
        \draw[red,-] (1) edge (5);
        \draw[red,-] (2) edge (4);
        \draw[red,-] (3) edge (5);
	\end{tikzpicture}
}

\newcommand{\dicfj}{
	\begin{tikzpicture}[scale=0.12,baseline=0.2mm]     \node[white] (S) at (4,2.5){$\triangle$};              \node[violet] (L) at (-4,1.5){$\diamondsuit$};              \node[magenta] (C) at (4,0){$\heartsuit$};
	\node[inner sep=0cm] (0) at (0,0){$\cdot$};
        \node[inner sep=0cm] (1) at (-1.5,1){$\cdot$};
        \node[inner sep=0cm] (2) at (1.5,1){$\cdot$};
        \node[inner sep=0cm] (3) at (-1.5,2){$\cdot$};
        \node[inner sep=0cm] (4) at (1.5,2){$\cdot$};
        \node[inner sep=0cm] (5) at (0,3){$\cdot$};
        \draw[red,-] (0) edge (1);
        \draw[red,-] (0) edge (2);
        \draw[red,-] (0) edge (3);
        \draw[red,-] (0) edge (4);
        \draw[red,-] (0) edge (5);
        \draw[red,-] (1) edge (3);
        \draw[red,-] (1) edge (4);
        \draw[red,-] (1) edge (5);
        \draw[red,-] (3) edge (5);
        \draw[red,-] (4) edge (5);
	\end{tikzpicture}
}

\newcommand{\dicga}{
	\begin{tikzpicture}[scale=0.12,baseline=0.2mm]     \node[white] (S) at (4,2.5){$\triangle$};              \node[violet] (L) at (-4,1.5){$\diamondsuit$};              \node[magenta] (C) at (4,0){$\heartsuit$};
	\node[inner sep=0cm] (0) at (0,0){$\cdot$};
        \node[inner sep=0cm] (1) at (-1.5,1){$\cdot$};
        \node[inner sep=0cm] (2) at (1.5,1){$\cdot$};
        \node[inner sep=0cm] (3) at (-1.5,2){$\cdot$};
        \node[inner sep=0cm] (4) at (1.5,2){$\cdot$};
        \node[inner sep=0cm] (5) at (0,3){$\cdot$};
        \draw[red,-] (0) edge (1);
        \draw[red,-] (0) edge (2);
        \draw[red,-] (0) edge (3);
        \draw[red,-] (0) edge (4);
        \draw[red,-] (0) edge (5);
        \draw[red,-] (1) edge (3);
        \draw[red,-] (1) edge (4);
        \draw[red,-] (1) edge (5);
        \draw[red,-] (2) edge (4);
        \draw[red,-] (2) edge (5);
        \draw[red,-] (3) edge (5);
	\end{tikzpicture}
}

\newcommand{\dicgb}{
	\begin{tikzpicture}[scale=0.12,baseline=0.2mm]     
    \node[cyan] (S) at (4,2.5){$\triangle$};              
    \node[violet] (L) at (-4,1.5){$\diamondsuit$};              
    \node[magenta] (C) at (4,0){$\heartsuit$};
	\node[inner sep=0cm] (0) at (0,0){$\cdot$};
        \node[inner sep=0cm] (1) at (-1.5,1){$\cdot$};
        \node[inner sep=0cm] (2) at (1.5,1){$\cdot$};
        \node[inner sep=0cm] (3) at (-1.5,2){$\cdot$};
        \node[inner sep=0cm] (4) at (1.5,2){$\cdot$};
        \node[inner sep=0cm] (5) at (0,3){$\cdot$};
        \draw[red,-] (0) edge (1);
        \draw[red,-] (0) edge (2);
        \draw[red,-] (0) edge (3);
        \draw[red,-] (0) edge (4);
        \draw[red,-] (0) edge (5);
        \draw[red,-] (1) edge (3);
        \draw[red,-] (1) edge (4);
        \draw[red,-] (1) edge (5);
        \draw[red,-] (2) edge (4);
        \draw[red,-] (2) edge (5);
        \draw[red,-] (3) edge (5);
        \draw[red,-] (4) edge (5);
	\end{tikzpicture}
}

\begin{figure}[H]
	\caption{The Hasse diagram of transfer systems for the group $Dic_p$, for $p$ an odd prime.}
	\label{transfersystemDicp}	
	\[
	\begin{array}{c}
	\begin{tikzpicture}[scale=1.1]
		\node(00) at (0,0) {$\dic$};
		\node(01) at (2,1.5) {$\dicb$};
            \node(02) at (0,1.5) {$\dicc$};
		\node(07) at (-2,1.5) {$\dich$};
            \node(12) at (-4.5,3) {$\dicbc$};
            \node(08) at (-3,3) {$\dici$};
            \node(05) at (-1.5,3) {$\dicf$};
		\node(04) at (0,3) {$\dice$};
		\node(03) at (3,3) {$\dicd$};
		\node(11) at (1.5,3) {$\dicbb$};
            \node(13) at (-4.5,4.5) {$\dicbd$};
            \node(17) at (-3,4.5) {$\dicbh$};
		\node(06) at (-1.5,4.5) {$\dicg$};
            \node(16) at (0,4.5) {$\dicbg$};
            \node(14) at (3,4.5) {$\dicbe$};
            \node(22) at (1.5,4.5) {$\diccc$};
            \node(19) at (-4.5,6) {$\dicbj$};
            \node(09) at (-3,6) {$\dicj$};
		\node(10) at (-1.5,6) {$\dicba$};
            \node(30) at (1.5,6) {$\dicda$};
            \node(20) at (0,6) {$\dicca$};
            \node(18) at (6,6) {$\dicbi$};
            \node(26) at (4.5,6) {$\diccg$};
            \node(28) at (3,6) {$\dicci$};
            \node(24) at (-6,7.5) {$\dicce$};
            \node(34) at (-4.5,7.5) {$\dicde$};
            \node(15) at (-3,7.5) {$\dicbf$};
		\node(23) at (-1.5,7.5) {$\diccd$};
            \node(25) at (0,7.5) {$\diccf$};
            \node(37) at (1.5,7.5) {$\dicdh$};
            \node(21) at (3,7.5) {$\diccb$};
            \node(33) at (6,7.5) {$\dicdd$};
            \node(35) at (4.5,7.5) {$\dicdf$};
            \node(38) at (-6,9) {$\dicdi$};
            \node(27) at (-4.5,9) {$\dicch$};
            \node(31) at (-3,9) {$\dicdb$};
		\node(40) at (-1.5,9) {$\dicea$};
            \node(41) at (0,9) {$\diceb$};
            \node(29) at (1.5,9) {$\diccj$};
            \node(42) at (3,9) {$\dicec$};
            \node(43) at (4.5,9) {$\diced$};
            \node(44) at (-4.5,10.5) {$\dicee$};
            \node(32) at (-3,10.5) {$\dicdc$};
		\node(46) at (-1.5,10.5) {$\diceg$};
            \node(47) at (0,10.5) {$\diceh$};
            \node(36) at (1.5,10.5) {$\dicdg$};
            \node(45) at (3,10.5) {$\dicef$};
            \node(49) at (-3,12) {$\dicej$};
		\node(39) at (-1.5,12) {$\dicdj$};
            \node(48) at (0,12) {$\dicei$};
            \node(50) at (1.5,12) {$\dicfa$};   
            \node(53) at (0,13.5) {$\dicfd$};
		\node(52) at (1.5,13.5) {$\dicfc$};
            \node(56) at (-1.5,13.5) {$\dicfg$};
            \node(51) at (-3,13.5) {$\dicfb$};
            \node(54) at (3,13.5) {$\dicfe$};    
		\node(59) at (-3,15) {$\dicfj$};
            \node(58) at (0,15) {$\dicfi$};
            \node(55) at (1.5,15) {$\dicff$};
		\node(60) at (0,16.5) {$\dicga$};
            \node(57) at (1.5,16.5) {$\dicfh$};
		\node(61) at (0,18) {$\dicgb$};
	
        \node[cyan] (S) at (-6.15,18){$\triangle$ = saturated}; 
        \node[magenta] (C) at (-6.05,17.5){$\heartsuit$ = cosaturated}; 
        \node[violet] (L) at (-5,17){$\vardiamond$ = lesser simply paired (LSP)};      
        \node[violet] (L) at (-5.3,16.5){$\diamondsuit$ = connected ($\implies$ LSP)};     

        \node[inner sep=0cm] (70) at (-5.5,13){$e$};
        \node[inner sep=0cm] (71) at (-4,14){$C_p$};
        \node[inner sep=0cm] (72) at (-7,14){$C_2$};
        \node[inner sep=0cm] (73) at (-4,15){$C_{2p}$};
        \node[inner sep=0cm] (74) at (-7,15){${}_p C_{4}$};
        \node[inner sep=0cm] (75) at (-5.5,16){$\Dic_p$};
        \draw[black,-] (70) edge (71);
        \draw[black,-] (70) edge (72);
        \draw[black,-] (70) edge (73);
        \draw[black,-] (70) edge (74);
        \draw[black,-] (70) edge (75);
        \draw[black,-] (71) edge (73);
        \draw[black,-] (71) edge (75);
        \draw[black,-] (72) edge (73);
        \draw[black,-] (72) edge (74);
        \draw[black,-] (72) edge (75);
        \draw[black,-] (73) edge (75);
        \draw[black,-] (74) edge (75);

		\path[-]
		(00) edge node {} (01)
		(00) edge node {} (02)
		(00) edge node {} (07)
		(02) edge node {} (08)
            (02) edge node {} (05)
            (02) edge node {} (04)
            (07) edge node {} (08)
		(07) edge node {} (12)
		(01) edge node {} (04)
		(01) edge node {} (03)
		(01) edge node {} (11)
		(12) edge node {} (13)
            (12) edge node {} (28)
		(08) edge node {} (13)
		(08) edge node {} (17)
		(05) edge node {} (17)
		(05) edge node {} (09)
		(04) edge node {} (06)
		(04) edge node {} (16)
		(03) edge node {} (14)
		(11) edge node {} (16)
            (11) edge node {} (14)
		(11) edge node {} (22)
		(13) edge node {} (24)
		(17) edge node {} (19)
		(06) edge node {} (09)
		(06) edge node {} (10)
		(06) edge node {} (20)
		(16) edge node {} (30)
		(16) edge node {} (20)
		(14) edge node {} (18)
		(14) edge node {} (26)
        (21) edge node {} (31)
		(22) edge node {} (26)
		(22) edge node {} (28)
            (19) edge node {} (24)
            (19) edge node {} (34)
            (09) edge node {} (15)
            (09) edge node {} (23)
            (10) edge node {} (15)
            (10) edge node {} (25)
            (30) edge node {} (37)
            (30) edge node {} (35)
            (20) edge node {} (25)
            (20) edge node {} (37)
            (20) edge node {} (21)
            (18) edge node {} (33)
            (26) edge node {} (33)
            (28) edge node {} (35)
            (28) edge node {} (43)
            (24) edge node {} (38)
            (34) edge node {} (38)
            (15) edge node {} (27)
            (23) edge node {} (27)
            (23) edge node {} (31)
            (23) edge node {} (40)
            (25) edge node {} (27)
            (25) edge node {} (41)
            (25) edge node {} (29)
            (37) edge node {} (40)
            (37) edge node {} (41)
            (37) edge node {} (42)
            (21) edge node {} (29)
            (21) edge node {} (42)
            (33) edge node {} (43)
            (35) edge node {} (45)
            (38) edge node {} (59)
            (27) edge node {} (44)
            (27) edge node {} (32)
            (31) edge node {} (32)
            (31) edge node {} (46)
            (40) edge node {} (46)
            (40) edge node {} (44)
            (41) edge node {} (44)
            (41) edge node {} (47)
            (29) edge node {} (32)
            (29) edge node {} (47)
            (29) edge node {} (36)
            (42) edge node {} (46)
            (42) edge node {} (47)
            (42) edge node {} (45)
            (43) edge node {} (54)
            (44) edge node {} (49)
            (32) edge node {} (49)
            (32) edge node {} (39)
            (46) edge node {} (49)
            (46) edge node {} (48)
            (47) edge node {} (49)
            (47) edge node {} (50)
            (36) edge node {} (39)
            (36) edge node {} (50)
            (45) edge node {} (54)
            (49) edge node {} (53)
            (49) edge node {} (52)
            (39) edge node {} (52)
            (48) edge node {} (53)
            (48) edge node {} (56)
            (48) edge node {} (51)
            (50) edge node {} (52)
            (50) edge node {} (54)
            (53) edge node {} (58)
            (53) edge node {} (55)
            (52) edge node {} (55)
            (56) edge node {} (58)
            (56) edge node {} (59)
            (51) edge node {} (59)
            (54) edge node {} (57)
            (58) edge node {} (60)
            (55) edge node {} (60)
            (55) edge node {} (57)
            (59) edge node {} (61)
            (60) edge node {} (61)
            (57) edge node {} (61)
		;
	\end{tikzpicture}
	\end{array}
	\]
\end{figure}
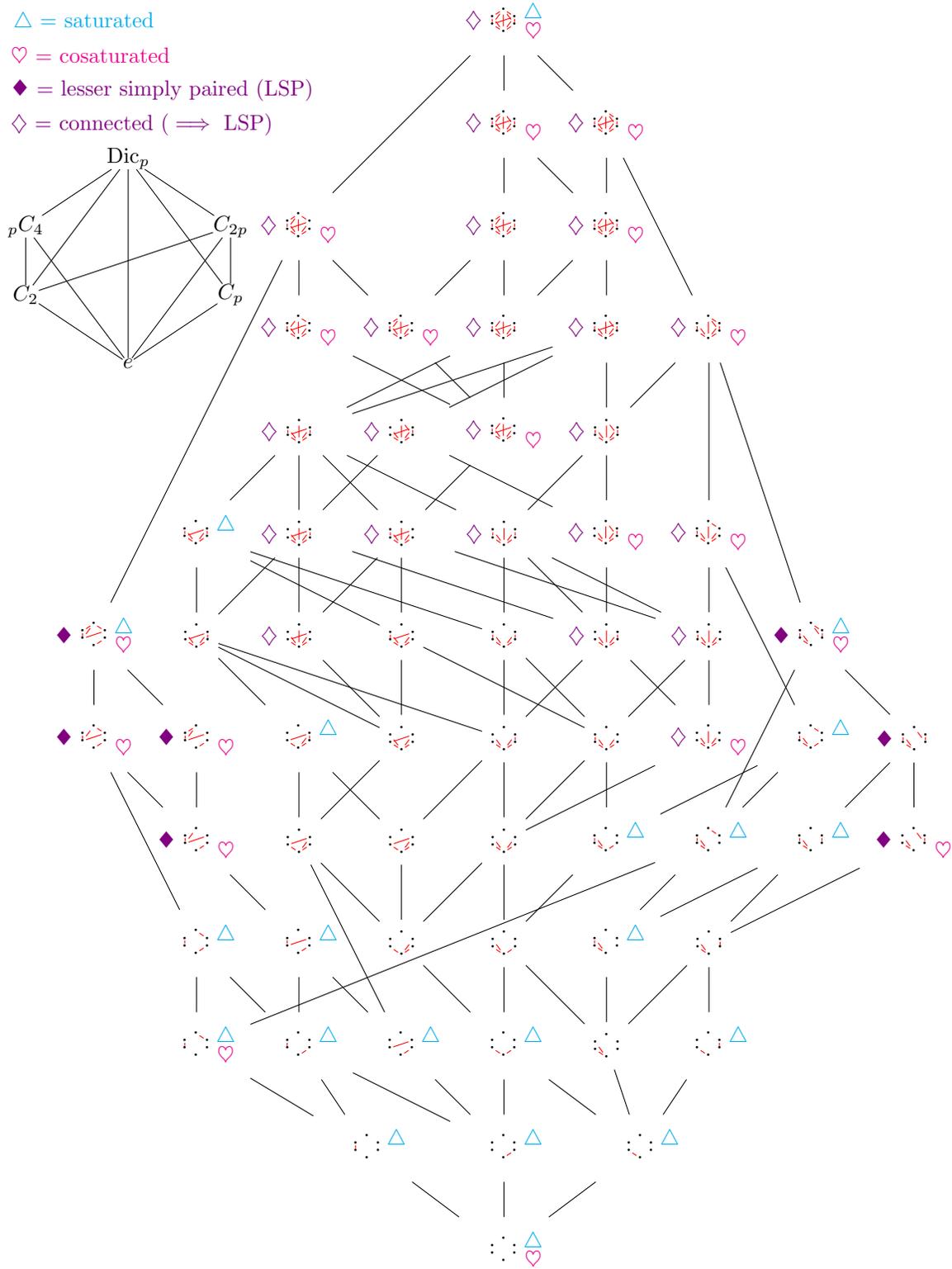

\newcommand{\fba}{
    \begin{tikzpicture}[scale=0.12,baseline=0.2mm]     \node[cyan] (S) at (4,2.5){$\triangle$};              \node[white] (L) at (-4,1.5){$\diamondsuit$};              \node[magenta] (C) at (4,0){$\heartsuit$};
	\node[inner sep=0cm] (0) at (0,0){$\cdot$};
        \node[inner sep=0cm] (1) at (-1.5,1){$\cdot$};
        \node[inner sep=0cm] (2) at (-1.5,2){$\cdot$};
        \node[inner sep=0cm] (3) at (1.5,1){$\cdot$};
        \node[inner sep=0cm] (4) at (1.5,2){$\cdot$};
        \node[inner sep=0cm] (5) at (0,3){$\cdot$};
    \end{tikzpicture}
}

\newcommand{\fbb}{
    \begin{tikzpicture}[scale=0.12,baseline=0.2mm]     \node[cyan] (S) at (4,2.5){$\triangle$};              \node[white] (L) at (-4,1.5){$\diamondsuit$};              \node[white] (C) at (4,0){$\heartsuit$};
	\node[inner sep=0cm] (0) at (0,0){$\cdot$};
        \node[inner sep=0cm] (1) at (-1.5,1){$\cdot$};
        \node[inner sep=0cm] (2) at (-1.5,2){$\cdot$};
        \node[inner sep=0cm] (3) at (1.5,1){$\cdot$};
        \node[inner sep=0cm] (4) at (1.5,2){$\cdot$};
        \node[inner sep=0cm] (5) at (0,3){$\cdot$};
        \draw[red,-] (0) edge (3);
    \end{tikzpicture}
}

\newcommand{\fbc}{
    \begin{tikzpicture}[scale=0.12,baseline=0.2mm]     \node[cyan] (S) at (4,2.5){$\triangle$};              \node[white] (L) at (-4,1.5){$\diamondsuit$};              \node[white] (C) at (4,0){$\heartsuit$};
	\node[inner sep=0cm] (0) at (0,0){$\cdot$};
        \node[inner sep=0cm] (1) at (-1.5,1){$\cdot$};
        \node[inner sep=0cm] (2) at (-1.5,2){$\cdot$};
        \node[inner sep=0cm] (3) at (1.5,1){$\cdot$};
        \node[inner sep=0cm] (4) at (1.5,2){$\cdot$};
        \node[inner sep=0cm] (5) at (0,3){$\cdot$};
        \draw[red,-] (1) edge (2);
    \end{tikzpicture}
}

\newcommand{\fbd}{
    \begin{tikzpicture}[scale=0.12,baseline=0.2mm]     \node[cyan] (S) at (4,2.5){$\triangle$};              \node[white] (L) at (-4,1.5){$\diamondsuit$};              \node[white] (C) at (4,0){$\heartsuit$};
	\node[inner sep=0cm] (0) at (0,0){$\cdot$};
        \node[inner sep=0cm] (1) at (-1.5,1){$\cdot$};
        \node[inner sep=0cm] (2) at (-1.5,2){$\cdot$};
        \node[inner sep=0cm] (3) at (1.5,1){$\cdot$};
        \node[inner sep=0cm] (4) at (1.5,2){$\cdot$};
        \node[inner sep=0cm] (5) at (0,3){$\cdot$};
        \draw[red,-] (0) edge (1);
    \end{tikzpicture}
}

\newcommand{\fbe}{
    \begin{tikzpicture}[scale=0.12,baseline=0.2mm]     \node[cyan] (S) at (4,2.5){$\triangle$};              \node[white] (L) at (-4,1.5){$\diamondsuit$};              \node[magenta] (C) at (4,0){$\heartsuit$};
	\node[inner sep=0cm] (0) at (0,0){$\cdot$};
        \node[inner sep=0cm] (1) at (-1.5,1){$\cdot$};
        \node[inner sep=0cm] (2) at (-1.5,2){$\cdot$};
        \node[inner sep=0cm] (3) at (1.5,1){$\cdot$};
        \node[inner sep=0cm] (4) at (1.5,2){$\cdot$};
        \node[inner sep=0cm] (5) at (0,3){$\cdot$};
        \draw[red,-] (1) edge (2);
        \draw[red,-] (4) edge (5);
    \end{tikzpicture}
}

\newcommand{\fbf}{
    \begin{tikzpicture}[scale=0.12,baseline=0.2mm]     \node[cyan] (S) at (4,2.5){$\triangle$};              \node[white] (L) at (-4,1.5){$\diamondsuit$};              \node[white] (C) at (4,0){$\heartsuit$};
	\node[inner sep=0cm] (0) at (0,0){$\cdot$};
        \node[inner sep=0cm] (1) at (-1.5,1){$\cdot$};
        \node[inner sep=0cm] (2) at (-1.5,2){$\cdot$};
        \node[inner sep=0cm] (3) at (1.5,1){$\cdot$};
        \node[inner sep=0cm] (4) at (1.5,2){$\cdot$};
        \node[inner sep=0cm] (5) at (0,3){$\cdot$};
        \draw[red,-] (0) edge (1);
        \draw[red,-] (3) edge (4);
    \end{tikzpicture}
}

\newcommand{\fbg}{
    \begin{tikzpicture}[scale=0.12,baseline=0.2mm]     \node[cyan] (S) at (4,2.5){$\triangle$};              \node[white] (L) at (-4,1.5){$\diamondsuit$};              \node[white] (C) at (4,0){$\heartsuit$};
	\node[inner sep=0cm] (0) at (0,0){$\cdot$};
        \node[inner sep=0cm] (1) at (-1.5,1){$\cdot$};
        \node[inner sep=0cm] (2) at (-1.5,2){$\cdot$};
        \node[inner sep=0cm] (3) at (1.5,1){$\cdot$};
        \node[inner sep=0cm] (4) at (1.5,2){$\cdot$};
        \node[inner sep=0cm] (5) at (0,3){$\cdot$};
        \draw[red,-] (0) edge (3);
        \draw[red,-] (1) edge (2);
    \end{tikzpicture}
}

\newcommand{\fbh}{
    \begin{tikzpicture}[scale=0.12,baseline=0.2mm]     \node[cyan] (S) at (4,2.5){$\triangle$};              \node[white] (L) at (-4,1.5){$\diamondsuit$};              \node[white] (C) at (4,0){$\heartsuit$};
	\node[inner sep=0cm] (0) at (0,0){$\cdot$};
        \node[inner sep=0cm] (1) at (-1.5,1){$\cdot$};
        \node[inner sep=0cm] (2) at (-1.5,2){$\cdot$};
        \node[inner sep=0cm] (3) at (1.5,1){$\cdot$};
        \node[inner sep=0cm] (4) at (1.5,2){$\cdot$};
        \node[inner sep=0cm] (5) at (0,3){$\cdot$};
        \draw[red,-] (0) edge (1);
        \draw[red,-] (0) edge (3);
    \end{tikzpicture}
}

\newcommand{\fbi}{
    \begin{tikzpicture}[scale=0.12,baseline=0.2mm]    
	\node[inner sep=0cm] (0) at (0,0){$\cdot$};
        \node[inner sep=0cm] (1) at (-1.5,1){$\cdot$};
        \node[inner sep=0cm] (2) at (-1.5,2){$\cdot$};
        \node[inner sep=0cm] (3) at (1.5,1){$\cdot$};
        \node[inner sep=0cm] (4) at (1.5,2){$\cdot$};
        \node[inner sep=0cm] (5) at (0,3){$\cdot$};
        \draw[red,-] (0) edge (1);
        \draw[red,-] (0) edge (3);
        \draw[red,-] (0) edge (4);
    \end{tikzpicture}
}

\newcommand{\fbj}{
    \begin{tikzpicture}[scale=0.12,baseline=0.2mm]     \node[cyan] (S) at (4,2.5){$\triangle$};              \node[white] (L) at (-4,1.5){$\diamondsuit$};              \node[white] (C) at (4,0){$\heartsuit$};
	\node[inner sep=0cm] (0) at (0,0){$\cdot$};
        \node[inner sep=0cm] (1) at (-1.5,1){$\cdot$};
        \node[inner sep=0cm] (2) at (-1.5,2){$\cdot$};
        \node[inner sep=0cm] (3) at (1.5,1){$\cdot$};
        \node[inner sep=0cm] (4) at (1.5,2){$\cdot$};
        \node[inner sep=0cm] (5) at (0,3){$\cdot$};
        \draw[red,-] (0) edge (3);
        \draw[red,-] (1) edge (2);
        \draw[red,-] (4) edge (5);
    \end{tikzpicture}
}

\newcommand{\fbba}{
    \begin{tikzpicture}[scale=0.12,baseline=0.2mm]     \node[cyan] (S) at (4,2.5){$\triangle$};              \node[white] (L) at (-4,1.5){$\diamondsuit$};              \node[white] (C) at (4,0){$\heartsuit$};
	\node[inner sep=0cm] (0) at (0,0){$\cdot$};
        \node[inner sep=0cm] (1) at (-1.5,1){$\cdot$};
        \node[inner sep=0cm] (2) at (-1.5,2){$\cdot$};
        \node[inner sep=0cm] (3) at (1.5,1){$\cdot$};
        \node[inner sep=0cm] (4) at (1.5,2){$\cdot$};
        \node[inner sep=0cm] (5) at (0,3){$\cdot$};
        \draw[red,-] (0) edge (1);
        \draw[red,-] (0) edge (3);
        \draw[red,-] (0) edge (4);
        \draw[red,-] (3) edge (4);
    \end{tikzpicture}
}

\newcommand{\fbbb}{
    \begin{tikzpicture}[scale=0.12,baseline=0.2mm]    
	\node[inner sep=0cm] (0) at (0,0){$\cdot$};
        \node[inner sep=0cm] (1) at (-1.5,1){$\cdot$};
        \node[inner sep=0cm] (2) at (-1.5,2){$\cdot$};
        \node[inner sep=0cm] (3) at (1.5,1){$\cdot$};
        \node[inner sep=0cm] (4) at (1.5,2){$\cdot$};
        \node[inner sep=0cm] (5) at (0,3){$\cdot$};
        \draw[red,-] (0) edge (1);
        \draw[red,-] (0) edge (2);
    \end{tikzpicture}
}

\newcommand{\fbbc}{
    \begin{tikzpicture}[scale=0.12,baseline=0.2mm]     
	\node[inner sep=0cm] (0) at (0,0){$\cdot$};
        \node[inner sep=0cm] (1) at (-1.5,1){$\cdot$};
        \node[inner sep=0cm] (2) at (-1.5,2){$\cdot$};
        \node[inner sep=0cm] (3) at (1.5,1){$\cdot$};
        \node[inner sep=0cm] (4) at (1.5,2){$\cdot$};
        \node[inner sep=0cm] (5) at (0,3){$\cdot$};
        \draw[red,-] (0) edge (1);
        \draw[red,-] (0) edge (2);
        \draw[red,-] (0) edge (3);
    \end{tikzpicture}
}

\newcommand{\fbbd}{
    \begin{tikzpicture}[scale=0.12,baseline=0.2mm]     
	\node[inner sep=0cm] (0) at (0,0){$\cdot$};
        \node[inner sep=0cm] (1) at (-1.5,1){$\cdot$};
        \node[inner sep=0cm] (2) at (-1.5,2){$\cdot$};
        \node[inner sep=0cm] (3) at (1.5,1){$\cdot$};
        \node[inner sep=0cm] (4) at (1.5,2){$\cdot$};
        \node[inner sep=0cm] (5) at (0,3){$\cdot$};
        \draw[red,-] (0) edge (1);
        \draw[red,-] (0) edge (2);
        \draw[red,-] (3) edge (4);
    \end{tikzpicture}
}

\newcommand{\fbbe}{
    \begin{tikzpicture}[scale=0.12,baseline=0.2mm]     
	\node[inner sep=0cm] (0) at (0,0){$\cdot$};
        \node[inner sep=0cm] (1) at (-1.5,1){$\cdot$};
        \node[inner sep=0cm] (2) at (-1.5,2){$\cdot$};
        \node[inner sep=0cm] (3) at (1.5,1){$\cdot$};
        \node[inner sep=0cm] (4) at (1.5,2){$\cdot$};
        \node[inner sep=0cm] (5) at (0,3){$\cdot$};
        \draw[red,-] (0) edge (1);
        \draw[red,-] (0) edge (2);
        \draw[red,-] (0) edge (3);
        \draw[red,-] (0) edge (4);
    \end{tikzpicture}
}

\newcommand{\fbbf}{
    \begin{tikzpicture}[scale=0.12,baseline=0.2mm]     
    \node[white] (S) at (4,2.5){$\triangle$};              
    \node[violet] (L) at (-4,1.5){$\vardiamond$};              
    \node[magenta] (C) at (4,0){$\heartsuit$};
	\node[inner sep=0cm] (0) at (0,0){$\cdot$};
        \node[inner sep=0cm] (1) at (-1.5,1){$\cdot$};
        \node[inner sep=0cm] (2) at (-1.5,2){$\cdot$};
        \node[inner sep=0cm] (3) at (1.5,1){$\cdot$};
        \node[inner sep=0cm] (4) at (1.5,2){$\cdot$};
        \node[inner sep=0cm] (5) at (0,3){$\cdot$};
        \draw[red,-] (0) edge (1);
        \draw[red,-] (0) edge (2);
        \draw[red,-] (3) edge (4);
        \draw[red,-] (3) edge (5);
    \end{tikzpicture}
}

\newcommand{\fbbg}{
    \begin{tikzpicture}[scale=0.12,baseline=0.2mm]     
	\node[inner sep=0cm] (0) at (0,0){$\cdot$};
        \node[inner sep=0cm] (1) at (-1.5,1){$\cdot$};
        \node[inner sep=0cm] (2) at (-1.5,2){$\cdot$};
        \node[inner sep=0cm] (3) at (1.5,1){$\cdot$};
        \node[inner sep=0cm] (4) at (1.5,2){$\cdot$};
        \node[inner sep=0cm] (5) at (0,3){$\cdot$};
        \draw[red,-] (0) edge (1);
        \draw[red,-] (0) edge (3);
        \draw[red,-] (0) edge (4);
        \draw[red,-] (1) edge (4);
    \end{tikzpicture}
}

\newcommand{\fbbh}{
    \begin{tikzpicture}[scale=0.12,baseline=0.2mm]    
	\node[inner sep=0cm] (0) at (0,0){$\cdot$};
        \node[inner sep=0cm] (1) at (-1.5,1){$\cdot$};
        \node[inner sep=0cm] (2) at (-1.5,2){$\cdot$};
        \node[inner sep=0cm] (3) at (1.5,1){$\cdot$};
        \node[inner sep=0cm] (4) at (1.5,2){$\cdot$};
        \node[inner sep=0cm] (5) at (0,3){$\cdot$};
        \draw[red,-] (0) edge (1);
        \draw[red,-] (0) edge (2);
        \draw[red,-] (0) edge (3);
        \draw[red,-] (0) edge (4);
        \draw[red,-] (3) edge (4);
    \end{tikzpicture}
}

\newcommand{\fbbi}{
    \begin{tikzpicture}[scale=0.12,baseline=0.2mm]     \node[white] (S) at (4,2.5){$\triangle$};              \node[violet] (L) at (-4,1.5){$\diamondsuit$};              \node[magenta] (C) at (4,0){$\heartsuit$};
	\node[inner sep=0cm] (0) at (0,0){$\cdot$};
        \node[inner sep=0cm] (1) at (-1.5,1){$\cdot$};
        \node[inner sep=0cm] (2) at (-1.5,2){$\cdot$};
        \node[inner sep=0cm] (3) at (1.5,1){$\cdot$};
        \node[inner sep=0cm] (4) at (1.5,2){$\cdot$};
        \node[inner sep=0cm] (5) at (0,3){$\cdot$};
        \draw[red,-] (0) edge (1);
        \draw[red,-] (0) edge (2);
        \draw[red,-] (0) edge (3);
        \draw[red,-] (0) edge (4);
        \draw[red,-] (0) edge (5);
    \end{tikzpicture}
}

\newcommand{\fbbj}{
    \begin{tikzpicture}[scale=0.12,baseline=0.2mm]     \node[cyan] (S) at (4,2.5){$\triangle$};              \node[white] (L) at (-4,1.5){$\diamondsuit$};              \node[white] (C) at (4,0){$\heartsuit$};
	\node[inner sep=0cm] (0) at (0,0){$\cdot$};
        \node[inner sep=0cm] (1) at (-1.5,1){$\cdot$};
        \node[inner sep=0cm] (2) at (-1.5,2){$\cdot$};
        \node[inner sep=0cm] (3) at (1.5,1){$\cdot$};
        \node[inner sep=0cm] (4) at (1.5,2){$\cdot$};
        \node[inner sep=0cm] (5) at (0,3){$\cdot$};
        \draw[red,-] (0) edge (1);
        \draw[red,-] (0) edge (3);
        \draw[red,-] (0) edge (4);
        \draw[red,-] (1) edge (4);
        \draw[red,-] (3) edge (4);
    \end{tikzpicture}
}

\newcommand{\fbca}{
    \begin{tikzpicture}[scale=0.12,baseline=0.2mm]    \node[violet] (L) at (-4,1.5){$\diamondsuit$}; 
    \node[white] (W) at (4,1.5){$\diamondsuit$};
	\node[inner sep=0cm] (0) at (0,0){$\cdot$};
        \node[inner sep=0cm] (1) at (-1.5,1){$\cdot$};
        \node[inner sep=0cm] (2) at (-1.5,2){$\cdot$};
        \node[inner sep=0cm] (3) at (1.5,1){$\cdot$};
        \node[inner sep=0cm] (4) at (1.5,2){$\cdot$};
        \node[inner sep=0cm] (5) at (0,3){$\cdot$};
        \draw[red,-] (0) edge (1);
        \draw[red,-] (0) edge (2);
        \draw[red,-] (0) edge (3);
        \draw[red,-] (0) edge (4);
        \draw[red,-] (0) edge (5);
        \draw[red,-] (3) edge (4);
    \end{tikzpicture}
}

\newcommand{\fbcb}{
    \begin{tikzpicture}[scale=0.12,baseline=0.2mm] 
    \node[white] (W) at (4,2.5){$\triangle$}; 
    \node[violet] (L) at (-4,1.5){$\diamondsuit$};              \node[magenta] (C) at (4,0){$\heartsuit$};
        \node[inner sep=0cm] (0) at (0,0){$\cdot$};
        \node[inner sep=0cm] (1) at (-1.5,1){$\cdot$};
        \node[inner sep=0cm] (2) at (-1.5,2){$\cdot$};
        \node[inner sep=0cm] (3) at (1.5,1){$\cdot$};
        \node[inner sep=0cm] (4) at (1.5,2){$\cdot$};
        \node[inner sep=0cm] (5) at (0,3){$\cdot$};
        \draw[red,-] (0) edge (1);
        \draw[red,-] (0) edge (2);
        \draw[red,-] (0) edge (3);
        \draw[red,-] (0) edge (4);
        \draw[red,-] (0) edge (5);
        \draw[red,-] (3) edge (4);
        \draw[red,-] (3) edge (5);	
    \end{tikzpicture}
}

\newcommand{\fbcc}{
    \begin{tikzpicture}[scale=0.12,baseline=0.2mm]     \node[cyan] (S) at (4,2.5){$\triangle$};              \node[white] (L) at (-4,1.5){$\diamondsuit$};              \node[white] (C) at (4,0){$\heartsuit$};
	\node[inner sep=0cm] (0) at (0,0){$\cdot$};
        \node[inner sep=0cm] (1) at (-1.5,1){$\cdot$};
        \node[inner sep=0cm] (2) at (-1.5,2){$\cdot$};
        \node[inner sep=0cm] (3) at (1.5,1){$\cdot$};
        \node[inner sep=0cm] (4) at (1.5,2){$\cdot$};
        \node[inner sep=0cm] (5) at (0,3){$\cdot$};
        \draw[red,-] (0) edge (1);
        \draw[red,-] (0) edge (2);
        \draw[red,-] (1) edge (2);
    \end{tikzpicture}
}

\newcommand{\fbcd}{
    \begin{tikzpicture}[scale=0.12,baseline=0.2mm]     \node[cyan] (S) at (4,2.5){$\triangle$};              \node[white] (L) at (-4,1.5){$\diamondsuit$};              \node[white] (C) at (4,0){$\heartsuit$};
	\node[inner sep=0cm] (0) at (0,0){$\cdot$};
        \node[inner sep=0cm] (1) at (-1.5,1){$\cdot$};
        \node[inner sep=0cm] (2) at (-1.5,2){$\cdot$};
        \node[inner sep=0cm] (3) at (1.5,1){$\cdot$};
        \node[inner sep=0cm] (4) at (1.5,2){$\cdot$};
        \node[inner sep=0cm] (5) at (0,3){$\cdot$};
        \draw[red,-] (0) edge (1);
        \draw[red,-] (0) edge (2);
        \draw[red,-] (0) edge (3);
        \draw[red,-] (1) edge (2);
    \end{tikzpicture}
}

\newcommand{\fbce}{
    \begin{tikzpicture}[scale=0.12,baseline=0.2mm]     
    \node[cyan] (S) at (4,2.5){$\triangle$};              
    \node[white] (L) at (-4,1.5){$\diamondsuit$};              
    \node[white] (C) at (4,0){$\heartsuit$};
	\node[inner sep=0cm] (0) at (0,0){$\cdot$};
        \node[inner sep=0cm] (1) at (-1.5,1){$\cdot$};
        \node[inner sep=0cm] (2) at (-1.5,2){$\cdot$};
        \node[inner sep=0cm] (3) at (1.5,1){$\cdot$};
        \node[inner sep=0cm] (4) at (1.5,2){$\cdot$};
        \node[inner sep=0cm] (5) at (0,3){$\cdot$};
        \draw[red,-] (0) edge (1);
        \draw[red,-] (0) edge (2);
        \draw[red,-] (1) edge (2);
        \draw[red,-] (4) edge (5);
    \end{tikzpicture}
}

\newcommand{\fbcf}{
    \begin{tikzpicture}[scale=0.12,baseline=0.2mm]     \node[cyan] (S) at (4,2.5){$\triangle$};              \node[white] (L) at (-4,1.5){$\diamondsuit$};              \node[white] (C) at (4,0){$\heartsuit$};
	\node[inner sep=0cm] (0) at (0,0){$\cdot$};
        \node[inner sep=0cm] (1) at (-1.5,1){$\cdot$};
        \node[inner sep=0cm] (2) at (-1.5,2){$\cdot$};
        \node[inner sep=0cm] (3) at (1.5,1){$\cdot$};
        \node[inner sep=0cm] (4) at (1.5,2){$\cdot$};
        \node[inner sep=0cm] (5) at (0,3){$\cdot$};
        \draw[red,-] (0) edge (1);
        \draw[red,-] (0) edge (2);
        \draw[red,-] (1) edge (2);
        \draw[red,-] (3) edge (4);
    \end{tikzpicture}
}
\newcommand{\fbcg}{
    \begin{tikzpicture}[scale=0.12,baseline=0.2mm]     \node[white] (W) at (4,1.5){$\triangle$};            \node[violet] (L) at (-4,1.5){$\vardiamond$};              
	\node[inner sep=0cm] (0) at (0,0){$\cdot$};
        \node[inner sep=0cm] (1) at (-1.5,1){$\cdot$};
        \node[inner sep=0cm] (2) at (-1.5,2){$\cdot$};
        \node[inner sep=0cm] (3) at (1.5,1){$\cdot$};
        \node[inner sep=0cm] (4) at (1.5,2){$\cdot$};
        \node[inner sep=0cm] (5) at (0,3){$\cdot$};
        \draw[red,-] (0) edge (1);
        \draw[red,-] (0) edge (2);
        \draw[red,-] (1) edge (2);
        \draw[red,-] (3) edge (4);
        \draw[red,-] (3) edge (5);
    \end{tikzpicture}
}

\newcommand{\fbch}{
    \begin{tikzpicture}[scale=0.12,baseline=0.2mm]     
	\node[inner sep=0cm] (0) at (0,0){$\cdot$};
        \node[inner sep=0cm] (1) at (-1.5,1){$\cdot$};
        \node[inner sep=0cm] (2) at (-1.5,2){$\cdot$};
        \node[inner sep=0cm] (3) at (1.5,1){$\cdot$};
        \node[inner sep=0cm] (4) at (1.5,2){$\cdot$};
        \node[inner sep=0cm] (5) at (0,3){$\cdot$};
        \draw[red,-] (0) edge (1);
        \draw[red,-] (0) edge (2);
        \draw[red,-] (0) edge (3);
        \draw[red,-] (0) edge (4);
        \draw[red,-] (1) edge (2);
    \end{tikzpicture}
}

\newcommand{\fbci}{
    \begin{tikzpicture}[scale=0.12,baseline=0.2mm]     \node[cyan] (S) at (4,2.5){$\triangle$};              \node[white] (L) at (-4,1.5){$\diamondsuit$};              \node[white] (C) at (4,0){$\heartsuit$};
	\node[inner sep=0cm] (0) at (0,0){$\cdot$};
        \node[inner sep=0cm] (1) at (-1.5,1){$\cdot$};
        \node[inner sep=0cm] (2) at (-1.5,2){$\cdot$};
        \node[inner sep=0cm] (3) at (1.5,1){$\cdot$};
        \node[inner sep=0cm] (4) at (1.5,2){$\cdot$};
        \node[inner sep=0cm] (5) at (0,3){$\cdot$};
        \draw[red,-] (0) edge (1);
        \draw[red,-] (0) edge (2);
        \draw[red,-] (0) edge (3);
        \draw[red,-] (1) edge (2);
        \draw[red,-] (4) edge (5);
    \end{tikzpicture}
}

\newcommand{\fbcj}{
    \begin{tikzpicture}[scale=0.12,baseline=0.2mm]     
	\node[inner sep=0cm] (0) at (0,0){$\cdot$};
        \node[inner sep=0cm] (1) at (-1.5,1){$\cdot$};
        \node[inner sep=0cm] (2) at (-1.5,2){$\cdot$};
        \node[inner sep=0cm] (3) at (1.5,1){$\cdot$};
        \node[inner sep=0cm] (4) at (1.5,2){$\cdot$};
        \node[inner sep=0cm] (5) at (0,3){$\cdot$};
        \draw[red,-] (0) edge (1);
        \draw[red,-] (0) edge (2);
        \draw[red,-] (0) edge (3);
        \draw[red,-] (0) edge (4);
        \draw[red,-] (1) edge (4);
    \end{tikzpicture}
}

\newcommand{\fbda}{
    \begin{tikzpicture}[scale=0.12,baseline=0.2mm]    
    \node[white] (W) at (4,1.5){$\triangle$}; 
    \node[violet] (L) at (-4,1.5){$\diamondsuit$};              
	\node[inner sep=0cm] (0) at (0,0){$\cdot$};
        \node[inner sep=0cm] (1) at (-1.5,1){$\cdot$};
        \node[inner sep=0cm] (2) at (-1.5,2){$\cdot$};
        \node[inner sep=0cm] (3) at (1.5,1){$\cdot$};
        \node[inner sep=0cm] (4) at (1.5,2){$\cdot$};
        \node[inner sep=0cm] (5) at (0,3){$\cdot$};
        \draw[red,-] (0) edge (1);
        \draw[red,-] (0) edge (2);
        \draw[red,-] (0) edge (3);
        \draw[red,-] (0) edge (4);
        \draw[red,-] (0) edge (5);
        \draw[red,-] (1) edge (4);
    \end{tikzpicture}
}

\newcommand{\fbdb}{
    \begin{tikzpicture}[scale=0.12,baseline=0.2mm]     \node[white] (S) at (4,2.5){$\triangle$};              \node[white] (L) at (-4,1.5){$\diamondsuit$};              \node[white] (C) at (4,0){$\heartsuit$};
	\node[inner sep=0cm] (0) at (0,0){$\cdot$};
        \node[inner sep=0cm] (1) at (-1.5,1){$\cdot$};
        \node[inner sep=0cm] (2) at (-1.5,2){$\cdot$};
        \node[inner sep=0cm] (3) at (1.5,1){$\cdot$};
        \node[inner sep=0cm] (4) at (1.5,2){$\cdot$};
        \node[inner sep=0cm] (5) at (0,3){$\cdot$};
        \draw[red,-] (0) edge (1);
        \draw[red,-] (0) edge (2);
        \draw[red,-] (0) edge (3);
        \draw[red,-] (0) edge (4);
        \draw[red,-] (1) edge (2);
        \draw[red,-] (3) edge (4);
    \end{tikzpicture}
}

\newcommand{\fbdc}{
    \begin{tikzpicture}[scale=0.12,baseline=0.2mm]    
	\node[inner sep=0cm] (0) at (0,0){$\cdot$};
        \node[inner sep=0cm] (1) at (-1.5,1){$\cdot$};
        \node[inner sep=0cm] (2) at (-1.5,2){$\cdot$};
        \node[inner sep=0cm] (3) at (1.5,1){$\cdot$};
        \node[inner sep=0cm] (4) at (1.5,2){$\cdot$};
        \node[inner sep=0cm] (5) at (0,3){$\cdot$};
        \draw[red,-] (0) edge (1);
        \draw[red,-] (0) edge (2);
        \draw[red,-] (0) edge (3);
        \draw[red,-] (0) edge (4);
        \draw[red,-] (1) edge (4);
        \draw[red,-] (3) edge (4);
    \end{tikzpicture}
}

\newcommand{\fbdd}{
    \begin{tikzpicture}[scale=0.12,baseline=0.2mm]     \node[cyan] (S) at (4,2.5){$\triangle$};              \node[violet] (L) at (-4,1.5){$\vardiamond$};              \node[magenta] (C) at (4,0){$\heartsuit$};
	\node[inner sep=0cm] (0) at (0,0){$\cdot$};
        \node[inner sep=0cm] (1) at (-1.5,1){$\cdot$};
        \node[inner sep=0cm] (2) at (-1.5,2){$\cdot$};
        \node[inner sep=0cm] (3) at (1.5,1){$\cdot$};
        \node[inner sep=0cm] (4) at (1.5,2){$\cdot$};
        \node[inner sep=0cm] (5) at (0,3){$\cdot$};
        \draw[red,-] (0) edge (1);
        \draw[red,-] (0) edge (2);
        \draw[red,-] (1) edge (2);
        \draw[red,-] (3) edge (4);
        \draw[red,-] (3) edge (5);
        \draw[red,-] (4) edge (5);
    \end{tikzpicture}
}

\newcommand{\fbde}{
    \begin{tikzpicture}[scale=0.12,baseline=0.2mm]     \node[white] (W) at (4,1.5){$\triangle$};             \node[violet] (L) at (-4,1.5){$\diamondsuit$};              
	\node[inner sep=0cm] (0) at (0,0){$\cdot$};
        \node[inner sep=0cm] (1) at (-1.5,1){$\cdot$};
        \node[inner sep=0cm] (2) at (-1.5,2){$\cdot$};
        \node[inner sep=0cm] (3) at (1.5,1){$\cdot$};
        \node[inner sep=0cm] (4) at (1.5,2){$\cdot$};
        \node[inner sep=0cm] (5) at (0,3){$\cdot$};
        \draw[red,-] (0) edge (1);
        \draw[red,-] (0) edge (2);
        \draw[red,-] (0) edge (3);
        \draw[red,-] (0) edge (4);
        \draw[red,-] (0) edge (5);
        \draw[red,-] (1) edge (2);
    \end{tikzpicture}
}

\newcommand{\fbdf}{
    \begin{tikzpicture}[scale=0.12,baseline=0.2mm]     \node[violet] (L) at (-4,1.5){$\diamondsuit$}; 
    \node[white] (W) at (4,1.5){$\diamondsuit$};
	\node[inner sep=0cm] (0) at (0,0){$\cdot$};
        \node[inner sep=0cm] (1) at (-1.5,1){$\cdot$};
        \node[inner sep=0cm] (2) at (-1.5,2){$\cdot$};
        \node[inner sep=0cm] (3) at (1.5,1){$\cdot$};
        \node[inner sep=0cm] (4) at (1.5,2){$\cdot$};
        \node[inner sep=0cm] (5) at (0,3){$\cdot$};
        \draw[red,-] (0) edge (1);
        \draw[red,-] (0) edge (2);
        \draw[red,-] (0) edge (3);
        \draw[red,-] (0) edge (4);
        \draw[red,-] (0) edge (5);
        \draw[red,-] (1) edge (4);
        \draw[red,-] (3) edge (4);
    \end{tikzpicture}
}

\newcommand{\fbdg}{
    \begin{tikzpicture}[scale=0.12,baseline=0.2mm] 
    \node[white] (W) at (4,2.5){$\triangle$}; 
    \node[violet] (L) at (-4,1.5){$\diamondsuit$};              \node[magenta] (C) at (4,0){$\heartsuit$};
	\node[inner sep=0cm] (0) at (0,0){$\cdot$};
        \node[inner sep=0cm] (1) at (-1.5,1){$\cdot$};
        \node[inner sep=0cm] (2) at (-1.5,2){$\cdot$};
        \node[inner sep=0cm] (3) at (1.5,1){$\cdot$};
        \node[inner sep=0cm] (4) at (1.5,2){$\cdot$};
        \node[inner sep=0cm] (5) at (0,3){$\cdot$};
        \draw[red,-] (0) edge (1);
        \draw[red,-] (0) edge (2);
        \draw[red,-] (0) edge (3);
        \draw[red,-] (0) edge (4);
        \draw[red,-] (0) edge (5);
        \draw[red,-] (1) edge (2);
        \draw[red,-] (4) edge (5);
    \end{tikzpicture}
}

\newcommand{\fbdh}{
    \begin{tikzpicture}[scale=0.12,baseline=0.2mm]     \node[white] (S) at (4,1.5){$\triangle$};              \node[violet] (L) at (-4,1.5){$\diamondsuit$};              
	\node[inner sep=0cm] (0) at (0,0){$\cdot$};
        \node[inner sep=0cm] (1) at (-1.5,1){$\cdot$};
        \node[inner sep=0cm] (2) at (-1.5,2){$\cdot$};
        \node[inner sep=0cm] (3) at (1.5,1){$\cdot$};
        \node[inner sep=0cm] (4) at (1.5,2){$\cdot$};
        \node[inner sep=0cm] (5) at (0,3){$\cdot$};
        \draw[red,-] (0) edge (1);
        \draw[red,-] (0) edge (2);
        \draw[red,-] (0) edge (3);
        \draw[red,-] (0) edge (4);
        \draw[red,-] (0) edge (5);
        \draw[red,-] (1) edge (2);
        \draw[red,-] (3) edge (4);
    \end{tikzpicture}
}

\newcommand{\fbdi}{
    \begin{tikzpicture}[scale=0.12,baseline=0.2mm]      \node[violet] (L) at (-4,1.5){$\diamondsuit$}; 
    \node[white] (W) at (4,1.5){$\diamondsuit$};
	\node[inner sep=0cm] (0) at (0,0){$\cdot$};
        \node[inner sep=0cm] (1) at (-1.5,1){$\cdot$};
        \node[inner sep=0cm] (2) at (-1.5,2){$\cdot$};
        \node[inner sep=0cm] (3) at (1.5,1){$\cdot$};
        \node[inner sep=0cm] (4) at (1.5,2){$\cdot$};
        \node[inner sep=0cm] (5) at (0,3){$\cdot$};
        \draw[red,-] (0) edge (1);
        \draw[red,-] (0) edge (2);
        \draw[red,-] (0) edge (3);
        \draw[red,-] (0) edge (4);
        \draw[red,-] (0) edge (5);
        \draw[red,-] (1) edge (4);
        \draw[red,-] (3) edge (4);
        \draw[red,-] (3) edge (5);
    \end{tikzpicture}
}

\newcommand{\fbdj}{
    \begin{tikzpicture}[scale=0.12,baseline=0.2mm]     \node[violet] (L) at (-4,1.5){$\diamondsuit$}; 
    \node[white] (W) at (4,1.5){$\diamondsuit$};
	\node[inner sep=0cm] (0) at (0,0){$\cdot$};
        \node[inner sep=0cm] (1) at (-1.5,1){$\cdot$};
        \node[inner sep=0cm] (2) at (-1.5,2){$\cdot$};
        \node[inner sep=0cm] (3) at (1.5,1){$\cdot$};
        \node[inner sep=0cm] (4) at (1.5,2){$\cdot$};
        \node[inner sep=0cm] (5) at (0,3){$\cdot$};
        \draw[red,-] (0) edge (1);
        \draw[red,-] (0) edge (2);
        \draw[red,-] (0) edge (3);
        \draw[red,-] (0) edge (4);
        \draw[red,-] (0) edge (5);
        \draw[red,-] (1) edge (2);
        \draw[red,-] (3) edge (4);
        \draw[red,-] (3) edge (5);
    \end{tikzpicture}
}

\newcommand{\fbea}{
    \begin{tikzpicture}[scale=0.12,baseline=0.2mm]                  \node[violet] (L) at (-4,1.5){$\diamondsuit$};              \node[magenta] (C) at (4,0){$\heartsuit$};
	\node[inner sep=0cm] (0) at (0,0){$\cdot$};
        \node[inner sep=0cm] (1) at (-1.5,1){$\cdot$};
        \node[inner sep=0cm] (2) at (-1.5,2){$\cdot$};
        \node[inner sep=0cm] (3) at (1.5,1){$\cdot$};
        \node[inner sep=0cm] (4) at (1.5,2){$\cdot$};
        \node[inner sep=0cm] (5) at (0,3){$\cdot$};
        \draw[red,-] (0) edge (1);
        \draw[red,-] (0) edge (2);
        \draw[red,-] (0) edge (3);
        \draw[red,-] (0) edge (4);
        \draw[red,-] (0) edge (5);
        \draw[red,-] (1) edge (2);
        \draw[red,-] (3) edge (4);
        \draw[red,-] (3) edge (5);
        \draw[red,-] (4) edge (5);
    \end{tikzpicture}
}

\newcommand{\fbeb}{
    \begin{tikzpicture}[scale=0.12,baseline=0.2mm]     
	\node[inner sep=0cm] (0) at (0,0){$\cdot$};
        \node[inner sep=0cm] (1) at (-1.5,1){$\cdot$};
        \node[inner sep=0cm] (2) at (-1.5,2){$\cdot$};
        \node[inner sep=0cm] (3) at (1.5,1){$\cdot$};
        \node[inner sep=0cm] (4) at (1.5,2){$\cdot$};
        \node[inner sep=0cm] (5) at (0,3){$\cdot$};
        \draw[red,-] (0) edge (1);
        \draw[red,-] (0) edge (2);
        \draw[red,-] (0) edge (3);
        \draw[red,-] (0) edge (4);
        \draw[red,-] (1) edge (2);
        \draw[red,-] (1) edge (4);
    \end{tikzpicture}
}

\newcommand{\fbec}{
    \begin{tikzpicture}[scale=0.12,baseline=0.2mm]     \node[cyan] (S) at (4,2.5){$\triangle$}; 
    \node[white] (W) at (4,0){$\triangle$}; 
    \node[white] (L) at (-4,1.5){$\diamondsuit$}; 
    
	\node[inner sep=0cm] (0) at (0,0){$\cdot$};
        \node[inner sep=0cm] (1) at (-1.5,1){$\cdot$};
        \node[inner sep=0cm] (2) at (-1.5,2){$\cdot$};
        \node[inner sep=0cm] (3) at (1.5,1){$\cdot$};
        \node[inner sep=0cm] (4) at (1.5,2){$\cdot$};
        \node[inner sep=0cm] (5) at (0,3){$\cdot$};
        \draw[red,-] (0) edge (1);
        \draw[red,-] (0) edge (2);
        \draw[red,-] (0) edge (3);
        \draw[red,-] (0) edge (4);
        \draw[red,-] (1) edge (2);
        \draw[red,-] (1) edge (4);
        \draw[red,-] (3) edge (4);
    \end{tikzpicture}
}

\newcommand{\fbed}{
    \begin{tikzpicture}[scale=0.12,baseline=0.2mm]     \node[violet] (L) at (-4,1.5){$\diamondsuit$}; 
    \node[white] (W) at (4,1.5){$\diamondsuit$};
	\node[inner sep=0cm] (0) at (0,0){$\cdot$};
        \node[inner sep=0cm] (1) at (-1.5,1){$\cdot$};
        \node[inner sep=0cm] (2) at (-1.5,2){$\cdot$};
        \node[inner sep=0cm] (3) at (1.5,1){$\cdot$};
        \node[inner sep=0cm] (4) at (1.5,2){$\cdot$};
        \node[inner sep=0cm] (5) at (0,3){$\cdot$};
        \draw[red,-] (0) edge (1);
        \draw[red,-] (0) edge (2);
        \draw[red,-] (0) edge (3);
        \draw[red,-] (0) edge (4);
        \draw[red,-] (0) edge (5);
        \draw[red,-] (1) edge (2);
        \draw[red,-] (1) edge (4);
    \end{tikzpicture}
}

\newcommand{\fbee}{
    \begin{tikzpicture}[scale=0.12,baseline=0.2mm]  
    \node[white] (W) at (4,2.5){$\triangle$}; 
    \node[violet] (L) at (-4,1.5){$\diamondsuit$};              \node[magenta] (C) at (4,0){$\heartsuit$};
	\node[inner sep=0cm] (0) at (0,0){$\cdot$};
        \node[inner sep=0cm] (1) at (-1.5,1){$\cdot$};
        \node[inner sep=0cm] (2) at (-1.5,2){$\cdot$};
        \node[inner sep=0cm] (3) at (1.5,1){$\cdot$};
        \node[inner sep=0cm] (4) at (1.5,2){$\cdot$};
        \node[inner sep=0cm] (5) at (0,3){$\cdot$};
        \draw[red,-] (0) edge (1);
        \draw[red,-] (0) edge (2);
        \draw[red,-] (0) edge (3);
        \draw[red,-] (0) edge (4);
        \draw[red,-] (0) edge (5);
        \draw[red,-] (1) edge (4);
        \draw[red,-] (2) edge (5);
    \end{tikzpicture}
}

\newcommand{\fbef}{
    \begin{tikzpicture}[scale=0.12,baseline=0.2mm]     \node[violet] (L) at (-4,1.5){$\diamondsuit$}; 
    \node[white] (W) at (4,1.5){$\diamondsuit$};
	\node[inner sep=0cm] (0) at (0,0){$\cdot$};
        \node[inner sep=0cm] (1) at (-1.5,1){$\cdot$};
        \node[inner sep=0cm] (2) at (-1.5,2){$\cdot$};
        \node[inner sep=0cm] (3) at (1.5,1){$\cdot$};
        \node[inner sep=0cm] (4) at (1.5,2){$\cdot$};
        \node[inner sep=0cm] (5) at (0,3){$\cdot$};
        \draw[red,-] (0) edge (1);
        \draw[red,-] (0) edge (2);
        \draw[red,-] (0) edge (3);
        \draw[red,-] (0) edge (4);
        \draw[red,-] (0) edge (5);
        \draw[red,-] (1) edge (4);
        \draw[red,-] (2) edge (5);
        \draw[red,-] (3) edge (4);
    \end{tikzpicture}
}

\newcommand{\fbeg}{
    \begin{tikzpicture}[scale=0.12,baseline=0.2mm]     \node[violet] (L) at (-4,1.5){$\diamondsuit$}; 
    \node[white] (W) at (4,1.5){$\diamondsuit$};
	\node[inner sep=0cm] (0) at (0,0){$\cdot$};
        \node[inner sep=0cm] (1) at (-1.5,1){$\cdot$};
        \node[inner sep=0cm] (2) at (-1.5,2){$\cdot$};
        \node[inner sep=0cm] (3) at (1.5,1){$\cdot$};
        \node[inner sep=0cm] (4) at (1.5,2){$\cdot$};
        \node[inner sep=0cm] (5) at (0,3){$\cdot$};
        \draw[red,-] (0) edge (1);
        \draw[red,-] (0) edge (2);
        \draw[red,-] (0) edge (3);
        \draw[red,-] (0) edge (4);
        \draw[red,-] (0) edge (5);
        \draw[red,-] (1) edge (2);
        \draw[red,-] (1) edge (4);
        \draw[red,-] (3) edge (4);
    \end{tikzpicture}
}

\newcommand{\fbeh}{
    \begin{tikzpicture}[scale=0.12,baseline=0.2mm] 
    \node[white] (W) at (4,2.5){$\triangle$};  
    \node[violet] (L) at (-4,1.5){$\diamondsuit$};              \node[magenta] (C) at (4,0){$\heartsuit$};
	\node[inner sep=0cm] (0) at (0,0){$\cdot$};
        \node[inner sep=0cm] (1) at (-1.5,1){$\cdot$};
        \node[inner sep=0cm] (2) at (-1.5,2){$\cdot$};
        \node[inner sep=0cm] (3) at (1.5,1){$\cdot$};
        \node[inner sep=0cm] (4) at (1.5,2){$\cdot$};
        \node[inner sep=0cm] (5) at (0,3){$\cdot$};
        \draw[red,-] (0) edge (1);
        \draw[red,-] (0) edge (2);
        \draw[red,-] (0) edge (3);
        \draw[red,-] (0) edge (4);
        \draw[red,-] (0) edge (5);
        \draw[red,-] (1) edge (4);
        \draw[red,-] (2) edge (5);
        \draw[red,-] (3) edge (4);
        \draw[red,-] (3) edge (5);
    \end{tikzpicture}
}

\newcommand{\fbei}{
    \begin{tikzpicture}[scale=0.12,baseline=0.2mm]     \node[white] (S) at (4,1.5){$\triangle$};              \node[violet] (L) at (-4,1.5){$\diamondsuit$};              
	\node[inner sep=0cm] (0) at (0,0){$\cdot$};
        \node[inner sep=0cm] (1) at (-1.5,1){$\cdot$};
        \node[inner sep=0cm] (2) at (-1.5,2){$\cdot$};
        \node[inner sep=0cm] (3) at (1.5,1){$\cdot$};
        \node[inner sep=0cm] (4) at (1.5,2){$\cdot$};
        \node[inner sep=0cm] (5) at (0,3){$\cdot$};
        \draw[red,-] (0) edge (1);
        \draw[red,-] (0) edge (2);
        \draw[red,-] (0) edge (3);
        \draw[red,-] (0) edge (4);
        \draw[red,-] (0) edge (5);
        \draw[red,-] (1) edge (2);
        \draw[red,-] (1) edge (4);
        \draw[red,-] (3) edge (4);
        \draw[red,-] (3) edge (5);
    \end{tikzpicture}

}

\newcommand{\fbej}{
    \begin{tikzpicture}[scale=0.12,baseline=0.2mm]
    \node[white] (W) at (4,2.5){$\triangle$}; 
    \node[violet] (L) at (-4,1.5){$\diamondsuit$};              \node[magenta] (C) at (4,0){$\heartsuit$};
	\node[inner sep=0cm] (0) at (0,0){$\cdot$};
        \node[inner sep=0cm] (1) at (-1.5,1){$\cdot$};
        \node[inner sep=0cm] (2) at (-1.5,2){$\cdot$};
        \node[inner sep=0cm] (3) at (1.5,1){$\cdot$};
        \node[inner sep=0cm] (4) at (1.5,2){$\cdot$};
        \node[inner sep=0cm] (5) at (0,3){$\cdot$};
        \draw[red,-] (0) edge (1);
        \draw[red,-] (0) edge (2);
        \draw[red,-] (0) edge (3);
        \draw[red,-] (0) edge (4);
        \draw[red,-] (0) edge (5);
        \draw[red,-] (1) edge (2);
        \draw[red,-] (1) edge (4);
        \draw[red,-] (1) edge (5);
    \end{tikzpicture}
}

\newcommand{\fbfa}{
    \begin{tikzpicture}[scale=0.12,baseline=0.2mm]
    \node[white] (W) at (4,2.5){$\triangle$};  
    \node[violet] (L) at (-4,1.5){$\diamondsuit$};              \node[magenta] (C) at (4,0){$\heartsuit$};
	\node[inner sep=0cm] (0) at (0,0){$\cdot$};
        \node[inner sep=0cm] (1) at (-1.5,1){$\cdot$};
        \node[inner sep=0cm] (2) at (-1.5,2){$\cdot$};
        \node[inner sep=0cm] (3) at (1.5,1){$\cdot$};
        \node[inner sep=0cm] (4) at (1.5,2){$\cdot$};
        \node[inner sep=0cm] (5) at (0,3){$\cdot$};
        \draw[red,-] (0) edge (1);
        \draw[red,-] (0) edge (2);
        \draw[red,-] (0) edge (3);
        \draw[red,-] (0) edge (4);
        \draw[red,-] (0) edge (5);
        \draw[red,-] (1) edge (2);
        \draw[red,-] (1) edge (4);
        \draw[red,-] (1) edge (5);
        \draw[red,-] (4) edge (5);
    \end{tikzpicture}
}

\newcommand{\fbfb}{
    \begin{tikzpicture}[scale=0.12,baseline=0.2mm]     \node[white] (S) at (4,1.5){$\triangle$};              \node[violet] (L) at (-4,1.5){$\diamondsuit$};              
	\node[inner sep=0cm] (0) at (0,0){$\cdot$};
        \node[inner sep=0cm] (1) at (-1.5,1){$\cdot$};
        \node[inner sep=0cm] (2) at (-1.5,2){$\cdot$};
        \node[inner sep=0cm] (3) at (1.5,1){$\cdot$};
        \node[inner sep=0cm] (4) at (1.5,2){$\cdot$};
        \node[inner sep=0cm] (5) at (0,3){$\cdot$};
        \draw[red,-] (0) edge (1);
        \draw[red,-] (0) edge (2);
        \draw[red,-] (0) edge (3);
        \draw[red,-] (0) edge (4);
        \draw[red,-] (0) edge (5);
        \draw[red,-] (1) edge (2);
        \draw[red,-] (1) edge (4);
        \draw[red,-] (1) edge (5);
        \draw[red,-] (3) edge (4);
    \end{tikzpicture}
}

\newcommand{\fbfc}{
    \begin{tikzpicture}[scale=0.12,baseline=0.2mm]     
    \node[white] (S) at (4,1.5){$\triangle$};              
    \node[violet] (L) at (-4,1.5){$\diamondsuit$};  
    \node[magenta] (C) at (4,0){$\heartsuit$};
	\node[inner sep=0cm] (0) at (0,0){$\cdot$};
        \node[inner sep=0cm] (1) at (-1.5,1){$\cdot$};
        \node[inner sep=0cm] (2) at (-1.5,2){$\cdot$};
        \node[inner sep=0cm] (3) at (1.5,1){$\cdot$};
        \node[inner sep=0cm] (4) at (1.5,2){$\cdot$};
        \node[inner sep=0cm] (5) at (0,3){$\cdot$};
        \draw[red,-] (0) edge (1);
        \draw[red,-] (0) edge (2);
        \draw[red,-] (0) edge (3);
        \draw[red,-] (0) edge (4);
        \draw[red,-] (0) edge (5);
        \draw[red,-] (1) edge (2);
        \draw[red,-] (1) edge (4);
        \draw[red,-] (1) edge (5);
        \draw[red,-] (3) edge (4);
        \draw[red,-] (3) edge (5);
    \end{tikzpicture}
}

\newcommand{\fbfd}{
    \begin{tikzpicture}[scale=0.12,baseline=0.2mm]
    \node[white] (W) at (4,2.5){$\triangle$};              
    \node[violet] (L) at (-4,1.5){$\diamondsuit$};              \node[magenta] (C) at (4,0){$\heartsuit$};
	\node[inner sep=0cm] (0) at (0,0){$\cdot$};
        \node[inner sep=0cm] (1) at (-1.5,1){$\cdot$};
        \node[inner sep=0cm] (2) at (-1.5,2){$\cdot$};
        \node[inner sep=0cm] (3) at (1.5,1){$\cdot$};
        \node[inner sep=0cm] (4) at (1.5,2){$\cdot$};
        \node[inner sep=0cm] (5) at (0,3){$\cdot$};
        \draw[red,-] (0) edge (1);
        \draw[red,-] (0) edge (2);
        \draw[red,-] (0) edge (3);
        \draw[red,-] (0) edge (4);
        \draw[red,-] (0) edge (5);
        \draw[red,-] (1) edge (2);
        \draw[red,-] (1) edge (4);
        \draw[red,-] (1) edge (5);
        \draw[red,-] (3) edge (4);
        \draw[red,-] (3) edge (5);
        \draw[red,-] (4) edge (5);
    \end{tikzpicture}
}

\newcommand{\fbfe}{
    \begin{tikzpicture}[scale=0.12,baseline=0.2mm] 
    \node[white] (W) at (4,2.5){$\triangle$};  
    \node[violet] (L) at (-4,1.5){$\diamondsuit$};              \node[magenta] (C) at (4,0){$\heartsuit$};
	\node[inner sep=0cm] (0) at (0,0){$\cdot$};
        \node[inner sep=0cm] (1) at (-1.5,1){$\cdot$};
        \node[inner sep=0cm] (2) at (-1.5,2){$\cdot$};
        \node[inner sep=0cm] (3) at (1.5,1){$\cdot$};
        \node[inner sep=0cm] (4) at (1.5,2){$\cdot$};
        \node[inner sep=0cm] (5) at (0,3){$\cdot$};
        \draw[red,-] (0) edge (1);
        \draw[red,-] (0) edge (2);
        \draw[red,-] (0) edge (3);
        \draw[red,-] (0) edge (4);
        \draw[red,-] (0) edge (5);
        \draw[red,-] (1) edge (2);
        \draw[red,-] (1) edge (4);
        \draw[red,-] (1) edge (5);
        \draw[red,-] (2) edge (5);
    \end{tikzpicture}
}

\newcommand{\fbff}{
    \begin{tikzpicture}[scale=0.12,baseline=0.2mm]     \node[white] (S) at (4,1.5){$\triangle$};              \node[violet] (L) at (-4,1.5){$\diamondsuit$};              
	\node[inner sep=0cm] (0) at (0,0){$\cdot$};
        \node[inner sep=0cm] (1) at (-1.5,1){$\cdot$};
        \node[inner sep=0cm] (2) at (-1.5,2){$\cdot$};
        \node[inner sep=0cm] (3) at (1.5,1){$\cdot$};
        \node[inner sep=0cm] (4) at (1.5,2){$\cdot$};
        \node[inner sep=0cm] (5) at (0,3){$\cdot$};
        \draw[red,-] (0) edge (1);
        \draw[red,-] (0) edge (2);
        \draw[red,-] (0) edge (3);
        \draw[red,-] (0) edge (4);
        \draw[red,-] (0) edge (5);
        \draw[red,-] (1) edge (2);
        \draw[red,-] (1) edge (4);
        \draw[red,-] (1) edge (5);
        \draw[red,-] (2) edge (5);
        \draw[red,-] (3) edge (4);
    \end{tikzpicture}
}

\newcommand{\fbfg}{
    \begin{tikzpicture}[scale=0.12,baseline=0.2mm] 
    \node[white] (W) at (4,2.5){$\triangle$};  
    \node[violet] (L) at (-4,1.5){$\diamondsuit$};              \node[magenta] (C) at (4,0){$\heartsuit$};
	\node[inner sep=0cm] (0) at (0,0){$\cdot$};
        \node[inner sep=0cm] (1) at (-1.5,1){$\cdot$};
        \node[inner sep=0cm] (2) at (-1.5,2){$\cdot$};
        \node[inner sep=0cm] (3) at (1.5,1){$\cdot$};
        \node[inner sep=0cm] (4) at (1.5,2){$\cdot$};
        \node[inner sep=0cm] (5) at (0,3){$\cdot$};
        \draw[red,-] (0) edge (1);
        \draw[red,-] (0) edge (2);
        \draw[red,-] (0) edge (3);
        \draw[red,-] (0) edge (4);
        \draw[red,-] (0) edge (5);
        \draw[red,-] (1) edge (2);
        \draw[red,-] (1) edge (4);
        \draw[red,-] (1) edge (5);
        \draw[red,-] (2) edge (5);
        \draw[red,-] (4) edge (5);
    \end{tikzpicture}
}

\newcommand{\fbfh}{
    \begin{tikzpicture}[scale=0.12,baseline=0.2mm]                   \node[violet] (L) at (-4,1.5){$\diamondsuit$};              \node[magenta] (C) at (4,0){$\heartsuit$};
	\node[inner sep=0cm] (0) at (0,0){$\cdot$};
        \node[inner sep=0cm] (1) at (-1.5,1){$\cdot$};
        \node[inner sep=0cm] (2) at (-1.5,2){$\cdot$};
        \node[inner sep=0cm] (3) at (1.5,1){$\cdot$};
        \node[inner sep=0cm] (4) at (1.5,2){$\cdot$};
        \node[inner sep=0cm] (5) at (0,3){$\cdot$};
        \draw[red,-] (0) edge (1);
        \draw[red,-] (0) edge (2);
        \draw[red,-] (0) edge (3);
        \draw[red,-] (0) edge (4);
        \draw[red,-] (0) edge (5);
        \draw[red,-] (1) edge (2);
        \draw[red,-] (1) edge (4);
        \draw[red,-] (1) edge (5);
        \draw[red,-] (2) edge (5);
        \draw[red,-] (3) edge (4);
        \draw[red,-] (3) edge (5);
    \end{tikzpicture}
}

\newcommand{\fbfi}{
    \begin{tikzpicture}[scale=0.12,baseline=0.2mm]     \node[cyan] (S) at (4,2.5){$\triangle$};              \node[violet] (L) at (-4,1.5){$\diamondsuit$};              \node[magenta] (C) at (4,0){$\heartsuit$};
	\node[inner sep=0cm] (0) at (0,0){$\cdot$};
        \node[inner sep=0cm] (1) at (-1.5,1){$\cdot$};
        \node[inner sep=0cm] (2) at (-1.5,2){$\cdot$};
        \node[inner sep=0cm] (3) at (1.5,1){$\cdot$};
        \node[inner sep=0cm] (4) at (1.5,2){$\cdot$};
        \node[inner sep=0cm] (5) at (0,3){$\cdot$};
        \draw[red,-] (0) edge (1);
        \draw[red,-] (0) edge (2);
        \draw[red,-] (0) edge (3);
        \draw[red,-] (0) edge (4);
        \draw[red,-] (0) edge (5);
        \draw[red,-] (1) edge (2);
        \draw[red,-] (1) edge (4);
        \draw[red,-] (1) edge (5);
        \draw[red,-] (2) edge (5);
        \draw[red,-] (3) edge (4);
        \draw[red,-] (3) edge (5);
        \draw[red,-] (4) edge (5);
    \end{tikzpicture}
}

\begin{figure}[H]
	\caption{The Hasse diagram of transfer systems for the group $F_5$.}
	\label{transfersystemF5}	
	\[
	\begin{array}{c}
	\begin{tikzpicture}[scale=1.1]
		\node(00) at (0,0) {$\fba$};
		\node(02) at (3,1.5) {$\fbc$};
            \node(01) at (0,1.5) {$\fbb$};
		\node(03) at (-3,1.5) {$\fbd$};
            \node(04) at (3,3) {$\fbe$};
            \node(06) at (1.5,3) {$\fbg$};
		\node(07) at (-1.5,3) {$\fbh$};
            \node(11) at (0,3) {$\fbbb$};
		\node(05) at (-3,3) {$\fbf$};
            \node(09) at (3,4.5) {$\fbj$};
		\node(08) at (-3,4.5) {$\fbi$};
            \node(12) at (0,4.5) {$\fbbc$};
            \node(13) at (-1.5,4.5) {$\fbbd$};
            \node(22) at (1.25,4.5) {$\fbcc$};
            \node(24) at (4.5,6) {$\fbce$};
            \node(10) at (-4.5,6) {$\fbba$};
		\node(16) at (-3,6) {$\fbbg$};
            \node(14) at (-1.5,6) {$\fbbe$};
            \node(23) at (3,6) {$\fbcd$};
            \node(15) at (0,6) {$\fbbf$};
            \node(25) at (1.5,6) {$\fbcf$};
            \node(28) at (4.5,7.5) {$\fbci$};
            \node(19) at (-4.5,7.5) {$\fbbj$};
            \node(17) at (-3,7.5) {$\fbbh$};
		\node(29) at (-1.5,7.5) {$\fbcj$};
            \node(18) at (0,7.5) {$\fbbi$};
            \node(27) at (1.5,7.5) {$\fbch$};
            \node(26) at (3,7.5) {$\fbcg$};
            \node(32) at (-4.5,9) {$\fbdc$};
            \node(20) at (-3,9) {$\fbca$};
            \node(30) at (-1.5,9) {$\fbda$};
		\node(41) at (1.5,9) {$\fbeb$};
            \node(34) at (3,9) {$\fbde$};
            \node(31) at (0,9) {$\fbdb$};
            \node(33) at (4,9) {$\fbdd$};
            \node(42) at (0,10.5) {$\fbec$};
            \node(35) at (-1.5,10.5) {$\fbdf$};
		\node(21) at (-3,10.5) {$\fbcb$};
            \node(44) at (-4.5,10.5) {$\fbee$};
            \node(43) at (1.5,10.5) {$\fbed$};
            \node(37) at (3,10.5) {$\fbdh$};
            \node(36) at (4.5,10.5) {$\fbdg$};
            \node(46) at (0,12) {$\fbeg$};
		\node(38) at (-3,12) {$\fbdi$};
            \node(45) at (-1.5,12) {$\fbef$};
            \node(49) at (1.5,12) {$\fbej$};
            \node(39) at (3,12) {$\fbdj$};
            \node(51) at (-1.5,13.5) {$\fbfb$};
		\node(47) at (-3,13.5) {$\fbeh$};
            \node(48) at (0,13.5) {$\fbei$};
            \node(54) at (1.5,13.5) {$\fbfe$};
            \node(50) at (3,13.5) {$\fbfa$};
            \node(40) at (4.5,13.5) {$\fbea$};
		\node(55) at (-1.5,15) {$\fbff$};
            \node(52) at (0,15) {$\fbfc$};
            \node(56) at (1.5,15) {$\fbfg$};
		\node(57) at (-1.5,16.5) {$\fbfh$};
            \node(53) at (0,16.5) {$\fbfd$};
		\node(58) at (0,18) {$\fbfi$};
	
        \node[cyan] (S) at (-6.15,18){$\triangle$ = saturated}; 
        \node[magenta] (C) at (-6.05,17.5){$\heartsuit$ = cosaturated}; 
        \node[violet] (L) at (-5,17){$\vardiamond$ = lesser simply paired (LSP)};      
        \node[violet] (L) at (-5.3,16.5){$\diamondsuit$ = connected ($\implies$ LSP)};         

        \node[inner sep=0cm] (70) at (-5.5,13){$e$};
        \node[inner sep=0cm] (71) at (-4,14){$C_5$};
        \node[inner sep=0cm] (72) at (-7,14){${}_5 C_2$};
        \node[inner sep=0cm] (73) at (-4,15){$D_5$};
        \node[inner sep=0cm] (74) at (-7,15){${}_5 C_{4}$};
        \node[inner sep=0cm] (75) at (-5.5,16){$F_5$};
        \draw[black,-] (70) edge (71);
        \draw[black,-] (70) edge (72);
        \draw[black,-] (70) edge (73);
        \draw[black,-] (70) edge (74);
        \draw[black,-] (70) edge (75);
        \draw[black,-] (71) edge (73);
        \draw[black,-] (71) edge (75);
        \draw[black,-] (72) edge (73);
        \draw[black,-] (72) edge (74);
        \draw[black,-] (72) edge (75);
        \draw[black,-] (73) edge (75);
        \draw[black,-] (74) edge (75);
        
		\path[-]
		(00) edge node {} (02)
		(00) edge node {} (01)
		(00) edge node {} (03)
		(02) edge node {} (04)
            (02) edge node {} (06)
		      (02) edge node {} (22)
		(01) edge node {} (06)
		(01) edge node {} (07)
		(03) edge node {} (07)
            (03) edge node {} (11)
		(03) edge node {} (05)
		(04) edge node {} (09)
        (04) edge node {} (24)
		(06) edge node {} (09)
		(06) edge node {} (23)
		(07) edge node {} (08)
		(07) edge node {} (12)
		(11) edge node {} (24)
		(11) edge node {} (12)
            (11) edge node {} (13)
		(11) edge node {} (22)
		(05) edge node {} (10)
		(05) edge node {} (13)
		(09) edge node {} (28)
		(08) edge node {} (10)
		(08) edge node {} (16)
		(08) edge node {} (14)
		(12) edge node {} (14)
		(12) edge node {} (23)
		(13) edge node {} (17)
		(13) edge node {} (15)
		(13) edge node {} (25)
            (22) edge node {} (23)
            (22) edge node {} (25)
            (24) edge node {} (28)
            (24) edge node {} (33)
            (10) edge node {} (19)
            (10) edge node {} (17)
            (16) edge node {} (19)
            (16) edge node {} (29)
            (14) edge node {} (17)
            (14) edge node {} (29)
            (14) edge node {} (18)
            (14) edge node {} (27)
            (23) edge node {} (28)
            (23) edge node {} (27)
            (15) edge node {} (26)
            (25) edge node {} (31)
            (25) edge node {} (26)
            (28) edge node {} (36)
            (19) edge node {} (32)
            (17) edge node {} (32)
            (17) edge node {} (20)
            (17) edge node {} (31)
            (29) edge node {} (30)
            (29) edge node {} (41)
            (18) edge node {} (20)
            (18) edge node {} (30)
            (18) edge node {} (34)
            (27) edge node {} (41)
            (27) edge node {} (34)
            (27) edge node {} (31)
            (26) edge node {} (33)
            (32) edge node {} (42)
            (32) edge node {} (35)
            (20) edge node {} (35)
            (20) edge node {} (21)
            (30) edge node {} (35)
            (30) edge node {} (44)
            (30) edge node {} (43)
            (41) edge node {} (43)
            (34) edge node {} (43)
            (34) edge node {} (37)
            (34) edge node {} (36)
            (31) edge node {} (42)
            (31) edge node {} (37)
            (33) edge node {} (39)
            (42) edge node {} (46)
            (35) edge node {} (46)
            (35) edge node {} (38)
            (35) edge node {} (45)
            (21) edge node {} (38)
            (21) edge node {} (39)
            (44) edge node {} (45)
            (43) edge node {} (46)
            (43) edge node {} (49)
            (37) edge node {} (46)
            (37) edge node {} (39)
            (36) edge node {} (40)
            (46) edge node {} (51)
            (46) edge node {} (48)
            (38) edge node {} (47)
            (38) edge node {} (48)
            (45) edge node {} (47)
            (49) edge node {} (52)
            (49) edge node {} (54)
            (49) edge node {} (50)
            (39) edge node {} (48)
            (39) edge node {} (40)
            (51) edge node {} (55)
            (51) edge node {} (52)
            (47) edge node {} (57)
            (48) edge node {} (52)
            (54) edge node {} (56)
            (50) edge node {} (56)
            (40) edge node {} (56)
            (55) edge node {} (57)
            (52) edge node {} (57)
            (52) edge node {} (53)
            (56) edge node {} (58)
            (57) edge node {} (58)
            (53) edge node {} (58)
		;
	\end{tikzpicture}
	\end{array}
	\]
\end{figure}
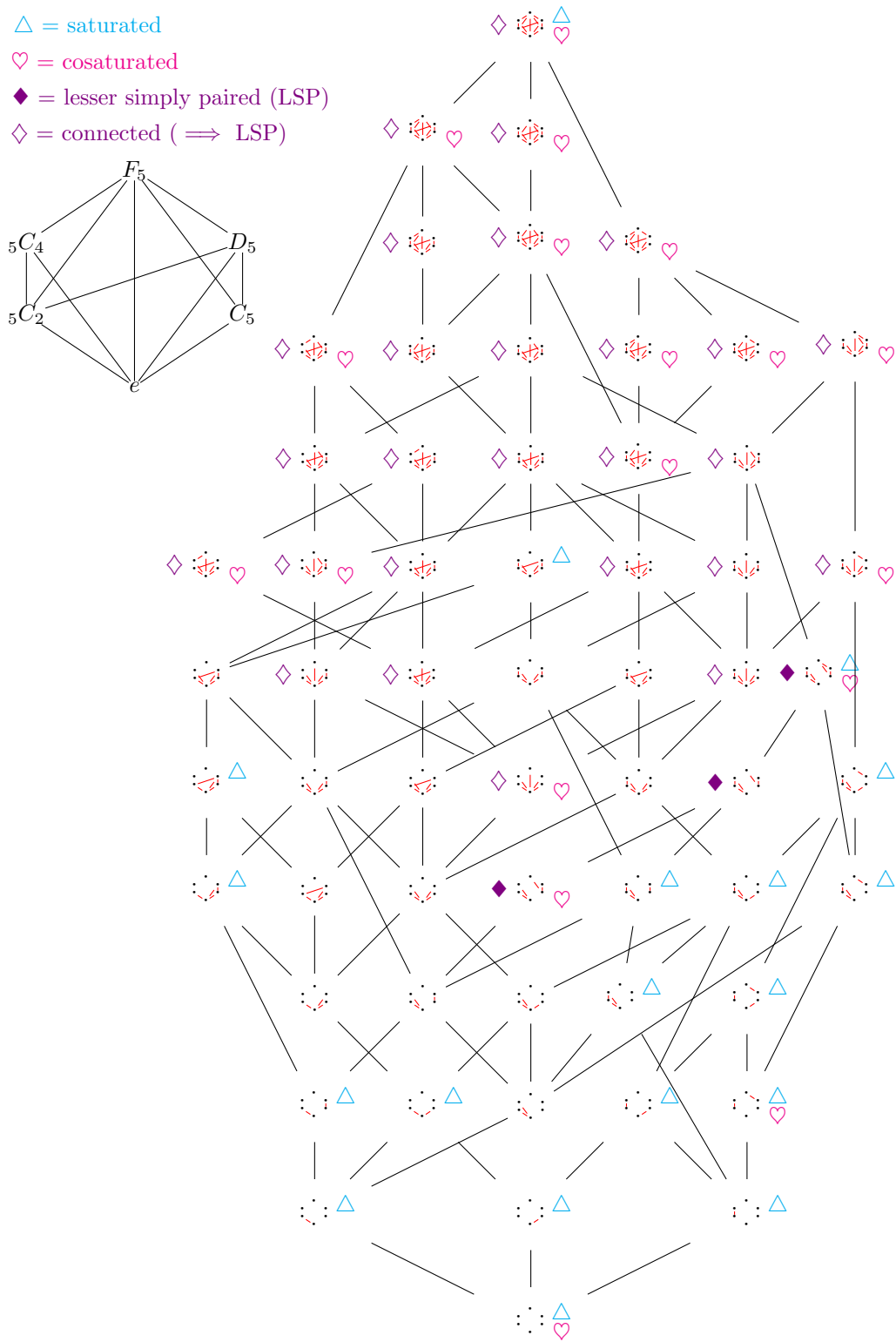

\newpage

\bibliographystyle{amsalpha}
\bibliography{bibby.bib}

\end{document}